\documentclass[a4paper]{amsart}

\usepackage{amsmath,amsthm,amssymb,enumerate}
\usepackage{epsfig}
\usepackage{amssymb}
\usepackage{amsmath}
\usepackage{amssymb}
\usepackage{amsmath,amsthm}
\usepackage[latin1]{inputenc}
\usepackage[T1]{fontenc}
\usepackage{path}
\usepackage{ae,aecompl}
\usepackage{amsfonts}
\usepackage{amsxtra}
\usepackage{euscript,mathrsfs}
\usepackage{color}
\usepackage[left=3.4cm,right=3.4cm,top=4cm,bottom=4cm]{geometry}
\usepackage[citecolor=blue,colorlinks=true]{hyperref}
\allowdisplaybreaks

\usepackage{enumitem}
\setenumerate{label={\rm (\alph{*})}}

\usepackage{amsfonts}
\usepackage{amsxtra}

\numberwithin{equation}{section}

\newtheorem{theorem}{Theorem}[section]

\newtheorem{lemma}{Lemma}[section]

\newtheorem{corollary}{Corollary}[section]

\newtheorem{remark}{Remark}[section]

\newtheorem{definition}{Definition}[section]

\newcommand{\bfu}{u}
\newcommand{\bfh}{\mathfrak{h}}
\newcommand{\bfk}{\mathfrak{k}}
\newcommand{\bfl}{\mathfrak{l}}
\newcommand{\bfU}{U}
\newcommand{\bfv}{v}

\newcommand{\bfz}{z}
\newcommand{\bfg}{\mathfrak{g}}
\newcommand{\bfxi}{\xi}
\newcommand{\bfphi}{\phi}
\newcommand{\bfvarphi}{\varphi}
\newcommand{\bfpsi}{\psi}
\newcommand{\bfeta}{\eta}
\newcommand{\bfzeta}{\zeta}

\newcommand{\UU}{\mathscr U}

\newcommand{\R}{\mathbb R}
\newcommand{\N}{\mathbb N}

\newcommand{\E}{\mathbb E}
\newcommand{\p}{\mathbb P}

\newcommand{\F}{\mathfrak F}

\DeclareMathOperator{\tr}{tr}
\DeclareMathOperator{\id}{Id}

\newcommand{\SSS}{\mathcal S}
\newcommand{\A}{\mathcal A}

\newcommand{\dd}{\mathrm d}
\newcommand{\dx}{\, \mathrm{d}x}

\newcommand{\ds}{\, \mathrm{d}s}
\newcommand{\dt}{\, \mathrm{d}t}

\newcommand{\dxs}{\,\mathrm{d}x\, \mathrm{d}s}
\newcommand{\dif}{\mathrm{d}}

\newcommand{\mf}{\mathfrak{F}}

\newcommand{\prst}{\mathbb{P}}

\newcommand{\mt}{\mathbb{T}^d}
\newcommand{\tor}{\mathbb{T}^d}

\begin{document}

\title[Weak error stochastic Allen-Cahn]{Weak error analysis for the stochastic Allen-Cahn equation}

\author{Dominic Breit}
\address{Department of Mathematics, Heriot-Watt University, Riccarton Edinburgh EH14 4AS, UK}
\email{d.breit@hw.ac.uk}

\address{Institute of Mathematics, TU Clausthal, Erzstra\ss e 1, 38678 Clausthal-Zellerfeld, Germany}
\email{dominic.breit@tu-clausthal.de}

\author{Andreas Prohl}
\address{}
\email{prohl@na.uni-tuebingen.de}

%
%

\begin{abstract}
We prove strong rate {\em resp.}~weak rate ${\mathcal O}(\tau)$ for a structure preserving temporal discretization (with $\tau$ the step size) of the stochastic Allen-Cahn equation with additive {\em resp.}~multiplicative colored noise in $d=1,2,3$ dimensions. 
Direct variational arguments exploit the one-sided Lipschitz property of the cubic nonlinearity in the first setting to settle first order strong rate. It is the same property which allows for uniform bounds for the derivatives of the solution of the related Kolmogorov equation, and then leads to weak rate ${\mathcal O}(\tau)$ in the presence of multiplicative noise.
Hence, we obtain twice the rate of convergence known for the strong error in the presence of multiplicative noise.

\end{abstract}

\keywords{Stochastic Allen-Cahn equation \and weak error analysis \and time discretisation  \and convergence rates}
\subjclass[2010]{65M15, 65C30, 60H15, 60H35}

\date{\today}

\maketitle

%
%
%
%
%
%
%
%
%
%

\section{Introduction}

We use time discretisation to approximate the stochastic Allen--Cahn equation
\begin{align}\label{eq:SNS}
\left\{\begin{array}{rc}
\dd\bfu= \bigl(\Delta\bfu - (\bfu^3-\bfu)\bigr)\dt+\Phi(\bfu)\dd W
& \mbox{in $\mathcal Q_T$,}\\
\bfu(0)=\bfu^0\,\qquad\qquad\qquad\qquad\qquad\qquad&\mbox{ \,in $\mt$,}\end{array}\right.
\end{align}
in $\mathcal Q_T:=(0,T)\times\mt$, where $\mathbb T^d = \left( (-\pi, \pi)|_{\{ - \pi, \pi \}} \right)^d$ (supplemented with periodic boundary conditions) with $T>0$ and $d=1,2,3$. The use of periodic boundary conditions is for an easy of presentation only, see Remark \ref{remark:noise'}, {\bf 4}.
The unknown  $\bfu$ in \eqref{eq:SNS} is
defined on a given filtered probability space $(\Omega,\F,(\F_t)_{t\geq0},\prst)$, and $\bfu_0$ is a given initial datum.
Here $W$ denotes a cylindrical Wiener process and $\Phi$ takes values in the space of Hilbert-Schmidt operators; see Section \ref{sec:prob} for details.

The deterministic version of (\ref{eq:SNS}) is the well-known  {\em Allen-Cahn equation} --- a phase-field model to approximate the dynamics of an (material) interface by a {\em diffuse interface}; see {\em e.g.} \cite{DuFe} for a recent review of (deterministic) phase-field models. 
The related mathematical, underlying modeling, and physical conclusions usually base on the underlying Helmholtz free energy functional ${\mathcal E} : W^{1,2}(\mt)  \rightarrow {\mathbb R}$, where
\begin{equation}\label{ac-1}
{\mathcal E}(\phi) = \frac{1}{2}\int_{{\mathbb T}^d} \vert \nabla \phi\vert^2\, {\rm d}x +  \int_{{\mathbb T}^d} F(\phi)\, {\rm d}x\,, \qquad \mbox{for} \qquad F(x) = \frac{1}{4} \bigl( \vert x\vert^2-1\bigr)^2\,.
\end{equation}
Here, in particular, the latter energy part accounts for the interfacial/mixing energy, and is related to $f(x) = x^3-x$ in (\ref{eq:SNS}) by $F' = f$. As a consequence, the Allen-Cahn equation is the {\em gradient flow} of (\ref{ac-1}), {\em i.e.},
\begin{equation}\label{ac-2}
\partial_t u = -D {\mathcal E}(u) =  \Delta u -  f(u) \quad \mbox{in}\quad L^2(\mt)\,,
\end{equation}
 where $D {\mathcal E}(u)$ denotes the Fr\'echet derivative of ${\mathcal E}$ at $u$: multiplication of (\ref{ac-2}) with $D {\mathcal E}(u)$, integration in space, and the chain rule then lead to the energy {\em identity}
 \begin{equation}\label{ac-3}
 {\mathcal E}\bigl( u(t, \cdot)\bigr) + 
  \int_0^t \int_{{\mathbb T}^d} \bigl\vert D{\mathcal E}\bigl( u(s, \cdot)\bigr)\bigr\vert^2\, {\rm d}x{\rm d}s = {\mathcal E}\bigl( u^0\bigr), \qquad 0 \leq t \leq T\, .
 \end{equation}
 For ${\mathcal E}(u_0) < \infty$, this identity may serve 
in mathematical analysis to deduce (a priori) bounds for 
solutions in physically relevant norms ---
reflecting the fact that the Helmholtz energy ${\mathcal E}$ is the proper functional to explain the dynamics of (\ref{ac-2}).

This {\em energy-driven} approach also serves as guidance to construct 
the numerical scheme (\ref{ac-2d}) below to properly address the specific nature of (\ref{ac-2}). But before we also mention here `general-purpose' operator-splitting methods, where the nonlinearity in $D{\mathcal E}(u)$ is treated explicitly in a time marching context to avoid the use of nonlinear numerical solvers: it is, however, the concomitant violation of the dissipative energy law (\ref{ac-3}) on a discrete level, see also (\ref{ac-3d}) below, that a related
stability/convergence analysis for splitting schemes usually 
only gives a priori bounds {\em in non-physical norms}, and its derivation is based on a discrete {\em Gronwall type estimate} --- rather than {\em identity} (\ref{ac-3}) ---, which heavily affects discrete long time stability. This drawback, in particular, may usually only be compensated in simulations by using
(comparatively) much smaller step sizes than in the context of the structure-preserving discretization below; we also mention here that the usual application of model (\ref{ac-2}) in sciences ({\em e.g.}, multiphase flow in complicated, or even moving domains, \cite[Ch.~2]{DuFe}) or geometric PDEs
({\em e.g.}, approximation of mean curvature flows, \cite{DeDzEl}) involves a small scaling factor $\varepsilon >0$ under which realistic (diffuse) interface motion holds --- which even further worsens inherent discrete instabilities of structure notably; see also \cite[Ch.'s~4 \& 5]{DuFe} for a related further discussion.

For these reasons we favor as starting point for a temporal discretisation for (\ref{eq:SNS}) one which inherits  the (gradient flow) structure of the original problem --- as is the following implicit Euler discretisation for (\ref{ac-2}) governing
 iterates $( u_m)_{m=0}^M$,
 \begin{equation}\label{ac-2d}
\frac{1}{\tau}\bigl( u_{m} - u_{m-1}\bigr) -\Delta u_m +  f \bigl( u_m\bigr) = 0 \quad \text{for}\quad m \geq 1\,, \quad u_0 = u^0\,,
\end{equation}
on an equi-distant mesh $(t_m)_{m=0}^M \subset [0,T]$ of size $0 < \tau < 1$, and satisfying the  {\em discrete energy inequality} (see \cite{FePr})
\begin{equation}\label{ac-3d}
 {\mathcal E}\bigl( u_m\bigr) +
  \tau\sum_{\ell=1}^m\int_{{\mathbb T}^d} \bigl\vert D{\mathcal E}\bigl( u_\ell\bigr)\bigr\vert^2\, {\rm d}x  \leq  {\mathcal E}\bigl( u_0\bigr), \quad 1 \leq m \leq \bigl[\tfrac{T}{\tau}\bigr]\, .
 \end{equation}

In this work, we derive weak rates of convergence for the time iteration scheme
 \begin{equation}\label{tdiscrAintro}
 u^{m} - u^{m-1} -\tau\Delta u^m +  \tau f \bigl( u^m\bigr) = \Phi(u^{m-1})\Delta_m W \quad \text{for}\quad m \geq 1\,, \quad u^0 = u_0\,,
\end{equation}
with $\Delta_mW=W(t_{m})-W(t_{m-1})$,
to address the SPDE (\ref{eq:SNS}) of gradient type, generalising (\ref{ac-2d}). Again, its construction is motivated by the demand to inherit properties in terms of the underlying energy ${\mathcal E}$, for which the concept of a {\em variational strong solution} for (\ref{eq:SNS}) is natural --- see Definition \ref{def:inc2d} below --- which satisfies the following energy identity
\begin{eqnarray}\label{ac-3s}
&& {\mathbb E}\bigl[{\mathcal E}\bigl( u(t, \cdot)\bigr)\bigr] + 
  {\mathbb E}\Bigl[\int_0^t\int_{{\mathbb T}^d} \bigl\vert D{\mathcal E}\bigl( u(s, \cdot)\bigr)\bigr\vert^2\, {\rm d}x{\rm d}s\Bigr] \\ \nonumber
  &&\qquad = {\mathbb E}\bigl[{\mathcal E}\bigl( u_0\bigr)\bigr] + \frac{1}{2} {\mathbb E}\Bigl[\int_0^t \int_{{\mathbb T}^d}\bigl( D^2 {\mathcal E}(u) \Phi(u),\Phi(u)\bigr) \, {\rm d}x{\rm d}s \Bigr],\quad 0 \leq t \leq T\, ,
 \end{eqnarray}
and therefore generalises property (\ref{ac-3}) for (\ref{ac-2});
 see {\em e.g.} \cite{Ge}, where so-called strong variational solutions to stochastic equations of gradient-type are considered.

A discretisation close to (\ref{tdiscrAintro}) for (\ref{eq:SNS}) has been studied in \cite{MP1} in this spirit, and it was shown that the {\em strong} error --- {\em i.e.}, the expectation of the discrete $L^\infty(0,T;L^2_x)\cap L^2(0,T;W^{1,2}_x)$ distance of discrete and continuous solution --- is of order $O(\sqrt{\tau})$.
 This is certainly optimal in general due to the low temporal regularity of the driving Wiener process in \eqref{eq:SNS}. 
%
Our goal in this work is to verify first order weak error estimates
for (\ref{tdiscrAintro}); it turns out, that its derivation crucially depends on the kind of
noise, which we assume to be spatially smooth throughout:
\begin{enumerate}
\item[{\bf a)}] If $\Phi(u) \equiv \Phi$ in (\ref{eq:SNS}) generates {\em additive} noise,
we even verify {\em strong} convergence rate ${\mathcal O}(\tau)$ for scheme
(\ref{tdiscrAintro}) in Theorem \ref{thm:3.1tilde} for $d=1,2$ and 3. Its derivation is based on 
focusing on the {\em random} PDE (\ref{eq:SNS1}) for the transformation $y(t) = u(t)- \Phi W(t)$, which is now
differentiable in time (see Corollary \ref{cor:add}), and a simple use of the
binomial formula to express $f\bigl( u(t)\bigr) = f\bigl( y(t) + \Phi W(t)\bigr)$ right below (\ref{eq:SNS1}). {\em No} discrete stability for iterates of (\ref{tdiscrAintro}) is needed here, but the implicit use of the one-sided Lipschitz nonlinearity $f$, as well as its
explicit form are exploited to verify the strong error bounds in Section \ref{sec:add}.
This approach is motivated from \cite{BP1}.
\item[{\bf b)}] If $\Phi(u)$ in (\ref{eq:SNS}) exerts {\em multiplicative} noise, we use the Kolmogorov equation (\ref{eq:Kolm'}) associated with \eqref{eq:SNS}, in combination with higher order discrete stability for
iterates $(u_m)_{m=0}^M$
from scheme
(\ref{tdiscrAintro}); cf. Lemma \ref{lemma:3.1}. For estimate \eqref{lem:3.1c} and $d=3$ concerning the second derivative of the time-discrete solution we are forced to work under the assumption of an affine linear noise, see \ref{N2} below, whereas for $d=2$ the same argument even works for a general nonlinear noise, see \ref{N1} below. 
For the weak error analysis, we conceptually borrow tools from \cite{BrGo,De,BrDe}; see Remark \ref{remark:noise'}, {\bf 3.} in particular, the structural restrictions are detailed. As in \cite{DePr,De,BrDe}, the error $\E[\varphi(u_m)-\varphi(u(t_m))]$ for a smooth function $\varphi$ can be linked  to the solution of the Kolmogorov equation by means of an application of It\^{o}'s formula. Therefore we need time-continuous interpolation of the discrete iterates $(u_m)_{m=0}^M$, which yields an $(\mathfrak F_t)$-adapted process.
Since we work with the fully implicit scheme (\ref{tdiscrAintro}) this is more complicated than in previous works \cite{DePr,De,BrDe}.
A natural candidate for the interpolation is 
\begin{align*}\bfu_\tau(t)=\frac{t_{m}-t}{\tau}\bfu_{m-1}+\mathcal T_\tau\bigg(\frac{1}{\tau}\int_{t_{m-1}}^t\bfu_{m-1}\ds+\int_{t_{m-1}}^t\Phi(\bfu_{m-1})\,\dd W\bigg),\qquad t_{m-1}\leq t\leq t_m,\end{align*}
where the solution map $\mathcal T_\tau$ for (\ref{eq:T}) is the discrete nonlinear semi-group corresponding to $D\mathcal E$, which we analyse in Section \ref{sec:Ttau}. We observe that the nonlinearity does not appear explicitely in the formula for $\bfu_\tau$. In previous
works the linear discrete semigroup $\mathcal S_\tau=(\mathrm{Id}-\tau\Delta)^{-1}$ is used instead and $f$ is explicitly evaluated.
It turns out that the nonlinear map $\mathcal T_\tau$ has nice properties similar to the linear case as a consequence of the one-sided Lipschitz property of $f$; see Lemma \ref{lem:T}. In particular, $\mathcal T_\tau$ can be linearised around the identity with an error of order $O(\tau)$ in various norms; see Corollary \ref{coro-1}.
\end{enumerate}
The goal in this paper is a weak error analysis of 
(\ref{tdiscrAintro}) as an example for an implementable scheme in a {\em broad setting of applications}, including multiphase flow dynamics in complicated domains, 
which inherits (long time) discrete stability even for scalings $\varepsilon \ll 1$ to
properly simulate diffuse phase field dynamics; see \cite{DuFe,DeDzEl}, and also
item {\bf 5.} in Remark \ref{others1}. In fact, there exist several other works on weak error analysis for different schemes to solve (\ref{eq:SNS}) in the literature, most of which apply to restricted data settings, such as domains for $d=1$ being intervals $(a,b) \subset {\mathbb R}$, or related drift operators being generators of linear semigroups, for which spectral properties need be available for actual computations; see items {\bf 3}.--{\bf 4.} of Remark \ref{others1}. 
In this context, admissible spatial meshes are often needed to be equi-distant, which clearly affects the capacity of related schemes to simulate multiscale phase evolution via \eqref{eq:SNS} for general data settings.

\section{Mathematical framework}
\label{sec:framework}

\subsection{Probability setup}\label{sec:prob}

Let $(\Omega,\F,(\F_t)_{t\geq0},\prst)$ be a stochastic basis with a complete, right-continuous filtration. The process $W$ is a cylindrical Wiener process, that is, $W(t)=\sum_{k\geq1}\beta_j(t) e_j$ with $(\beta_i)_{i\geq1}$ being mutually independent real-valued standard Wiener processes relative to $(\F_t)_{t\geq0}$, and $(e_i)_{i\geq1}$ a complete orthonormal system in a separable Hilbert space $\mathfrak{U}$.
Let us now give the precise definition of the diffusion coefficient $\varPhi$ taking values in the set of Hilbert-Schmidt operators $L_2(\mathfrak U;\mathbb H)$, where
$\mathbb H$ can take the role of various Hilbert spaces such as $L^2(\mt)$, $W^{1,2}(\mt)$ and $W^{2,2}(\mt)$ for which we use the shorthand notations $L^2_x, W^{1,2}_x$ and $W^{2,2}_x$.

In the following we formulate two sets of assumptions regarding the (regularity of the) diffusion coefficient $\Phi$, which will allow us to derive high moment bounds for higher derivatives of the solution of (\ref{eq:SNS}) in Section \ref{subsec:solution}. In the first case {\bf (a)} below we consider a general
nonlinear multiplicative noise.
In the second case \ref{N2} we assume that $\Phi$ is an affine linear function of $u$. As we shall see below, assumption \ref{N1} is always sufficient for our analysis in the case $d=1,2$; see Remark \ref{remark:noise'}, {\bf 3.}.
\begin{enumerate}[label={({\bf N}\arabic{*})}]
\item\label{N1} {\bf (a)}
 For $\bfz\in L^2(\mt)$ let $\,\varPhi(\bfz):\mathfrak{U}\rightarrow L^2(\mt)$ be defined by $\Phi(\bfz)e_k=\bfg_k(\cdot,\bfz(\cdot))$, where
$\bfg_k\in C^1(\mt\times\R)$ and\footnote{We denote by $\nabla_x$ the derivative with respect to first variable, \emph{i.e.}, the $d$-valued spatial variable and by $\nabla_\xi$ the derivative with respect to the second variable.}
\begin{align}\label{eq:phi}
\begin{aligned}
\sum_{k\geq1}|\bfg_k(x,\bfxi)|^2 \leq c(1+|\bfxi|^2)&,\qquad
\sum_{k\geq1}|\nabla_{\bfxi} \bfg_k(x,\bfxi)|^2\leq c,\\
\sum_{k\geq1}|\nabla_x \bfg_k(x,\bfxi)|^2 &\leq c(1+|\bfxi|^2),\quad x\in \mt,\,\bfxi\in\R.
\end{aligned}
\end{align}
Note that this implies
\begin{align}\label{eq:phi1a}
\|\Phi(\bfu)-\Phi(\bfv)\|_{L_2(\mathfrak U;L^2_x)}&\leq\,c\|\bfu-\bfv\|_{L^2_x}\qquad\forall \bfu,\bfv\in L^2_x,\\
\label{eq:phi1b}
\|\Phi(\bfu)\|_{L_2(\mathfrak U;W^{1,2}_x)}&\leq\,c\big(1+\|\bfu\|_{W^{1,2}_x}\big)\qquad\forall \bfu\in W^{1,2}_x,\\
\label{eq:phi1c}
\|D\Phi(\bfu)\|_{L_2(\mathfrak U;\mathcal L( L^{2}_x))}&\leq\,c\qquad\forall \bfu\in L^{2}_x.
\end{align}
 {\bf (b)} We require
 $\bfg_k\in C^2(\mt\times\R)$, together with
\begin{align}\label{eq:phi2}
\begin{aligned}
\sum_{k\geq1}|\nabla_x^2\bfg_k(x,\bfxi)|^2 \leq c(1+|\bfxi|^2)&,\qquad
\sum_{k\geq1}|\nabla_{x,\xi}\nabla_{\bfxi} \bfg_k(x,\bfxi)|^2&\leq c,\quad x\in \mt,\,\bfxi\in\R,
\end{aligned}
\end{align}
and sometimes also
\begin{align}\label{eq:phi3}
\begin{aligned}
\sum_{k\geq1}|\nabla_{x}^3\bfg_k(x,\bfxi)|^2 \leq c(1+|\bfxi|^2)&,
\qquad \sum_{k\geq1}|\nabla_\xi\nabla_{x,\bfxi}^2 \bfg_k(x,\bfxi)|^2 \leq c.
\end{aligned}
\end{align}
\item\label{N2}
We assume that
  $\Phi(\bfz)e_k=\alpha_k(\cdot)\bfz(\cdot)+\beta_k(\cdot)$, where
$\alpha_k,\beta_k\in C^3(\mt)$ and 
\begin{align}\label{eq:phismooth}
\sum_{k\geq1}\|\alpha_k\|_{C^3(\mt)}+ \sum_{k\geq1}\|\beta_k\|_{C^3(\mt)}\leq c.
\end{align}
\end{enumerate}
Note that \eqref{eq:phi} and \eqref{eq:phi2} imply
\begin{align}\label{eq:phi2b}
\|\Phi(\bfu)\|_{L_2(\mathfrak U;W^{2,2}_x)}&\leq\,c\big(1+\|\bfu\|_{W^{1,4}_x}^2+\|\bfu\|_{W^{2,2}_x}\big)\qquad\forall \bfu\in W^{2,2}_x,\\
\label{eq:phi2a}
\|D\Phi(\bfu)\|_{L_2(\mathfrak U;\mathcal L( W^{1,2}_x))}&\leq\,c\qquad\forall \bfu\in W^{1,2}_x,\\
\label{eq:phi2c}
\|D^2\Phi(\bfu)\|_{L_2(\mathfrak U;\mathcal L( L^{4}_x\times L^4_x;L^2_x))}&\leq\,c\qquad\forall \bfu\in L^{2}_x,
\end{align}
whereas \eqref{eq:phi}, \eqref{eq:phi2} and  \eqref{eq:phi3} together yield 
\begin{align}\label{eq:phi3b}
&\|\Phi(\bfu)\|_{L_2(\mathfrak U;W^{3,2}_x)}\leq\,c\big(1+\|\bfu\|_{W^{1,6}_x}^3+\|\bfu\|_{W^{2,2}_x}^2+\|\bfu\|_{W^{3,2}_x}\big)\qquad\forall \bfu\in W^{3,2}_x,\\\label{eq:phi3c}
&\|D^3\Phi(\bfu)\|_{L_2(\mathfrak U;\mathcal L( L^{6}_x\times L^6_x\times L^6_x;L^2_x))}\leq\,c\qquad\forall \bfu\in L^{2}_x.
\end{align}
Assumption \eqref{eq:phi} allows us to define stochastic integrals.
Given an $(\mathfrak F_t)-$adapted process $\bfu\in L^2(\Omega;C([0,T];L^2(\mt)))$, the stochastic integral $$t\mapsto\int_0^t\varPhi(\bfu)\,\dif W$$
is a well-defined process taking values in $L^2(\mt)$ (see \cite{PrZa} for a detailed construction). Moreover, we can multiply by test functions to obtain
 \begin{align*}
\bigg\langle\int_0^t \varPhi(\bfu)\,\dd W,\bfphi\bigg\rangle_{L^2_x}=\sum_{i\geq 1} \int_0^t\langle \varPhi(\bfu) e_i,\bfphi\rangle_{L^2_x}\,\dd\beta_i \qquad \forall\, \bfphi\in L^2(\mt).
\end{align*}
Similarly, we can define stochastic integrals with values in $W^{1,2}(\mt)$ and $W^{2,2}(\mt)$ respectively if $\bfu$ belongs to the corresponding class.


\subsection{The concept of solutions}
\label{subsec:solution}

In this section we give a precise definition of a solution to \eqref{eq:SNS} and derive some of its basic properties.
\begin{definition}\label{def:inc2d}
 Let $(\Omega,\mf,(\mf_t)_{t\geq0},\prst)$ be a given stochastic basis with a complete right-con\-ti\-nuous filtration and an $(\mf_t)$-cylindrical Wiener process $W$. Suppose that $\Phi$ satisfies \ref{N1} {\bf (a)} and that $d=1,2,3$. Let $\bfu_0$ be an $\mf_0$-measurable random variable with values in $L^2(\mt)$. Then $\bfu$ is called a \emph{weak pathwise solution} \index{incompressible Navier--Stokes system!weak pathwise solution} to \eqref{eq:SNS} with the initial condition $\bfu_0$ provided
\begin{enumerate}
\item the function $\bfu$ is $(\mf_t)$-adapted and
$$\bfu \in C([0,T];L^2(\tor))\cap L^2(0,T; W^{1,2}(\tor))\quad\text{$\p$-a.s.},$$
\item the equation
\begin{align*}
\int_{\mt}\bfu(t)&\,\bfvarphi\dx-\int_{\mt}\bfu_0\,\bfvarphi\dx
\\&=-\int_0^t\int_{\mt}f(\bfu)\,\bfvarphi\dx\,\dif t-\int_0^t\int_{\mt}\nabla\bfu\cdot\nabla\bfvarphi\dx\,\dif s+\int_0^t\int_{\mt}\Phi(\bfu)\,\bfvarphi\dx\,\dif W
\end{align*}
holds $\p$-a.s. for all $\bfvarphi\in W^{1,2}(\mt)$ and all $t\in[0,T]$.
\end{enumerate}
\end{definition}
The existence of a solution can be shown by the popular variational approach; see \emph{e.g.} \cite{PrRo}.
\begin{theorem}\label{thm:inc2d}
Let $(\Omega,\mf,(\mf_t)_{t\geq0},\prst)$ be a given stochastic basis with a complete right-continuous filtration and an $(\mf_t)$-cylindrical Wiener process $W$. Suppose that $\Phi$ satisfies \ref{N1} {\bf (a)} and that $d=1,2,3$. Let $\bfu_0$ be an $\mf_0$-measurable random variable such that $\bfu_0\in L^q(\Omega;L^2(\mathbb T^d))$ for some $q>2$. Then there exists a unique weak pathwise solution to \eqref{eq:SNS} in the sense of Definition \ref{def:inc2d} with the initial condition $\bfu_0$.
\end{theorem}
%

\begin{lemma}\label{lemma:3.1A} 
 Suppose that $\Phi$ satisfies \ref{N1} {\bf (a)} and let $\bfu$ be the weak pathwise solution to \eqref{eq:SNS}.
 \begin{enumerate}
 \item[(a)] 
Assume that $\mathcal E(\bfu_0)\in L^{q}(\Omega)$ for some $q\geq 1$. Then we have
\begin{align}
\label{lem:}\E\bigg[\bigg(\sup_{0\leq s\leq T}\mathcal E(\bfu(s))+\int_0^T\|D\mathcal E(u(s))\|_{L^2_x}^2\,\mathrm{ds}\bigg)^q\bigg]&\leq\,c\,\E\big[\big(\mathcal E(\bfu_0)+1\big)^q\big],
\end{align}
where $c=c(q,T,\bfu_0)>0$.
 \item[(b)] Assume that $\mathcal E(\bfu_0)\in L^{1}(\Omega)$. Then we have for any $\tau\in(0,T)$
 \begin{align}
 \label{eq:inerror}\E\bigg[\sup_{0\leq s\leq \tau}\|\bfu(s)-\bfu_0\|^{2}_{L^{2}_x}+\int_0^\tau\|\nabla\bfu(s)\|^2_{L^2_x}\,\mathrm{ds}\bigg]&\leq\,c\tau,
\end{align}
where $c=c(T,\bfu_0)>0$ is independent of $\tau>0$.
\end{enumerate}
\end{lemma}
\begin{proof}
Part (a) is standard and similar results can be found in the literature, see, \emph{e.g.}, \cite{MP1}). For the reader's convenience we decided to give the details nevertheless.
Applying It\^{o}'s formula (this can be justified by truncating the function $F$ and applying the It\^o-formula in Hilbert spaces from \cite[Theorem 4.17]{PrZa}.) to the function $t\mapsto \mathcal E\bigl(\bfu(t)\bigr)$ yields
\begin{align*}
&{\mathcal E}\bigl( u(t)\bigr) + 
  \int_0^t\big\| D{\mathcal E}\bigl(u(s)\bigr)\big\|_{L^2_x}^2\,{\rm d}s \\ \nonumber
  & = {\mathcal E}\bigl( u_0\bigr)+\int_{\mt}\int_0^t D\mathcal E(u(s))\,\Phi(u(s))\,\dd W\dx\\
&\quad  + \frac{1}{2}\sum_{k\geq1} \int_0^t \int_{{\mathbb T}^d} D^2 {\mathcal E}(u(s))\bigl( \Phi(u(s))e_k,\Phi(u(s))e_k\bigr) \, {\rm d}x\,{\rm d}s\\
&=:(\mathrm{I})+(\mathrm{II}).
\end{align*}
We have  by \eqref{eq:phi}
\begin{align*}
(\mathrm{II})&=\frac{1}{2}\sum_{k\geq1} \int_0^t \bigl\| \nabla\Phi(u(s))e_k\bigr\|_{L^2_x}^2\,{\rm d}s+\frac{1}{2}\sum_{k\geq1} \int_0^t \int_{{\mathbb T}^d}f'(u(s))|\Phi(u(s))e_k|^2 \, {\rm d}x\,{\rm d}s\\
&\leq\,c \int_0^t \int_{{\mathbb T}^d}\bigl(|u|^2+\bigl| \nabla u\bigr|^2\bigr) \, {\rm d}x\,{\rm d}s+c\int_0^t \int_{{\mathbb T}^d}\big(|u|^2+1\big)^2 \, {\rm d}x\\
&\leq\,c \int_0^t \int_{{\mathbb T}^d}\bigl| \nabla u\bigr|^2 \, {\rm d}x\,{\rm d}s+ c\int_0^t \int_{{\mathbb T}^d}\big(F(u)+1\big) \, {\rm d}x\,{\rm d}s\\
&\leq\,c\bigg(\int_0^t\mathcal E(u(s))\,{\rm d}s+1\bigg).
\end{align*}
Similarly, by Burkholder-Davis-Gundy inequality,
\begin{align*}
\E\bigg[\bigg(\sup_{0\leq t\leq T}(\mathrm{I})\bigg)^q\bigg]&\leq\,c\,\E\bigg[\bigg(\sum_{k\geq1}\int_0^t\bigg( \int_{\mt}D\mathcal E(u(s))\,\Phi(u(s))e_k\dx\bigg)^2\,{\rm d}s\bigg)^{\frac{q}{2}}\bigg]\\
&\leq\,c\,\E\bigg[\bigg(\int_0^t\|D\mathcal E(u(s))\|_{L^2_x}^2\big( \|u(s)\|_{L^2_x}^2+1\big)\,{\rm d}s\bigg)^{\frac{q}{2}}\bigg]\\
&\leq\,c\,\E\bigg[\bigg(\sup_{0\leq s\leq t}\mathcal E(u(s))+1\bigg)^q\bigg]+c\,\E\bigg[\bigg(\int_0^t\|D\mathcal E(u(s))\|^2_{L^2_x}\,{\rm d}s\bigg)^q\bigg].
\end{align*}
For the second estimate we apply It\^{o}'s formula to $t\mapsto \frac{1}{2}\|{\bfu}(t)-\bfu_0\|_{L^2_x}^2$
and obtain similarly to \eqref{eq:2304}
 \begin{align*}
\tfrac{1}{2}&\|\bfu(\tau)-\bfu_0\|^2_{L^2_x}+\int_0^\tau\|\nabla\bfu(s)\|^2_{L^2_x}\,{\rm d}s\\
&=\int_0^\tau\int_{\mt}\nabla\bfu(s):\nabla\bfu_0\dx\,{\rm d}s-\int_0^\tau\int_{\mt}f(\bfu(s))\,(\bfu(s)-\bfu_0)\dx\,{\rm d}s
\\&\quad +\int_{\mt}\int_0^\tau(\bfu(s)-\bfu_0)\cdot\Phi(\bfu(s))\,\dd W\dx+\tfrac{1}{2}\int_0^\tau\|\Phi(\bfu(s))\|^2_{L_2(\mathfrak U;L^2(\mt))}\,{\rm d}s.
\end{align*}
We clearly have
\begin{align*}
\E\biggl[\int_0^\tau\int_{\mt}\nabla\bfu(s):\nabla\bfu_0\dx\,{\rm d}s\biggr]&\leq\,\bigg(\E\biggl[\int_0^\tau\|\nabla\bfu(s)\|_{L^2_x}^2\,{\rm d}s\biggr]\bigg)^{\frac{1}{2}}\bigg(\E\biggl[\int_0^\tau\|\nabla\bfu_0\|_{L^2_x}^2\,{\rm d}s\biggr]\bigg)^{\frac{1}{2}}\leq\,c\tau
\end{align*}
using \eqref{lem:}.
Now \eqref{lem:} implies
\begin{align*}
{\mathbb E}\Bigl[&\int_0^\tau\int_{\mt}f(\bfu(s))\,(\bfu(s)-\bfu_0)\dx\,{\rm d}s\Bigr]\leq\,c\tau\E\bigg[\sup_{0\leq s\leq \tau}\|\bfu(s)\|^{q+1}_{L^{4}_x}+\|\bfu_0\|_{L^{4}_x}^{4}\bigg]\\
&\leq \,c\tau\E\bigg[\sup_{0\leq s\leq \tau}\mathcal E(u(s))+\mathcal E(u_0)+1\bigg]\leq\,c\tau.
\end{align*}
Finally,
\begin{align*}
\E\Bigl[\int_0^\tau\|\Phi(\bfu(s))\|^2_{L_2(\mathfrak U;L^2_x)}\,{\rm d}s\Bigr]&\leq\,c\, {\mathbb E}\biggl[\int_0^\tau\big( \|\bfu(s)\|_{L^2_x}^2+1 \big)\,{\rm d}s\biggr]\\&\leq\, c\tau\bigg({\mathbb E}\biggl[\sup_{0\leq s\leq \tau}\|\bfu(s)\|_{L^2_x}^2\biggr]+1\bigg)\leq\,c\tau
\end{align*}
by \eqref{eq:phi} and \eqref{lem:}.
Arguing as above in the proof of \eqref{lem:} we have
\begin{align*}
\E\bigg[\sup_{0\leq t\leq \tau}\bigg|\int_{\mt}&\int_0^t(\bfu(s)-\bfu_0)\cdot\Phi(\bfu(s))\,\dd W\dx\bigg|\bigg]\\
&\leq\,c\,\E\bigg[\max_{0\leq s\leq \tau}\|\bfu(s)-\bfu_0\|_{L^2_x}\bigg(\int_0^\tau\|\Phi(\bfu(s))\|^2_{L_2(\mathfrak U,L^2_x)}\,{\rm d}s\bigg)^{\frac{1}{2}}\bigg]\\
&\leq\,\kappa\,\E\bigg[\sup_{0\leq s\leq \tau}\|\bfu(s)-\bfu_0\|^2_{L^2_x}\bigg]+c(\kappa)\E\bigg[\int_0^\tau\|\Phi(\bfu(s))\|^2_{L_2(\mathfrak U,L^2_x)}\,{\rm d}s\bigg]\\
&\leq\,\kappa\,\E\bigg[\sup_{0\leq s\leq \tau}\|\bfu(s)-\bfu_0\|^2_{L^2_x}\bigg]+c(\kappa)\E\bigg[\int_0^\tau\big(1+\|\bfu(s)\|^2_{L^2_x}\big)\,{\rm d}s\bigg]\\
&\leq\,\kappa\,\E\bigg[\sup_{0\leq s\leq \tau}\|\bfu(s)-\bfu_0\|^2_{L^2_x}\bigg]+c(\kappa)\tau\E\bigg[\sup_{0\leq s\leq \tau}\|\bfu\|^2_{L^2_x}+1\bigg]\\
&\leq\,\kappa\,\E\bigg[\sup_{0\leq s\leq \tau}\|\bfu(s)-\bfu_0\|^2_{L^2_x}\bigg]+c(\kappa)\tau
\end{align*} 
using \eqref{lem:}. Plugging all together shows \eqref{eq:inerror}.
\end{proof}
The following lemma collects moment bounds (\emph{i.e.}, \textcolor{red}{$p\geq 2$}) in higher norms and is reminiscent of  \cite[Lemma 3.1]{MP1}.
\begin{lemma}\label{lemma:3.1B} 
 Suppose that $\Phi$ satisfies \ref{N1} {\bf (a)} and let $\bfu$ be the weak pathwise solution to \eqref{eq:SNS}.
 \begin{enumerate}
 \item[(a)] Assume that $\bfu_0\in L^{p}(\Omega,L^{p}(\mt))$ for some \textcolor{red}{$p\geq 2$}. Then we have
\begin{align}
\label{lem:B1}\E\bigg[\sup_{0\leq s\leq T}\|\bfu(s)\|^{p}_{L^{p}_x}+\int_0^T\|\nabla|\bfu(s)|^{\frac{p}{2}}\|^2_{L^2_x}\,\mathrm{ds}\bigg]&\leq\,c\,\E\big[\|\bfu_0\|^{p}_{L^{p}_x}+1\big].
\end{align}
 \item[(b)] Assume that $\bfu_0\in L^{p}(\Omega,W^{1,p}(\mt))$ for some $p>1$. Then we have
\begin{align}
\label{lem:B2}\E\bigg[\sup_{0\leq s\leq T}\|\nabla\bfu(s)\|^{p}_{L^{p}_x}+\int_0^T\|\nabla|\nabla\bfu(s)|^{\frac{p}{2}}\|^2_{L^2_x}\,\mathrm{ds}\bigg]&\leq\,c\,\E\big[\|\nabla\bfu_0\|^{p}_{L^{p}_x}+1\big].
\end{align}
 \item[(c)] Assume that $\bfu_0\in L^{p}(\Omega,W^{2,p}(\mt))\cap L^{2p}(\Omega,W^{1,p}(\mt))$ hold and that $\Phi$ satisfies additionally \ref{N1} {\bf (b)}. Then we have
\begin{align}
\label{lem:B3}\E\bigg[\sup_{0\leq s\leq T}\|\nabla^2\bfu(s)\|^{p}_{L^{p}_x}+\int_0^T\|\nabla|\nabla^2\bfu(s)|^{\frac{p}{2}}\|^2_{L^2_x}\,\mathrm{ds}\bigg]&\leq\,c\,\E\big[\|\nabla^2\bfu_0\|^{p}_{L^{p}_x}+\|\nabla\bfu_0\|^{2p}_{L^{2p}_x}+1\big].
\end{align}
\end{enumerate}
Here $c=c(T,\Phi,p)>0$.
\end{lemma}
\begin{proof}
Ad~{\bf (a)}.
We apply It\^{o}'s formula (see \cite[Thm. 4.17]{PrZa}) to the functional $t\mapsto \frac{1}{p}\|\bfu(t)\|_{L^p_x}^p$ and obtain
 \begin{align}
\nonumber
 \tfrac{1}{p}\|\bfu(t)\|^p_{L^p_x}&+\int_0^t\int_{\mt}|\bfu(s)|^{p-2}|\nabla u(s)|^2\dx\ds+\int_0^t\int_{\mt}|\bfu(s)|^{p-2}f(\bfu(s))\,\bfu(s)\dxs\\
\label{eq:2304}&=
\tfrac{1}{p}\|\bfu_{0}\|^p_{L^p_x}
+\int_{\mt}\int_0^t|\bfu|^{p-2}\bfu\,\Phi(\bfu)\,\dd W\dx\\&\quad+\tfrac{p-1}{2}\sum_{i\geq 1}\int_0^t\int_{\mt}|\bfu(s)|^{p-2}|\Phi(\bfu(s))e_i|^2\dx\ds.
\nonumber
\end{align}
We have
\begin{align*}
\int_0^t\int_{\mt}|\bfu(s)|^{p-2}&|\nabla\bfu(s)|^2\dx\ds+\int_0^t\int_{\mt}|\bfu(s)|^{p-2}f(\bfu(s))\,u(s)\dxs\\
&\geq c\bigg(\int_0^t\int_{\mt}|\nabla|\bfu(s)|^{\frac{p}{2}}|^2\dx\ds-\int_0^t\|\bfu(s)\|^p_{L^p_x}\ds\bigg).
\end{align*}
On account of \eqref{eq:phi} we obtain
\begin{align*}
\sum_{i\geq 1}\int_0^t\int_{\mt}&|\bfu(s)|^{p-2}|\Phi(\bfu(s))e_i|^2\dx\ds
\leq\,c\int_0^t\big(\|\bfu(s)\|_{L^p_x}^{p}+1\big)\ds.
\end{align*}
Finally,
we use Burkholder-Davis-Gundy inequality,
\eqref{eq:phi} and Young's inequality to conclude
\begin{align*}
\E\bigg[\sup_{0\leq t\leq T}&\bigg|\int_{\mt}\int_0^t|u|^{p-2}\bfu\,\Phi(\bfu)\,\dd W\dx\bigg|\bigg]\\&\leq\,c\,\E\bigg[\bigg(\sum_{i\geq1}\int_{0}^{T}\bigg(\int_{\mt}|\bfu(s)|^{p-2}\bfu(s)\,\Phi(\bfu(s))e_i\dx\bigg)^2\ds\bigg)^{\frac{1}{2}}\bigg]\\
&\leq\,c\,\E\bigg[\bigg(\int_0^T\big(\|\bfu(s)\|_{L^p_x}^{2p}+1\big)\ds\bigg)^{\frac{1}{2}}\bigg]\\
&\leq\,c\,\E\bigg[\bigg(\sup_{0\leq s\leq T}\|\bfu(s)\|_{L^p_x}^{p}\int_0^T\|\bfu(s)\|_{L^p_x}^{p}\ds+1\bigg)^{\frac{1}{2}}\bigg]\\
&\leq\,\kappa\,\E\bigg[\sup_{0\leq s\leq T}\|\bfu(s)\|_{L^p_x}^{p}\bigg]+c(\kappa)\E\bigg[\int_0^T\|\bfu(s)\|_{L^p_x}^{p}\ds+1\bigg],
\end{align*} 
where Gronwall's lemma comes again into play.
Combining everything, choosing $\kappa$ small enough and using Gronwall's lemma yields the claim.

Ad~{\bf (b)}. Differentating \eqref{eq:SNS} with respect to $\gamma\in\{1,2,3\}$ yields
\begin{align}\label{eq:DSNS}
\dd\partial_\gamma\bfu= \Bigl(\Delta\partial_\gamma\bfu - f'(u)\partial_\gamma u\Bigr)\dt+\partial_\gamma\{\Phi(\bfu)\dd W\}.
\end{align}
Applying It\^o's formula to $t\mapsto \frac{1}{p}\| \partial_\gamma\bfu(t)\|_{L^p_x}^p$
 we obtain
 \begin{align}
\label{eq:2304}
\begin{aligned}
 \tfrac{1}{p}\|\partial_\gamma\bfu(t)\|_{L^p_x}^p&+\int_0^t\int_{\mt}|\partial_\gamma\bfu(s)|^{p-2}|\partial_\gamma\nabla\bfu(s)|^2\dx\ds\\&\quad+\int_0^t\int_{\mt}|\partial_\gamma\bfu(s)|^{p-2}f'(\bfu(s))|\partial_\gamma\bfu(s)|^2\dxs\\
&=
\tfrac{1}{p}\|\partial_\gamma\bfu_{0}\|^p_{L^p_x}
+\int_{\mt}\int_0^t|\partial_\gamma\bfu|^{p-2}\partial_\gamma\bfu\,\partial_\gamma\{\Phi(\bfu)\,\dd W\}\dx\\&\quad+\tfrac{p-1}{2}\sum_{i\geq 1}\int_0^t\int_{\mt}|\partial_\gamma\bfu(s)|^{p-2}|\partial_\gamma\{\Phi(\bfu(s))e_i\}|^2\dx\ds.
\end{aligned}
\end{align}
We have
\begin{align*}
\int_0^t\int_{\mt}|\partial_\gamma\bfu(s)|^{p-2}&|\nabla\partial_\gamma\bfu(s)|^2\dx\ds+\int_0^t\int_{\mt}|\partial_\gamma\bfu(s)|^{p-2}f'(\bfu(s))\,|\partial_\gamma\bfu(s)|^2\dxs\\
&\geq c\bigg(\int_0^t\big\|\nabla|\partial_\gamma\bfu(s)|^{\frac{p}{2}}\|^2_{L^2_x}\ds-\int_0^t\|\partial_\gamma\bfu(s)\|^p_{L^p_x}\ds\bigg).
\end{align*}
On account of \eqref{eq:phi} we obtain
\begin{align*}
\sum_{i\geq 1}\int_0^t\int_{\mt}&|\partial_\gamma\bfu(s)|^{p-2}|\partial_\gamma\{\Phi(\bfu(s))e_i\}|^2\dx\ds\\
&\leq\,c\int_0^t\big(\|\partial_\gamma\bfu(s)\|^{p}_{L^p_x}+\|u(s)\|_{L^p_x}^p+1\big)\ds\leq\,c\int_0^t\big(\|\nabla\bfu(s)\|_{L^p_x}^{p}+1\big)\ds.
\end{align*}
Finally,
we use Burkholder-Davis-Gundy inequality,
\eqref{eq:phi} and Young's inequality to conclude
\begin{align*}
\E\bigg[\sup_{0\leq t\leq T}&\bigg|\int_{\mt}\int_0^t|\partial_\gamma u|^{p-2}\partial_\gamma\bfu\,\partial_\gamma\{\Phi(\bfu)\}\,\dd W\dx\bigg|\bigg]\\&\leq\,c\,\E\bigg[\bigg(\sum_{i\geq1}\int_{0}^{T}\bigg(\int_{\mt}|\partial_\gamma\bfu(s)|^{p-2}\partial_\gamma\bfu(s)\,\partial_\gamma\{\Phi(\bfu(s))\}e_i\dx\bigg)^2\ds\bigg)^{\frac{1}{2}}\bigg]\\
&\leq\,c\,\E\bigg[\bigg(\int_0^T\big(\|\nabla\bfu(s)\|_{L^p_x}^{2p}+1\big)\ds\bigg)^{\frac{1}{2}}\bigg]\\
&\leq\,c\,\E\bigg[\bigg(\sup_{0\leq s\leq T}\|\nabla\bfu(s)\|_{L^p_x}^{p}\int_0^T\|\nabla\bfu(s)\|_{L^p_x}^{p}\ds+1\bigg)^{\frac{1}{2}}\bigg]\\
&\leq\,\kappa\,\E\bigg[\sup_{0\leq s\leq T}\|\nabla\bfu(s)\|_{L^p_x}^{p}\bigg]+c(\kappa)\E\bigg[\int_0^T\|\nabla\bfu(s)\|_{L^p_x}^{p}\ds+1\bigg],
\end{align*} 
where Gronwall's lemma comes again into play.
Combining everything, choosing $\kappa$ small enough and using Gronwall's lemma yields the claim.

Ad~{\bf (c)}. Differentating \eqref{eq:DSNS} again yields
\begin{align}\label{eq:D2SNS}
\begin{aligned}
\dd\partial_{\gamma_1}\partial_{\gamma_2}\bfu&= \Bigl(\Delta\partial_{\gamma_1}\partial_{\gamma_2}\bfu - f'(u)\partial_{\gamma_1}\partial_{\gamma_2} u\Bigr)\dt+\partial_{\gamma_1}\partial_{\gamma_2}\{\Phi(\bfu)\}\dd W\\
&\quad- f''(u)\partial_{\gamma_1} u\partial_{\gamma_2} u\dt,
\end{aligned}
\end{align}
where $\gamma_1,\gamma_2\in\{1,2,3\}$. The nonlinearity can be handled as in {\bf (b)}, but for the noise we obtain the two terms
\begin{align*}
\partial_{\gamma_1}&\partial_{\gamma_2}\Phi(\bfu)\dd W+
\partial_{\gamma_1}D_\xi\Phi(\bfu)\partial_{\gamma_2}u\dd W+
\partial_{\gamma_2}D_\xi\Phi(\bfu)\partial_{\gamma_1}u\dd W\\&+D^2_\xi\Phi(\bfu)\partial_{\gamma_1}\partial_{\gamma_2}u\dd W+D^2_\xi\Phi(\bfu)\partial_{\gamma_1}u\partial_{\gamma_2}u\dd W.
\end{align*}
The first four of them can be estimated similarly to {\bf (b)} using \eqref{eq:phi2} but the last one requires more care (note that it disappears for affine linear noise). For the correction term we have 
\begin{align*}
\sum_{i\geq 1}\int_0^t\int_{\mt}&|\partial_{\gamma_1}\partial_{\gamma_2}\bfu(s)|^{p-2}|D^2_\xi\Phi(\bfu(s))e_i\partial_{\gamma_1}u(s)\,\partial_{\gamma_2}u(s)|^2\dx\ds\\
&\leq\,c\int_0^t\big(\|\nabla^2\bfu(s)\|_{L^p_x}^{p}+\|\nabla u(s)\|_{L^{2p}_x}^{2p}+1\big)\ds
\end{align*}
on account of \ref{N1}.
Applying expectations and using part {\bf (b)} (with the assumption $u_0\in L^{2p}(\Omega;W^{1,p}(\mt))$ the second term is bounded, whereas the first one can be handled by Gronwall's lemma. The supremum of the corresponding stochastic integral
\begin{align*}
\int_{\mt}\int_0^t|\partial_{\gamma_1}\partial_{\gamma_2} u|^{p-2}D^2_\xi\Phi(\bfu)\partial_{\gamma_1}u\,\partial_{\gamma_2}u\,\dd W\dx
\end{align*}
can be estimated by similar means.
\end{proof}


%

\subsection{The semigroup of the Laplace operator}
In this subsection we recall some well-known facts concerning the Laplace operator and (the discretisation of) its semigroup. We denote by $\mathcal L(\mathbb X,\mathbb Y)$ the space of bounded linear operators between two Banach spaces $\mathbb X$ and $\mathbb Y$ and write $\mathcal L(\mathbb X)$ for $\mathcal L(\mathbb X,\mathbb X)$. 
It is classical that
there is a basis of $L^2(\mt)$ consisting of eigenfunctions $(v_j)_{j\geq1}$ of $\mathcal A=-\Delta $ with positive eigenvalues $(\nu_j)_{j\geq1}$ such that $\nu_j\rightarrow\infty$ as $j\rightarrow\infty$. We can use that to define powers of $\mathcal A$ by setting
\begin{align*}
\mathcal A^r\bfu=\sum_{j\geq1}\nu_j^r\langle\bfu,v_j\rangle_{L^2_x}v_j,\quad r\geq0,
\end{align*}
for functions $\bfu$ from 
\begin{align*}
W^{r,2}(\mt):=\bigg\{\bfu=\sum_{j\geq1}\langle\bfu,v_j\rangle_{L^2_x}v_j:\,\sum_{j\geq1}\nu_j^{r}\langle\bfu,v_j\rangle_{L^2_x}^2<\infty\bigg\}.
\end{align*}
Similarly, we can define for $r<0$ the space $W^{r,2}(\mt)$ as the closure
of $L^2$ with respect to the norm
\begin{align*}
\|\bfu\|^2_{W^{r,2}}:=\sum_{j\geq1}\nu_j^r\langle\bfu,v_j\rangle_{L^2_x}^2.
\end{align*}
It is classical that $\mathcal A$ generates a strongly continuous semigroup and it is well-known that its discretisation $\mathcal S_{\tau}=(\mathrm{Id}+\tau\mathcal A)^{-1}$ satisfies
\begin{align}
\label{eq:S1}
\|\mathcal A^{\beta}\mathcal S_{\tau}^k\|_{\mathcal L(W^{r,2}_x)}&\leq\,c_r(k\tau)^{-\beta}\quad \forall\beta\geq0,\quad k\in\N,\quad r\geq0\\
\label{eq:S2}\|\mathcal A^{\beta}\mathcal S(t)\|_{\mathcal L(L^2_x)}&\leq\,ct^{-\beta}\quad \forall\beta\geq0,\\
\label{eq:S3}\|\mathcal A^{-\beta}(\mathrm{Id}-\mathcal S_{\tau})\|_{\mathcal L(L^2_x)}&\leq\,c\tau^{\beta}\quad \forall\beta\in[0,1].
\end{align}
Note that \eqref{eq:S2} follows from $\sup_t\|t\mathcal A\mathcal S(t)\|_{\mathcal L(L^2_x)}<\infty$, while \eqref{eq:S3} is a consequence of $\mathrm{Id}-\mathcal S_{\tau}=\tau\mathcal A \mathcal S_{\tau}$ and \eqref{eq:S1} with $k=1$. Let us also mention the following estimate which controls the error between $\SSS(\tau)$ and $\SSS_{\tau}$
\begin{align}
\label{eq:S4}\|\mathcal A^{\beta}(\SSS(\tau)-\mathcal S_{\tau})\|_{\mathcal L(L^2_x)}\leq\,c\tau^{-\beta}\quad \forall\beta\in[-2,\infty).
\end{align}
Finally, \eqref{eq:S4} easily implies
\begin{align}
\label{eq:S5}\|\mathcal A^{\beta}(\SSS(\tau)-\mathcal S_{\tau})\|_{\mathcal L(W^{r,2}_x,L^2_x)}\leq\,c\tau^{-\beta+\textcolor{blue}{r/2}}\quad \forall\beta,r:\,\beta-\textcolor{blue}{r/2}\in[-2,\infty).
\end{align}

\subsection{Malliavin calculus}\label{sec:mal}
We recall some basic facts from Malliavin calculus, see \cite{Nu} for a thorough introduction.
Given a cylindrical $(\mathfrak F_t)$-Wiener process as in Section \ref{sec:prob}, a smooth \textcolor{blue}{real-valued} function
$F$ defined on $\mathfrak U^n$ and smooth random variables with values in
$\psi_1,\dots,\psi_n\in L^2(0,T;\mathfrak U)$ we set
\begin{align*}
\mathcal D_t^{f}F\bigg(&\int_0^T\langle \psi_1,\dd W\rangle_{\mathfrak U},\dots,\int_0^T\langle \psi_n,\dd W\rangle_{\mathfrak U}\bigg)\\&=\sum_{i=1}^n\partial_iF\bigg(\int_0^T\langle \psi_1,\dd W\rangle_{\mathfrak U},\dots,\int_0^T\langle \psi_n,\dd W\rangle_{\mathfrak U}\bigg)\langle\psi_i(t),f\rangle_{\mathfrak U},\quad f\in\mathfrak U.
\end{align*}
This allows to define $\mathcal DF $ by $\mathcal DF(t)f=\mathcal D_t^{f}F$.
With this definition it can be shown $\mathcal D$ defines a closable operator on
$L^2(\Omega\times(0,T);\mathfrak U)$. A natural domain space for $\mathcal D$ is given by $\mathbb D^{1,2}$ which is the closure of the random variables (taking values in the set of smooth functions on $\mathfrak U^n$) with respect to the norm
\begin{align*}
\|F\|_{\mathbb D^{1,2}}^2:=\E\bigg[|F|^2+\int_0^T|\mathcal D_s F|^2\,\dd s\bigg].
\end{align*}
Note that $\mathcal D_{t} F=0$ for $t\geq t'$ provided $F$ is $(\mathfrak F_{t'})$-adapted. We have a version of the chain rule: For $\varphi\in C^1_b(\R)$ and $F\in\mathbb D^{1,2}$ we have $\varphi(F)\in \mathbb D^{1,2}$ together with the usual form $\mathcal D \varphi(F)=\varphi'(F)\mathcal D F$.\\
Most important for us is the Malliavin integration by parts formula
\begin{align*}
\E \biggl[ F\bigg(\int_0^T\langle \psi,\dd W\rangle_{\mathfrak U}\bigg)\bigg]=\E \bigg[\int_0^T\langle \mathcal D_s F,\psi(s)\rangle_{\mathfrak U}\,\dd s\bigg].
\end{align*}
It holds for all $F\in\mathbb D^{1,2}$ and all adapted $\psi\in L^2(\Omega\times(0,T);\mathfrak U)$ with $\E\big[\int_0^T\int_0^T|\mathcal D_s\psi(t)|^2\,\dd s\dt\big]<\infty$.

Similarly, we can define the Malliavin derivative of random variables taking values in $L^2(\mt)$. We denote by $\mathbb D^{1,2}(L^2(\mt))$ the set of random variables $G\in L^2(\Omega;(L^2(\mt)))$ with $G=\sum_{j\geq1} \alpha_j v_j$ with $(\alpha_j)\subset\mathbb D^{1,2}$ and $\sum_{j\geq 1}\int_0^T|\mathcal D_s G_j|^2\,\dd s<\infty$. Here $(v_j)_{j\geq 1}$ denotes a basis of $L^2(\mt)$ consisting of eigenfunctions of $-\Delta$. For $G\in \mathbb D^{1,2}(L^2(\mt))$ with $\bfU=\sum_{j\geq1} \alpha_j v_j$ and $f\in\mathfrak U$ we define
\begin{align*}
\mathcal D_s^f\bfU:=\sum_{j\geq1} (\mathcal D_s^f\alpha_j) v_j,\quad \mathcal D_s\bfU:=\sum_{j\geq1} (\mathcal D_s\alpha_j) v_j.
\end{align*}
Again the chain rule holds: For $\bfU\in \mathbb D^{1,2}(L^2(\mt))$ with $\bfU=\sum_{j\geq1} \alpha_j v_j$ and $\varphi\in C^1_b(L^2(\mt))$ we have
 $\varphi(\bfU)\in \mathbb D^{1,2}(L^2(\mt))$ with
 $\mathcal D \varphi(\bfU)=\langle D\varphi(\bfU),\mathcal D \bfU\rangle_{L^2_x}$.
Finally, for $\bfU\in\mathbb D^{1,2}(L^2(\mt))$, $u\in C^2_b(L^2(\mt))$ and $\bfpsi\in L^2\bigl(\Omega\times(0,T);L_2(\mathfrak U;L^2(\mt)\bigr)$ adapted we have
\begin{align}\label{eq:malpart}
\begin{aligned}
\E \biggl[\bigg\langle Du(\bfU),\int_0^T\bfpsi\,\dd W\bigg\rangle_{L^2_x}\biggr]
&=\E \biggl[\sum_{i\geq 1} \int_0^T D^2u(\bfU)(\mathcal D_s^{e_i}\bfU,\bfpsi(s)e_i)\,\dd s\biggr].
\end{aligned}
\end{align}

\section{Strong first Order Convergence Rate for additive Noise}\label{sec:add}
In \cite{MP1}, strong order convergence rate
${\mathcal O}(\sqrt{\tau})$ is shown for a time-implicit discretisation of (\ref{eq:SNS}) with multiplicative noise. In this section, this result is improved to
${\mathcal O}(\tau)$ in the presence of additive noise $\Phi (u) \equiv \Phi$ in (\ref{eq:SNS}). For this purpose, we
use the transform
\begin{equation}\label{transf1}
y(t) = u(t)- \int_0^t \Phi\, {\rm d}W(s) = u(t) - \Phi W(t) \qquad [ 0\leq t\leq T]
\end{equation}
to recast (\ref{eq:SNS}) into the form
\begin{align}\nonumber
\partial_t y &= \Delta y -f(y+ \Phi W) + \Delta [\Phi W] \\ \label{eq:SNS1}
&= \Delta y - \Bigl[f(y) +  3 y^2 \Phi W + 3 y \bigl[\Phi W\bigr]^2 \Bigr] + \Delta [\Phi W] - f\bigl( \Phi W\bigr)\,,
\end{align}
where we employ the following calculation for the last identity,
\begin{align*}
f(y+ \Phi W) &= \bigl( y+ \Phi W\bigr)^3 - \bigl( y + \Phi W\bigr) = y^3 + \bigl[ \Phi W\bigr]^3 + 3 y^2 \Phi W + 3 y \bigl[\Phi W\bigr]^2 - \bigl( y + \Phi W\bigr)\\
&= f(y) + f\bigl( \Phi W\bigr) + 3 y^2 \Phi W + 3 y \bigl[\Phi W\bigr]^2\, .
\end{align*}
Assuming sufficient regularity of $\Phi $, Lemma \ref{lemma:3.1B} 
implies the following corollary concerning the regularity of $y$.
\begin{corollary}\label{cor:add}
Let $\bfu$ be the unique weak pathwise solution to \eqref{eq:SNS}.
\begin{enumerate}
\item Assume that $\bfu_0\in L^p(\Omega,W^{2,2}(\mt))$ for some $p\geq2$ and $\Phi\in L_2(\mathfrak U;W^{2,2}(\mt))$. Then we have 
\begin{align}\label{eq:wwww}
\E\bigg[\sup_{0\leq s\leq T}\|\partial_ty(s)\|_{L^{2}_x}^p\bigg]+
\E\bigg[\sup_{0\leq s\leq T}\|y(s)\|_{W^{2,2}_x}^p\bigg]&\leq c(p,T,\Phi,\bfu_0).
\end{align}
\item
Assume that $\bfu_0\in L^2(\Omega,W^{3,2}(\mt))$ and $\Phi\in L_2(\mathfrak U;W^{3,2}(\mt))$. Then we have 
\begin{align}\label{eq:star}
\E\bigg[\sup_{0\leq s\leq T}\|\nabla\partial_t y(s)\|_{L^{2}_x}^2\bigg]\leq c(T,\Phi,u_0).
\end{align}
\end{enumerate}
\end{corollary}
A reformulation corresponding to \eqref{eq:SNS1} was used in \cite{BP1} to accomplish a corresponding goal, leading to an improved (local) rate of convergence for a discretisation of the 2D Navier-Stokes equation with additive noise. In comparison, the nonlinearity $f= F'$ in
(\ref{eq:SNS}) is again non-Lipschitz, but now is one-sided Lipschitz,
 which here leads to a discretisation scheme with strong rate ${\mathcal O}(\tau)$. 

 For its derivation,
we define $y^m := u^m - \Phi W(t_m)$ for iterates $(u^m)_{m=1}^M$ from (\ref{tdiscrAintro}),
%
which
satisfy the following identity,
\begin{equation}\label{numR1}
d_t y^m - \Delta y^m + f(y^m)+ 3 \vert y^m\vert^2\Phi W(t_m) + 3 y^m \vert \Phi 
W(t_m)\vert^2= \Delta [\Phi W(t_m)] - f(\Phi W(t_m))
\end{equation}
for $m \geq 1$, and $y^0 = u_0$. Here $d_t$ denotes the discrete time derivative, that is $d_t y^m=\tau^{-1}(y^m-y^{m-1})$ for $m\geq1$.
We make the following observations:
\begin{itemize}
\item[(1)] An equivalent derivation of equation (\ref{numR1}) is via an implicit Euler-based time discretisation for the random PDE (\ref{eq:SNS1}). Equation (\ref{numR1}) may be used instead of (\ref{tdiscrAintro}) as an alternative scheme to solve (\ref{eq:SNS}).
\item[(2)] In the strong error analysis for $(y^m)_{m=1}^M$ below, we prove order ${\mathcal O}(\tau)$; the proof rests on bounds for the time-derivative of the (weak variational) solution $y$ from (\ref{eq:SNS1})  in different norms, which are now available. Since $y^m - y(t_m) = u^m - u(t_m)$, this result transfers to
(\ref{tdiscrAintro}); see also Remark \ref{others1}, item {\bf 4.}
\end{itemize}
The following perturbation analysis evidences the role that the involved nonlinearity $f = F'$ will play in the strong error analysis in Subsection \ref{suse2}.

\subsection{A perturbation analysis for (\ref{numR1})}\label{suse1}
 We start with a perturbation analysis, where we refer to $(y^m_i)_{m=1}^M$ as the solution of (\ref{numR1}), with $y^0_i =y_{0,i}$, for $i \in \{1,2\}$. Subtracting both identities then leads to ($e^m := y^m_2-y^m_1$)
\begin{eqnarray*}
d_t e^m - \Delta e^m + \Bigl[ f(y^m_2) - f(y^m_1)\Bigr]+ 
3 \Bigl[ \vert y^m_2\vert^2 - \vert y^m_1\vert^2\Bigr] \Phi W(t_m)+ 3 e^m \vert \Phi W(t_m)\vert^2 = 0\, ,
\end{eqnarray*}
where $d_t e^m := \frac{1}{\tau} (e^m - e^{m-1})$. We write
\begin{equation}\label{erro1} d_te_m-\Delta e_m+{\tt I}^m+{\tt II}^m+{\tt III}^m=0
\end{equation}
 with an obvious meaning of the terms ${\tt I}^m$, ${\tt II}^m$ and ${\tt III}^m$.
 When written in weak form, we test with $e^m$, and treat resulting terms for ${\tt I}^m$ and ${\tt II}^m$ independently. For the first term ${\tt I}^m$ we use the following identity  which is based on binomial formula ($a,b \in {\mathbb R}$), 
 \begin{eqnarray*}
   b^3 -  a^3  &=&
 \frac{1}{2} \bigl( b^3 - b^2 a + b^2 a - a^3 \bigr) + 
 \frac{1}{2} \bigl( b^3 - b a^2 + b a^2 - a^3 \bigr) \\
 &=&\frac{1}{2} \bigl( b^2 (b-a)+ a ( b^2 - a^2)\bigr) 
 + \frac{1}{2} \bigl( a^2 (b-a)+ b ( b^2 - a^2)\bigr) \\
 &=& \frac{1}{2} \bigl( a^2 + b^2\bigr) (b-a) + \frac{1}{2} \bigl( a  + b \bigr)^2 (b-a)\, .
 \end{eqnarray*}
 With its help, we obtain that
\begin{align}\nonumber
\bigl\langle{\tt I}^m, e^m\bigr\rangle_{L^2_x}  + \Vert e^m\Vert^2_{L^2_x} &= \Bigl\langle (y^m_1)^3 - (y^m_2)^3,e^m\Bigr\rangle_{L^2_x}  \\ \label{helpce}
&= \Bigl\langle \frac{1}{2}\bigl\vert y^m_1 + y^m_2\vert^2 +\frac{1}{2}\bigl[(y^m_1)^2 + (y^m_2)^2\bigr]  e^m,e^m\Bigr\rangle_{L^2_x} \,.
\end{align}
And, by binomial formula, and (\ref{helpce}),
\begin{align*}
\big|\langle{\tt II}^m, e^m\rangle_{L^2_x}\big| &=3\Big|\Bigl\langle e^m\bigl[y^m_1+y^m_2],e^m \Phi W(t_m)\Bigr\rangle_{L^2_x}\Big| \\&\leq 3 \Vert e^m \Phi W(t_m)\Vert^2_{L^2_x}  +
\frac{3}{4} \Vert e^m[y^m_1+y^m_2]\Vert^2_{L^2_x}  \\
&\leq  3 \Vert e^m \Phi W(t_m)\Vert^2_{L^2_x}  +
\frac{1}{2}\Vert e^m[y^m_1+y^m_2]\Vert^2_{L^2_x}  + \frac{1}{2} \Bigl( \Vert e^m y^m_1\Vert^2_{L^2_x}  + \Vert e^m y^m_2\Vert^2_{L^2_x} \Bigr)\\
&=\bigl\langle{\tt III}^m, e^m\bigr\rangle_{L^2_x} +\bigl\langle{\tt I}^m, e^m\bigr\rangle_{L^2_x} + \Vert e^m\Vert^2_{L^2_x}\,, 
\end{align*}
where we used that
\begin{align*}
\frac{1}{4} \Vert e^m[y^m_1+y^m_2]\Vert^2_{L^2_x} 
&\leq \frac{1}{2} \Bigl( \Vert e^m y^m_1\Vert^2_{L^2_x}  + \Vert e^m y^m_2\Vert^2_{L^2_x} \Bigr)
\end{align*}
by Young's inequality.
By now using these calculations for the error identity, as well as the index $m$, 
and taking into account that
\begin{align*}
\frac{1}{2}d_t\|e^m\|^2_{L^2_x}+\frac{\tau}{2}\|d_te^m\|_{L^2_x}^2=\frac{1}{\tau}\big\langle e^m-e^{m-1},e^m\big\rangle_{L^2_x}
\end{align*}
we find
\begin{equation}\label{pert1}
\frac{1}{2} d_t \Vert e^m\Vert^2_{L^2_x} + \frac{\tau}{2} \Vert d_t e^m\Vert^2_{L_x^2} + 
\Vert \nabla e^m\Vert^2_{L_x^2} \leq \Vert e^m\Vert^2_{L^2_x}\,,
\end{equation}
and the implicit version of Gronwall's lemma settles the estimate
\begin{equation}\label{pert1a}
\frac{1}{2} \max_{1 \leq m \leq M} {\mathbb E}\bigl[ \Vert e^m\Vert^2_{L^2_x}\bigr]
+ {\mathbb E}\Bigl[ \sum_{m=1}^M \tau \Vert \nabla e^m\Vert^2_{L^2_x}\Bigr] \leq C {\mathbb E}\bigl[\Vert y_{0,1} - y_{0,2}\Vert^2_{L^2_x} \bigr]\, .
\end{equation}

\subsection{Error analysis for (\ref{numR1})}\label{suse2} We integrate in (\ref{eq:SNS1}) over $[t_{m-1},t_m]$ to get
\begin{align}\nonumber
&y(t_m) - y({t_{m-1}}) + \int_{t_{m-1}}^{t_m} \big(-\Delta y  + f\bigl( y \bigr) +
3y^2 \Phi W + 3y [\Phi W]^2\big)\, {\rm d}s  \\ \label{weint1}
&\qquad = \int_{t_{m-1}}^{t_m}\big( \Delta [\Phi W] - f(\Phi W)\big)\, {\rm d}s\, 
\end{align}
and define
\begin{align*}
{\tt IV}^m(s)&:=-\Delta y(s)  + f\bigl( y(s) \bigr) +
3y^2(s) \Phi W(s) + 3y [\Phi W(s)]^2,\quad s\in[t_{m-1},t_m],\\
{\tt V}^{m}&:=\int_{t_{m-1}}^{t_m} \big(\Delta [\Phi W] - f(\Phi W)\big)\, {\rm d}s.
\end{align*}
In order to apply the parts from Subsection \ref{suse1}, we write  $${\tt IV}^m(s) = -\Delta y(t_m) + f\bigl(y(t_m)\bigr) +
3 \bigl( y(t_m)\bigr)^2 \Phi W(t_m) + 3 y(t_m) [\Phi W(t_m)]^2 + {\tt rest}_{{\tt IV}^m}(s)$$
for $s\in[t_{m-1},t_m]$,
where
\begin{align}\nonumber
{\tt rest}_{{\tt IV}^m}(s) &= \int_{s}^{t_m} \partial_t \Bigl( \Delta  y(\xi) - f\bigl( y(\xi)\bigr) \Bigr)\, {\rm d}\xi + 3\Bigl( \bigl(y(s) \bigr)^2 \Phi W(s) - \bigl(y(t_m) \bigr)^2 \Phi W(t_m)\Bigr) \\ \label{pert2}
&\quad 
 +3 \Bigl(y(s) [\Phi W(s)]^2 - y(t_m) [\Phi W(t_m)]^2\Bigr) =:\sum_{i=1}^3  {\tt rest}^{(i)}_{{\tt IV}^m}(s)\,.
\end{align}
For the error $E^m := y(t_m) - y^m$, on using ${\tt IV}^m(t_{m-1})$, we may now easily deduce an equation
of the form (\ref{erro1}), with terms similar to ${\tt I}^m, {\tt II}^m$, and ${\tt III}^m$,
and a non-vanishing right-hand side $\int_{t_{m-1}}^{t_m}  {\tt rest}_{{\tt IV}^m}(s)\, {\rm d}s$ instead. When tested with $E^m$, the arguments to estimate
the first three terms may then be estimated as in Subsection \ref{suse1}, leading to inequality (\ref{pert1}), such that we finally arrive at ${\mathbb P}$-a.s.:
\begin{eqnarray}\nonumber
&&\frac{1}{2} d_t \Vert E^m\Vert^2_{L^2_x} + \frac{\tau}{2} \Vert d_t E^m\Vert^2_{L_x^2} + 
\Vert \nabla E^m\Vert^2_{L_x^2} \\ \label{pert3}
&&\quad \leq 2\Vert E^m\Vert^2_{L^2_x} + 
\Bigl \vert \int_{t_{m-1}}^{t_m} \bigl\langle{\tt rest}_{{\tt IV}^m}(s),  E^m \bigr\rangle_{L^2_x}\, {\rm d}s
\Bigr\vert +   \frac{1}{2}\Vert {\tt V}^m\Vert^2_{L^2_x} \, .
\end{eqnarray}
Note that $ {\mathbb E}[\Vert {\tt V}^m\Vert^2_{L^2_x}] \leq \tau^2 t_m^3$; for the remaining term on the right-hand side, we treat the three terms $\{ {\tt rest}^{(i)}_{{\tt IV}^m}\}_{i=1}^3$ separately:
\begin{align}\label{pert4}
\begin{aligned} 
\quad {\mathbb E}\Bigl[ \Bigl\langle\int_s^{t_m}  \nabla \partial_t y(\xi)\, {\rm d}\xi, \nabla E^m\Bigr\rangle_{L^2_x}\Bigr] &\leq \frac{1}{2} {\mathbb E}\Bigl[\Bigl\|\int_s^{t_m}  \nabla \partial_t y(\xi)\, {\rm d}\xi\Bigr\|_{L^2_x}^2\Bigr]+ \frac{1}{2} {\mathbb E}\bigl[ \Vert \nabla E^m\Vert^2_{L^2_x}\bigr]\, \\
&\leq \frac{\tau}{2} {\mathbb E}\Bigl[  \int_{t_{m-1}}^{t_m} \Vert \nabla \partial_t y(\xi)\Vert^2_{L^2_x}\, {\rm d}\xi\Bigr]+ \frac{1}{2} {\mathbb E}\bigl[ \Vert \nabla E^m\Vert^2_{L^2_x}\bigr]\,,
\end{aligned}
 \end{align}
where the first term is bounded by $C \tau^2$, thanks to \eqref{eq:star}, and 
Corollary \ref{cor:add}. By Sobolev's embedding, we estimate for $\kappa>0$ arbitrary
\begin{align}
\nonumber
3 {\mathbb E}\Bigl[ &\Bigl\langle\int_{s}^{t_m} y^2(\xi) \partial_t y(\xi)\, {\rm d}\xi, E^m\Bigr\rangle_{L^2_x}\Bigr] \\\nonumber&\leq C\tau \, {\mathbb E}\Bigl[ \int_s^{t_m} \Vert y^2(\xi)\Vert_{L^6_x}^2 
\Vert \partial_t y(\xi)\Vert_{L^2_x}^2\, {\rm d}\xi\Bigr] + {\mathbb E}\bigl[ \Vert E^m\Vert^2_{L^3_x}\bigr]\, \\\nonumber
&\leq C\tau \, {\mathbb E}\Bigl[ \int_s^{t_m} {\tiny \Vert  y(\xi)\Vert^4_{W^{2,2}_x} \Vert \partial_t y(\xi)\Vert^2_{L^2_x}}\,{\rm d}\xi\Bigr] + {\mathbb E}\bigl[ \Vert E^m\Vert_{L^2_x}\Vert \nabla E^m\Vert_{L^2_x}\bigr]\\\nonumber
&\leq C\tau^2 \, {\mathbb E}\Bigl[\sup_{t_{m-1}\leq \xi\leq t_m}\Vert  y(\xi)\Vert^6_{W^{2,2}_x} +\sup_{t_{m-1}\leq \xi\leq t_m}\Vert \partial_t y(\xi)\Vert^6_{L^2_x}\Bigr]\\&\quad + C(\kappa){\mathbb E}\bigl[ \Vert E^m\Vert_{L^2_x}^2\bigr]+\kappa  {\mathbb E}\bigl[ \Vert \nabla E^m\Vert^2_{L^2_x}\bigr],
\label{pert5} 
\end{align}
where we also used interpolation of $L^3_x$ between $L^2_x$ and $W^{1,2}_x$. Now we can absorb the last term into the left-hand side choosing $\kappa$ sufficiently small. By (\ref{weint1}) and 
Corollary \ref{cor:add} we can estimate the first term by $C \tau^2$. The missing term $ {\mathbb E}\bigl[ \bigl\langle\int_{s}^{t_m}  \partial_t y(\xi)\, {\rm d}\xi, E^m\bigr\rangle_{L^2_x}\bigr]$ in  ${\tt rest}^{(1)}_{{\tt IV}^m}(s)$ can be estimated via an analogous (even easier) chain of inequalities.
We resume with 
\begin{align*}
{\tt rest}^{(3)}_{{\tt IV}^m}(s) &=  \bigl(y(s)  \bigr)^2 \Phi \bigl[W(s)-W(t_m)\bigr] +  \bigl[  (y(s) )^2 -  (y(t_m) )^2 \bigr]\Phi W(t_m) \\
&=  \bigl(y(s)  \bigr)^2 \Phi \bigl[W(s)-W(t_m)\bigr] +   [y(s) - y(t_m)][y(s) + y(t_m)] \Phi W(t_m) \\
&=: {\tt rest}^{(3,a)}_{{\tt IV}^m}(s) + {\tt rest}^{(3,b)}_{{\tt IV}^m}(s)\,,
\end{align*} 
and use ${\mathbb E}[\Vert \Phi (W(s) - W(t_m))\Vert^2_{L^\infty_x}\bigl\vert {\mathcal F}_s] \leq C\tau$, which is a consequence of assumption \eqref{eq:phi}, to further estimate
\begin{align*}
{\mathbb E}\Bigl[\bigl\langle {\tt rest}^{(3,a)}_{{\tt IV}^m}(s), E^m\bigr\rangle_{L^2_x}\Bigr] &=
{\mathbb E}\Bigl[\bigl\langle {\tt rest}^{(3,a)}_{{\tt IV}^m}(s), \tau d_t E^m\bigr\rangle_{L^2_x}\Bigr] \\
&\leq \frac{\tau}{4} {\mathbb E}\bigl[ \Vert d_t E^m\Vert^2_{L^2_x}\bigr]
+ \tau {\mathbb E}\bigl[ \Vert {\tt rest}^{(3,a)}_{{\tt IV}^m}(s)\Vert^2_{L^2_x}\bigr]\\
&\leq \frac{\tau}{4} {\mathbb E}\bigl[ \Vert d_t E^m\Vert^2_{L^2_x}\bigr]
+ C\tau^2  {\mathbb E}[\Vert y(s)\Vert^4_{L^4_x}]\,.
\end{align*}
Also, since $\big({\mathbb E}[\Vert y(s) - y(t_m)\Vert_{L^2_x}^2]^3\big)^{1/3} \leq C\tau^2$ which easily follows from (\ref{weint1}) and 
Corollary \ref{cor:add},
\begin{align*}
{\mathbb E}\Bigl[\bigl\langle {\tt rest}^{(3,b)}_{{\tt IV}^m}(s), E^m\bigr\rangle_{L^2_x}\Bigr] &\leq
{\mathbb E}\bigl[ \Vert E^m\Vert^2_{L^2_x}\bigr] + C\tau^2  \bigg({\mathbb E}\biggl[\sup_{0\leq t\leq T} \Vert y\Vert_{L^\infty_x}^6\biggr]\bigg)^{\frac{1}{3}}\bigg({\mathbb E}\biggl[\sup_{0\leq t\leq T} \Vert \Phi W\Vert_{L^\infty_x}^6\biggr]\bigg)^{\frac{1}{3}}\\
 &\leq
{\mathbb E}\bigl[ \Vert E^m\Vert^2_{L^2_x}\bigr] + C\tau^2  \bigg({\mathbb E}\biggl[\sup_{0\leq t\leq T} \Vert y\Vert_{W^{2,2}_x}^6\biggr]\bigg)^{\frac{1}{3}}\\
 &\leq
{\mathbb E}\bigl[ \Vert E^m\Vert^2_{L^2_x}\bigr] + C\tau^2.
\end{align*}
The estimate for ${\tt rest}^{(2)}_{{\tt IV}^m}(s)$ is analogous to that for ${\tt rest}^{(3)}_{{\tt IV}^m}(s)$ such that we conclude
\begin{eqnarray*}
&&{\mathbb E}\bigl[ d_t \Vert E^m\Vert^2_{L^2_x} + \tau \Vert d_t E^m\Vert^2_{L_x^2} + 
\Vert \nabla E^m\Vert^2_{L_x^2}\bigr] \\ \label{pert3}
&&\quad \leq C{\mathbb E}\bigl[\Vert E^m\Vert^2_{L^2_x}\bigr] + 
C\tau^2 +   C{\mathbb E}\bigl[\Vert {\tt V}^m\Vert^2_{L^2_x}\bigr] \, .
\end{eqnarray*}
Estimating the final term by $C\tau^2$ and applying the discrete Gronwall lemma,
the above considerations lead to the following theorem.
\begin{theorem}\label{thm:3.1tilde}
 Let $(\Omega,\mf,(\mf_t)_{t\geq0},\prst)$ be a given stochastic basis with a complete right-con\-ti\-nuous filtration and an $(\mf_t)$-cylindrical Wiener process $W$. Let $T>0$ be fixed.
Assume that $\bfu_0\in L^p(\Omega,W^{2,2}(\mt))\cap L^2(\Omega,W^{3,2}(\mt))$ for some $p\geq8$ and $\Phi\in L_2(\mathfrak U;W^{3,2}(\mt))$. Let $u$ be the solution
to \eqref{eq:SNS}, and $(u^m)_{m=1}^M$ be the solution to (\ref{tdiscrAintro}).
Then we have the error estimate
\begin{align}\label{eq:thm:4}
\begin{aligned}
\max_{1\leq m\leq M}\E\bigg[\|u({{t}}_m)-u^{m}\|^2_{L^2_x}&+\sum_{n=1}^m \tau \|\nabla u({t}_n)-\nabla u^{n}\|^2_{L^2_x}\bigg]\leq \,C\,\tau^{2}.
\end{aligned}
\end{align}
\end{theorem}
By its proof, this result also applies to $(y^m)_{m=1}^M$ from \eqref{numR1}, and
 $y$ from \eqref{eq:SNS1}.
\begin{remark}\label{others1}
{\bf 1.}~Problem (\ref{eq:SNS}) usually involves a small-scale parameter $\varepsilon>0$ to address the width of diffuse interfaces of
adjacent material phases, whose resolution in terms of numerical scalings is crucial 
for accurate simulation; see \cite{AnBaNuPr}. In this work, we choose $\varepsilon = 1$ to address non-Lipschitzness of $f$ only, but expect a corresponding analysis as in \cite{AnBaNuPr} to go through for $\varepsilon \ll 1$ --- avoiding a factor $\exp(\frac{T}{\varepsilon})$ that otherwise would occur in a straight-forward application of Gronwall's lemma.

{\bf 2.}~The proof of a strong rate ${\mathcal O}(\sqrt{\tau})$ for an implicit time discretisation --- which slightly differs from (\ref{ac-2d}) --- in the case of (\ref{eq:SNS}) with multiplicative noise in \cite{MP1} exploits its character as a structure preserving discretisation,
and therefore inherits the gradient structure of the problem, and related energy estimates. The crucial step in the error analysis then uses the weak monotonicity
property of the cubic nonlinearity (see \cite[(4.11)]{MP1}) to {\em avoid a truncation argument} for the nonlinearity in the sense of \cite{Pr1} for the stochastic Navier-Stokes equation; see also \cite{BP1}.

{\bf 3.}~The construction and analysis of numerical schemes in \cite{MP1,AnBaNuPr} is based on the {\em strong variational solution} concept for (\ref{eq:SNS}) --- which is in contrast to other numerical works in the literature where the linear semi-group ${\mathscr S} := \{ {\mathscr S}(t);\, t \geq 0\}$, for ${\mathscr S}(t) = e^{t\Delta}$ is used as building block to set up the
{\em mild solution} concept for (\ref{eq:SNS}) {\em with additive noise} for $d=1$ and $\varepsilon =1$,
\begin{equation}\label{mildsol1}
u(t) = {\mathscr S}(t) u_0 + \int_0^t {\mathscr S}(t-s) f\bigl( u(s)\bigr)\, {\rm d}s +
\int_0^t {\mathscr S}(t-s) \, {\rm d}W(s) \qquad\forall\, t \geq 0\, .
\end{equation}
In this setting, the
authors of \cite{BCH1,BG1} prove strong rates of convergence for a {\em splitting scheme}, even addressing space-time white noise. Splitting schemes have a long tradition in evolutionary problems to avoid solving nonlinear PDEs for iterates, but cause structure violation of (\ref{eq:SNS}); in fact, the underlying stability bounds in \cite[Prop.~3]{BG1} and \cite[Lemma 3.1]{BCH1} are {\em not} in the natural {\em energy norm} of the underlying problem (\ref{eq:SNS}) with gradient structure, and also require a {\em Gronwall-type estimation} --- which reflects usually needed smaller mesh sizes in `splitting-scheme based' simulations at intermediate or large times $T \gg 1$ to keep accuracy. In comparison, the {\color{red} (energy)} structure-preserving schemes (\ref{ac-2d}) and (\ref{numR1}) are nonlinear, but the larger 
computational effort caused by Newton-type fast nonlinear solvers here usually goes along with admissible larger mesh sizes to attain the same accuracy.

{\bf 4.}~In \cite{BCH1},  the strong  rate ${\mathcal O}(\tau)$ is shown
for the scheme \cite[(1.2)]{BCH1} in \cite[Thm.~4.6]{BCH1}. A crucial step in their proof is also to employ a transformation, see \cite[(2.4) and (2.7)]{BCH1}, which is similar conceptionally to (\ref{transf1})  --- however in \cite{BCH1} on the level of mild solutions, using (\ref{mildsol1}); in fact, the (corresponding) identity 
\begin{equation}\label{mildsol3} u(t) = Y(t) + \widetilde{W}^{\Delta}(t)
\end{equation} is used instead of SPDE (\ref{mildsol1}). Eventually,  only the {\em random} PDE
\begin{equation}\label{mildsol2}
Y(t) = {\mathscr S}(t) u_0 + \int_0^t {\mathscr S}(t-s) f \bigl( Y(t) + \widetilde{W}^{\Delta}(t)\bigr)\, {\rm d}s \qquad \forall\, t \geq 0
\end{equation}
is studied in the main part of the analysis in \cite{BCH1}. Here the last term is  the `stochastic convolution process
  $\widetilde{W}^{\Delta} := \{\widetilde{W}^{\Delta}(t);\, t \geq 0\}$'. 
 From a practical point of view, however, the authors in \cite{BCH1} clearly mention the  restricted applicability of their scheme,  which requires the {\em explicit} knowledge of ${\mathscr S}$, and (an approximation) of $\widetilde{W}^{\Delta}$: as such, its efficient use requires to know its spectrum --- which restricts its {\em practical application to prototypic} domains ${\mathcal O} \subset {\mathbb R}^d$ where eigenvalues of $\Delta$ are explicitly known.
On the other hand, the approach in \cite{BCH1} 
allows less regular noise compared to here.

Corresponding restrictions hold for schemes that are studied in \cite{LQ1, BJ, BGJK,CHS}, where again the construction rests on (\ref{mildsol2}) and (\ref{mildsol3}), and involves (the approximation of) $\widetilde{W}^{\Delta}$.

{\bf 5.}~To our knowledge, optimal strong convergence ${\mathcal O}(\tau)$ for 
a time discretization to solve (\ref{eq:SNS}) with additive trace-class noise was first established in \cite{QW1}\footnote{The focus in \cite{QW1} is, however, on a finite-element based space-time discretisation.}: here, again, the authors start with a {\rm mild solution} for (\ref{eq:SNS}), and base their error analysis of  (\ref{ac-2d}) on its reformulation \cite[(4.2)]{QW1}, and the 
known (smoothing) properties of  the semigroup ${\mathscr S}_t = e^{t \Delta}$.
The error analysis provided in this section rather exploits variational arguments: it bases on reformulation (\ref{numR1}), as well as simple, explicit calculations for the specific $f$ in Subsection \ref{suse1}, and may easily be generalized
to weakly elliptic ({\em e.g.}~non-selfadjoint) operators which do not generate a semi-group; see {\em e.g.}~\cite[Section 2.5]{DuFe} for the {\em convected or degenerate Allen-Cahn equation} as prominent examples in multiphase fluid flow models, or the {\em anisotropic} Allen-Cahn equation in \cite[Section 8]{DeDzEl}.
\end{remark}

\section{Preparations for the weak error analysis}\label{sec:prep}

\subsection{The nonlinear semigroup}
\label{sec:Ttau}
In this section we study properties of the discrete nonlinear semigroup
$\mathcal T_\tau$ on $L^2(\mt)$, which is the solution operator to the equation
\begin{align}\label{eq:T}
\bfv+\tau\mathcal A\bfv+\tau f(\bfv)=\bfg,
\end{align}
where $\bfg\in L^2(\mt)$ is given.
We start with some stability estimates. We can write \eqref{eq:T} equivalently as
\begin{align}\label{eq:T'}
\bfv+\tau D\mathcal E(v)=\bfg.
\end{align}
Due to the choice of the continuous interpolation $u_\tau$ in \eqref{eq:utau} below,
the weak error analysis in Section \ref{sec:error}
heavily depends on stability estimates for $\mathcal T_\tau$, which we derive in the following lemma. Eventually, we estimate the distance of $\mathcal T_\tau$ to the identity and consider its Fr\'echet-derivatives. They are used for the same purpose due to the representation for $u_\tau$ in \eqref{tdiscr'}.  
\begin{lemma}\label{lem:T}
Suppose that $\tau<\frac{1}{2}$ and $p\geq 2$. Then we have
\begin{align}\label{est:TL2}
\|\mathcal T_{\tau}\bfg\|_{L^p_x}^p+\tau\|\nabla|\mathcal T_{\tau}\bfg|^{p/2}\|^2_{L^2_x}&\leq\,c\|\bfg\|_{L^p_x}^p,\\
\label{est:TW12}
\|\nabla\mathcal T_{\tau}\bfg\|^p_{L^p_x}+\tau\|\nabla|\nabla\mathcal T_{\tau}\bfg|^{p/2}\|^2_{L^2_x}&\leq\,c\|\bfg\|_{W^{1,p}_x}^p,\\
\label{est:TW22}
\|\nabla^2\mathcal T_{\tau}\bfg\|^p_{L^p_x}+\tau\|\nabla|\nabla^2\mathcal T_{\tau}\bfg|^{p/2}\|_{L^2_x}^2&\leq\,c\big(\|\bfg\|_{W^{2,p}_x}^p+\|\bfg\|_{W^{1,3p}_x}^{3p}\big),
\end{align}
for all functions $\bfg$
for which the quantities on the right-hand side are finite.
Here $c=c(p)>0$ is independent of $\tau$ and $\bfg$.
\end{lemma}
\begin{proof}
Ad~{\bf \eqref{est:TL2}}. Testing \eqref{eq:T} by $|\bfv|^{p-2}\bfv$
shows
\begin{align*}
\|\bfv\|_{L^p_x}^p+\tau\,c(p)\|\nabla|\bfv|^{p/2}\|_{L^2_x}^2+\tau\int_{\mt}|\bfv|^{p-2}f(v)\,v\dx&=\int_{\mt}\bfg\cdot|v|^{p-2}\bfv\dx
\end{align*}
with $c(p)=\frac{4(p-1)}{p^2}$,
where we used 
\begin{align}\label{eq:1909}
|\nabla|\bfv|^{p/2}|^2=\Big|\frac{p}{2}|\bfv|^{(p-2)/2}\frac{v}{|v|}\nabla v\Big|^2=\Big|\frac{p}{2}|\bfv|^{(p-2)/2}\nabla v\Big|^2=\frac{p^2}{4(p-1)}\nabla\big(|v|^{p-2}v\big)\cdot\nabla v.
\end{align}
The term on the right-hand side can be handled by Young's inequality, while we have
\begin{align*}
\|\bfv\|_{L^p_x}^p+\tau\int_{\mt}|\bfv|^{p-2}f(v)v\dx\geq (1-\tau)\|\bfv\|_{L^p_x}^p.
\end{align*}
This proves \eqref{est:TL2} for $\tau<\frac{1}{2}$.

Ad~{\bf \eqref{est:TW12}}. Similarly, by applying $\partial_i$ to \eqref{eq:T}, multiplying with $|\partial_i\bfv|^{p-2}\partial_i\bfv$ (and arguing as in \eqref{eq:1909}) integrating in space\footnote{This step can be made rigorous be working with difference quotients.} and summing with respect to $i\in\{1,2,3\}$ yields
\begin{align*}
\|\nabla\bfv\|_{L^p_x}^p+\tau\,c(p)\|\nabla|\nabla\bfv|^{p/2}\|_{L^2_x}^2+\tau\sum_{i=1}^3\int_{\mt}|\partial_i\bfv|^pf'(v)\dx&=\int_{\mt}\partial_i\bfg\,|\partial_i\bfv|^{p-2}\partial_i\bfv\dx.
\end{align*}
The two last terms can be handled analogously to the proof of
\eqref{est:TL2},
such that \eqref{est:TW12} follows.

Ad~{\bf \eqref{est:TW22}}. Now we apply $\partial_i\partial_j$ to the equation and multiply with
$|\partial_i\partial_j\bfv|^{p-2}\partial_i\partial_j\bfv$. It remains to control the nonlinear term as the rest can be handled analogously to the estimates above. We obtain on the right-hand side
\begin{align*}
-\tau f'(v)|\nabla^2\bfv|^p-\tau f''(v)\nabla\bfv\otimes\nabla\bfv:\nabla^2\bfv|\nabla^{2}\bfv|^{p-2}&\leq\,\tau(1+|\bfv|)|\nabla\bfv|^2|\nabla^{2}\bfv|^{p-1}+\tau|\nabla^{2}\bfv|^{p}\\
&\leq\,\tau|\nabla^{2}\bfv|^{p}+\tau|\bfv|^{3p}+\tau|\nabla\bfv|^{3p}.
\end{align*}
After integration over $\mt$ the first term can be absorbed, while the other two can be controlled using \eqref{est:TL2} and \eqref{est:TW12}.

%
\end{proof}

On using $\mathcal T_\tau:L^2(\mt)\rightarrow W^{2,2}(\mt)$, \emph{cf.}~equation \eqref{eq:T}, we can write
\begin{align}\label{eq:0107}
\mathcal T_{\tau}\bfg=\mathcal S_\tau\bfg-\tau\mathcal S_\tau \big(f(\mathcal T_\tau\bfg)\big).
\end{align}

Using equation \eqref{eq:0107} and estimates \eqref{eq:S1}--\eqref{eq:S3} one can derive the following from Lemma \ref{lem:T}.
\begin{corollary}
Under the assumptions of Lemma \ref{lem:T} we have
\begin{align}
\label{est:TS}
\|\big(\mathcal S_\tau-\mathcal T_{\tau}\big)\bfg\|_{L^2_x}
&\leq\,c\tau\big(\|\bfg\|_{L^{2}_x}+\|\bfg\|_{L^{6}_x}^3\big),\\
\label{est:TShigher}
\|\big(\mathcal S_\tau-\mathcal T_{\tau}\big)\bfg\|_{W^{1,2}_x}
&\leq\,c\tau\big(1+\|\bfg\|^3_{W^{2,2}_x}\big),\\
\|\big(\mathcal S_\tau-\mathcal T_{\tau}\big)\bfg\|_{W^{2,2}_x}
&\leq\,c\tau(1+\|\bfg\|_{W^{2,2}_x}^3\big)\label{est:TShigher2}
\end{align}
where $c>0$ is independent of $\tau$.
\end{corollary}
Writing now $\mathrm{Id}-\mathcal T_{\tau}=\mathrm{Id}-\SSS_\tau+\SSS_\tau-\mathcal T_{\tau}$,
and combining \eqref{est:TS}--\eqref{est:TShigher2} with \eqref{eq:S3} we obtain the following result. 
\begin{corollary}\label{coro-1}
Under the assumptions of Lemma \ref{lem:T} we have
\begin{align}
\label{est:Tid}
\|\big(\mathcal T_{\tau}-\id\big)\bfg\|_{L^2_x}
&\leq\,c\tau\big(1+\|\bfg\|_{W^{2,2}_x}+\|\bfg\|_{L^{6}_x}^3\big),\\
\label{est:Tidhigher}
\|\big(\mathcal T_{\tau}-\id\big)\bfg\|_{W^{1,2}_x}
&\leq\,c\tau\big(1+\|\bfg\|_{W^{3,2}_x}+\|\bfg\|^3_{W^{2,2}_x}\big),\\
\|\big(\mathcal T_{\tau}-\id\big)\bfg\|_{W^{2,2}_x}
&\leq\,c\tau(1+\|\bfg\|_{W^{4,2}_x}+\|\bfg\|_{W^{2,2}_x}^3\big),\label{est:Tidhigher2}
\end{align}
where $c>0$ is independent of $\tau$.
\end{corollary}
Now we turn to the Fr\'echet derivative of $\mathcal T_\tau$ and prove estimates on $L^p(\mt)$, $W^{1,2}(\mt)$ and $W^{2,2}(\mt)$, respectively. 
\begin{lemma}\label{lem:1909}
Let $\tau<\frac{1}{2}$ and $p\geq 2$.
\begin{enumerate}
\item For all $\bfg\in L^p(\mt)$ we have
\begin{align}\label{eq:2704}
\|D\mathcal T_\tau(\bfg)\|_{\mathcal L(L^p_x)}\leq 1
\end{align}

\item For all $\bfg\in L^6(\mt)$ we have
\begin{align}\label{eq:2704B}
\|D\mathcal T_\tau(\bfg)\|_{\mathcal L(W^{1,2}_x)}\leq \,c\Big(1+c\tau\|\bfg\|_{L^{6}_x}^{4}\Big).
\end{align}
\item For all $\bfg\in W^{2,2}(\mt)$.We have
\begin{align}\label{eq:2704C}
\|D\mathcal T_\tau(\bfg)\|_{\mathcal L(W^{2,2}_x)}\leq \,c\Big(1+\tau\|\bfg\|_{W^{2,2}_x}^{8}\Big).
\end{align}
\end{enumerate}
\end{lemma}
\begin{proof}
Ad~{\bf (a)} By inverse function theorem\footnote{Note that $D\mathcal T^{-1}_\tau= \mathrm{id}+\tau\A+\tau f'(\cdot)$ is clearly invertible.} for Fr\'echet derivatives we have
\begin{align}\label{???}
\begin{aligned}
D\mathcal T_\tau(\bfg)&=\big(D\mathcal T^{-1}_\tau(\mathcal T_\tau(\bfg))\big) ^{-1}\\
&=\big(\id+\tau\A+\tau f'(\mathcal T_\tau(\bfg))\big)^{-1}.
\end{aligned}
\end{align}
The operator $D\mathcal T^{-1}_\tau(\bfg)$ is coercive on $L^2_x$ as
\begin{align*}
\|D\mathcal T^{-1}_\tau(\bfg)\|_{\mathcal L(L^2_x)}&=\sup_{\|\bfh\|_{ L^2_x}=1}\|(\id+\tau\A+\tau f'(\bfg))\bfh\|_{L^2_x}\\
&\geq \sup_{\|\bfh\|_{L^2_x}=1}\big\langle(\id+\tau\A+\tau f'(\bfg)\bfh,\bfh\big\rangle_{L^2_x}\geq1
\end{align*} 
independently of $\bfg$. A similar argument applies in $L^p(\mt)$ for $p>2$ since
\begin{align*}
\|D\mathcal T^{-1}_\tau(\bfg)\|_{\mathcal L(L^p_x)}&=\sup_{\|\bfh\|_{ L^p_x}=1}\|(\id+\tau\A+\tau f'(\bfg)\bfh\|_{L^p_x}\\
&\geq \sup_{\|\bfh\|_{L^p_x}=1}\big\langle(\id+\tau\A+\tau f'(\bfg)\bfh,|\bfh|^{p-2}\bfh\big\rangle_{L^2_x}\geq1.
\end{align*}
Hence
the claim follows.

Ad~{\bf (b)} We obtain
\begin{align*}
\|D\mathcal T^{-1}_\tau(\bfg)\bfh\|_{W^{1,2}_x}&=\sup_{\tilde\bfh\in W^{1,2}_x}\frac{\langle(\id+\tau\A+\tau f'(\bfg))\bfh,\tilde\bfh\rangle_{W^{1,2}_x}}{\|\tilde\bfh\|_{W^{1,2}_x}}\\
&\geq \frac{\big\langle(\id+\tau\A+\tau f'(\bfg)\bfh,\bfh\big\rangle_{W^{1,2}_x}}{\|\bfh\|_{W^{1,2}_x}}\\
&=\|\bfh\|_{W^{1,2}_x}+\tau\frac{ \|\bfh\|_{W^{2,2}_x}^2}{\|\bfh\|_{W^{1,2}_x}}+\tau\frac{\big\langle \nabla(f'(\bfg)\bfh),\nabla\bfh\big\rangle_{L^{2}_x}}{\|\bfh\|_{W^{1,2}_x}}.
\end{align*} 
The last term does not have an obvious sign and needs to be estimated. We have
\begin{align*}
\tau\frac{\big\langle \nabla(f'(\bfg)\bfh),\nabla\bfh\big\rangle_{L^{2}_x}}{\|\bfh\|_{W^{1,2}_x}}&=-\tau\frac{\big\langle f'(\bfg)\bfh,\Delta\bfh\big\rangle_{L^{2}_x}}{\|\bfh\|_{W^{1,2}_x}}\\
&\geq -\frac{\tau}{2}\frac{\|\bfh\|_{W^{2,2}_x}^2}{\|\bfh\|_{W^{1,2}_x}}-\frac{\tau}{2}\frac{\|f'(\bfg)\bfh\|_{L^{2}_x}^2}{\|\bfh\|_{W^{1,2}_x}}\\
&\geq -\frac{\tau}{2}\frac{\|\bfh\|_{W^{2,2}_x}^2}{\|\bfh\|_{W^{1,2}_x}}-\frac{\tau}{2}\frac{\|f'(\bfg)\|_{L^{3}_x}^2\|\bfh\|_{L^6_x}^2}{\|\bfh\|_{W^{1,2}_x}}\\
&\geq -\frac{\tau}{2}\frac{\|\bfh\|_{W^{2,2}_x}^2}{\|\bfh\|_{W^{1,2}_x}}-\frac{c\tau}{2}(\|\bfg\|_{L^{6}_x}^{4}+1)\|\bfh\|_{L^{6}_x}
\end{align*}
such that we conclude
\begin{align*}
\|D\mathcal T^{-1}_\tau(\bfg)\bfh\|_{W^{1,2}_x}+c\tau\Big(\|\bfg\|_{L^{6}_x}^4+1\Big)\|\bfh\|_{L^{6}_x}&\geq \|\bfh\|_{W^{1,2}_x}+\frac{\tau}{2}\frac{\|\bfh\|_{W^{2,2}_x}^2}{\|\bfh\|_{W^{1,2}_x}}\\
&\geq\|\bfh\|_{W^{1,2}_x}+\frac{\tau}{2} \|\bfh\|_{W^{2,2}_x}
\end{align*}
Replacing $\bfh$ by $D\mathcal T_\tau(\bfg)\bfh$ (recall that $\mathcal T_\tau:L^2(\mt)\rightarrow W^{2,2}(\mt)$, \emph{cf.} equation \ref{eq:T}) and using \eqref{eq:2704} shows
\begin{align*}
 \|D\mathcal T_\tau(\bfg)\bfh\|_{W^{1,2}_x}+\tau \|D\mathcal T_\tau(\bfg)\bfh\|_{W^{2,2}_x}&\leq c\|\bfh\|_{W^{1,2}_x}+c\tau\Big(\|\bfg\|_{L^{6}_x}^{4}+1\Big)\|D\mathcal T_\tau(\bfg)\bfh\|_{L^{6}_x}\\
 &\leq c\|\bfh\|_{W^{1,2}_x}+c\tau\Big(\|\bfg\|_{L^{3}_x}^{4}+1\Big)\|\bfh\|_{L^{6}_x}\\
        &\leq c\|\bfh\|_{W^{1,2}_x}\Big(1+c\tau\|\bfg\|_{L^{6}_x}^{4}\Big)
\end{align*}
and the claim follows.

Ad~{\bf (c)} We have
\begin{align}\nonumber
\tau\frac{\big\langle \nabla^2(f'(\bfg)\bfh),\nabla^2\bfh\big\rangle_{L^{2}_x}}{\|\bfh\|_{W^{2,2}_x}}&=-\tau\frac{\big\langle \nabla(f'(\bfg)\bfh),\Delta\nabla\bfh\big\rangle_{L^{2}_x}}{\|\bfh\|_{W^{2,2}_x}}\\
&=-\tau\frac{\big\langle f'(\bfg)\nabla\bfh,\Delta\nabla\bfh\big\rangle_{L^{2}_x}}{\|\bfh\|_{W^{2,2}_x}}-\tau\frac{\big\langle f''(\bfg)\nabla\bfg\bfh,\Delta\nabla\bfh\big\rangle_{L^{2}_x}}{\|\bfh\|_{W^{2,2}_x}},\label{A}
\end{align}
where the first term can be estimated by
\begin{align*}
&\geq -\frac{\tau}{4}\frac{\|\bfh\|_{W^{3,2}_x}^2}{\|\bfh\|_{W^{2,2}_x}}-\frac{\tau}{2}\frac{\|f'(\bfg)\nabla\bfh\|_{L^{2}_x}^2}{\|\bfh\|_{W^{2,2}_x}}\\
&\geq -\frac{\tau}{4}\frac{\|\bfh\|_{W^{3,2}_x}^2}{\|\bfh\|_{W^{2,2}_x}}-\frac{\tau}{2}\frac{\|f'(\bfg)\|_{L^{4}_x}^2\|\nabla\bfh\|_{L^4_x}^2}{\|\bfh\|_{W^{2,2}_x}}\\
&\geq -\frac{\tau}{4}\frac{\|\bfh\|_{W^{3,2}_x}^2}{\|\bfh\|_{W^{2,2}_x}}-c\tau\frac{\Big(\|\bfg\|_{L^{8}_x}^{4}+1\Big)\|\nabla\bfh\|_{L^2_x}^{1/2}\|\nabla\bfh\|_{L^6_x}^{3/2}}{\|\bfh\|_{W^{2,2}_x}},
\end{align*}
by interpolation and Sobolev's embedding. We further conclude
\begin{align*}
&\geq -\frac{\tau}{4}\frac{\|\bfh\|_{W^{3,2}_x}^2}{\|\bfh\|_{W^{2,2}_x}}-c\tau\Big(\|\bfg\|_{L^{8}_x}^{4}+1\Big)\|\nabla\bfh\|_{L^2_x}^{1/2}\|\nabla\bfh\|_{L^6_x}^{1/2}\\
&\geq -\frac{\tau}{2}\frac{\|\bfh\|_{W^{3,2}_x}^2}{\|\bfh\|_{W^{2,2}_x}}-c\tau\Big(\|\bfg\|_{L^{8}_x}^{4}+1\Big)\|\bfh\|_{W^{2,2}_x},
\end{align*}
and the second term in \eqref{A} has the lower bound
\begin{align*}
&\geq -\frac{\tau}{4}\frac{\|\bfh\|_{W^{3,2}_x}^2}{\|\bfh\|_{W^{2,2}_x}}-c\tau\frac{\|f''(\bfg)\nabla\bfg\bfh\|_{L^{2}_x}^2}{\|\bfh\|_{W^{2,2}_x}}\\
&\geq -\frac{\tau}{4}\frac{\|\bfh\|_{W^{3,2}_x}^2}{\|\bfh\|_{W^{2,2}_x}}-c\tau\frac{\|f''(\bfg)\nabla\bfg\|_{L^{3}_x}^2\|\bfh\|_{L^6_x}^2}{\|\bfh\|_{W^{2,2}_x}}\\
&\geq -\frac{\tau}{4}\frac{\|\bfh\|_{W^{3,2}_x}^2}{\|\bfh\|_{W^{2,2}_x}}-c\tau\|f''(\bfg)\nabla\bfg\|_{L^{3}_x}^2\|\bfh\|_{W^{1,2}_x}.
\end{align*}
We conclude as in the proof of \eqref{eq:2704B}
\begin{align*}
c\big(\tau\|\bfg\|_{L^{8}_x}^{4}+1\big)\|D\mathcal T^{-1}_\tau(\bfg)\bfh\|_{W^{2,2}_x}+\|f''(\bfg)\nabla\bfg\|_{L^{3}_x}^2\|\bfh\|_{W^{1,2}_x}\geq \|\bfh\|_{W^{2,2}_x}+\frac{\tau}{2}\frac{\|\bfh\|_{W^{3,2}_x}^2}{\|\bfh\|_{W^{2,2}_x}}\\
\geq \|\bfh\|_{W^{2,2}_x}+\frac{\tau}{2} \|\bfh\|_{W^{3,2}_x}.
\end{align*}
Replacing $\bfh$ by $D\mathcal T_\tau(\bfg)\bfh$ and using \eqref{eq:2704B} shows
\begin{align*}
 \|D\mathcal T_\tau(\bfg)\bfh\|_{W^{2,2}_x}+&\tau \|D\mathcal T_\tau(\bfg)\bfh\|_{W^{3,2}_x}\\&\leq c\big(\tau\|\bfg\|_{L^{8}_x}^{4}+1\big)\|\bfh\|_{W^{2,2}_x}+c\|f''(\bfg)\nabla\bfg\|_{L^{3}_x}^2\|D\mathcal T_\tau(\bfg)\bfh\|_{W^{1,2}_x}\\
 &\leq c\Big(\tau\|\bfg\|_{L^{8}_x}^{4}+1\Big)\|\bfh\|_{W^{2,2}_x}+c\tau\|f''(\bfg)\|_{L^6_x}^2\|\nabla\bfg\|_{L^{6}_x}^2(1+\|\bfg\|_{L^{6}_x}^{4})\|\bfh\|_{W^{1,2}_x}\\
&\leq c\|\bfh\|_{W^{2,2}_x}\Big(1+\tau \|\bfg\|_{W^{2,2}_x}^{8}\Big),
\end{align*}
which yields the claim.
\end{proof}
Finally, we estimate the distance between $D\mathcal T_\tau$ and the identity.
\begin{lemma}
Let $\tau<\frac{1}{2}$.
\begin{enumerate}
\item For all $\bfh\in W^{2,2}(\mt)$ and $\bfg\in L^6(\mt)$ we have \begin{align}\label{eq:T-id}
\|(D\mathcal T_\tau(\bfg)-\id)\bfh\|_{L^2_x}
&\leq\,c\tau\|\bfh\|_{W^{2,2}_x}\Big(1+\|\bfg\|_{L^{6}_x}^{2}\Big).
\end{align}
\item For all $\bfh\in W^{4,2}(\mt)$ and $\bfg\in W^{2,2}(\mt)$ we have
\begin{align}\label{eq:T-idW22}
\begin{aligned}
\|(D\mathcal T_\tau(\bfg)-\id)\bfh\|_{W^{2,2}_x}&\leq\,c\tau\Big(1+\|\bfg\|_{W^{2,2}_x}^{8}\Big)\|\bfh\|_{W^{4,2}_x}.
\end{aligned}
\end{align}
\item For all $\bfg\in L^6(\mt)$ we have \begin{align}\label{eq:D2T}
\|D^2\mathcal T_\tau(\bfg)\|_{\mathcal L(W^{1,2}_x\times W^{1,2}_x;L^2_x)}\leq\,c\tau\Big(1+\|\bfg\|_{L^{6}_x}\Big).
\end{align}
\end{enumerate}
\end{lemma}
\begin{proof}
 We rewrite formula \eqref{???} as
\begin{align*}
D\mathcal T_\tau(\bfg)&=\id+\big(D\mathcal T^{-1}_\tau(\mathcal T_\tau(\bfg))\big) ^{-1}-\id\\
&=\id+\tau\big(\id+\tau\A+\tau f'(\mathcal T_\tau(\bfg))\big)^{-1}(\A+f'(\mathcal T_\tau(\bfg))\\
&=\id+\tau D\mathcal T_\tau(\bfg)(\A+f'(\mathcal T_\tau(\bfg))).
\end{align*}
Now we estimate the term on the right-hand side by means of Lemma \ref{lem:1909}.
 
Ad~{\bf (a)}. By \eqref{eq:2704} it follows that
\begin{align*}
\|(D\mathcal T_\tau(\bfg)-\id)\bfh\|_{L^2_x}&\leq\,c\tau\big(\|\bfh\|_{W^{2,2}_x}+\|f'(\mathcal T_\tau(\bfg))\|_{L^{3}_x}\|\bfh\|_{L^6_x}\big)\\
&\leq\,c\tau\big(\|\bfh\|_{W^{2,2}_x}+\Big(1+\|\mathcal T_\tau(\bfg)\|_{L^{6}_x}^{2}\Big)\|\bfh\|_{W^{1,2}_x}\big)\\
&\leq\,c\tau\|\bfh\|_{W^{2,2}_x}\Big(1+\|\bfg\|_{L^{6}_x}^{2}\Big).
\end{align*}
Ad~{\bf (b)}. Similarly to {\bf (a)}, \eqref{eq:2704B} yields
\begin{align}\label{eq:T-idW12}
\begin{aligned}
&\|(D\mathcal T_\tau(\bfg)-\id)\bfh\|_{W^{1,2}_x}\\
&\leq\,c\tau(1+\|\bfg\|_{W^{1,2}}^4)\big(\|\bfh\|_{W^{3,2}_x}+\|f''(\mathcal T_\tau(\bfg))\nabla (\mathcal T_\tau(\bfg))\bfh\|_{L^{2}_x}+\|f'(\mathcal T_\tau(\bfg))\nabla\bfh\|_{L^{2}_x}\big)\\
&\leq\,c\tau(1+\|\bfg\|_{W^{1,2}_x}^4)\big(\|\bfh\|_{W^{3,2}_x}+\|f''(\mathcal T_\tau(\bfg))\|_{L^6_x}\|\nabla(\mathcal T_\tau(\bfg))\|_{L^6_x}\|\bfh\|_{L^{6}_x}+\|f'(\mathcal T_\tau(\bfg))\|_{L^3_x}\|\nabla\bfh\|_{L^{6}_x}\big)\\
&\leq\,c\tau(1+\|\bfg\|_{W^{1,2}_x}^4)\Big(\|\bfh\|_{W^{3,2}_x}+\Big(1+\|\bfg\|_{L^{6}_x}\Big)\|\bfg\|_{W^{2,2}_x}\|\bfh\|_{W^{1,2}_x}+\Big(1+\|\bfg\|^{2}_{L^{6}_x}\Big)\|\bfh\|_{W^{2,2}_x}\Big)\\
&\leq\,c\tau(1+\|\bfg\|_{W^{1,2}_x}^4)\|\bfh\|_{W^{3,2}_x}+c\tau\Big(1+\|\bfg\|_{W^{2,2}_x}^{2}\Big)\|\bfh\|_{W^{2,2}_x}
\end{aligned}
\end{align}
using also \eqref{est:TL2} and \eqref{est:TW12}. Employing also
\eqref{eq:2704B} and \eqref{est:TW12} an analogous chain
gives
\begin{align*}
&\|(D\mathcal T_\tau(\bfg)-\id)\bfh\|_{W^{2,2}_x}\\
&\leq\,c\tau\Big(1+\tau \|\bfg\|_{W^{2,2}_x}^{8}\Big)\Big(\|\bfh\|_{W^{4,2}_x}+\|f''(\mathcal T_\tau(\bfg))\nabla\mathcal T_\tau(\bfg)\nabla\bfh\|_{L^{2}_x}\Big)\\
&\quad+c\tau\Big(1+\tau \|\bfg\|_{W^{2,2}_x}^{8}\Big)\|f'(\mathcal T_\tau(\bfg))\nabla^2\bfh\|_{L^{2}_x}+\|f''(\mathcal T_\tau(\bfg))\nabla^2\mathcal T_\tau(\bfg)\bfh\|_{L^{2}_x}\big)\\
&\quad+c\tau\Big(1+\tau \|\bfg\|_{W^{2,2}_x}^{8}\Big)\|f'''(\mathcal T_\tau(\bfg))\nabla\mathcal T_\tau(\bfg)\nabla\mathcal T_\tau(\bfg)\bfh\|_{L^{2}_x}\\
&\leq\,c\tau\Big(1+\|\bfg\|_{W^{2,2}_x}^{8}\Big)\big(\|\bfh\|_{W^{4,2}_x}+\|f'(\mathcal T_\tau(\bfg))\|_{L^{\infty}_x}\|\bfh\|_{W^{2,2}_x}\big)\\
&\quad+c\tau\Big(1+\|\bfg\|_{W^{2,2}_x}^{8}\Big)\big(\|f''(\mathcal T_\tau(\bfg))\|_{L^6_x}\|\nabla\mathcal T_\tau(\bfg)\|_{L^6_x}\|\nabla\bfh\|_{L^{6}_x}\big)\\
&\quad+c\tau\Big(1+\|\bfg\|_{W^{2,2}_x}^{8}\Big)\|f''(\mathcal T_\tau(\bfg))\|_{L^\infty_x}\|\nabla^2\mathcal T_\tau(\bfg)\|_{L^2_x}\|\bfh\|_{L^{\infty}_x}\\
&\quad+c\tau\Big(1+\|\bfg\|_{W^{2,2}_x}^{8}\Big)\|f'''(\mathcal T_\tau(\bfg))\|_{L^\infty_x}\|\nabla\mathcal T_\tau(\bfg)\|_{L^6_x}^2\|\bfh\|_{L^{6}_x}\\
&\leq\,c\tau\Big(1+\|\bfg\|_{W^{2,2}_x}^{8}\Big)\Big(\|\bfh\|_{W^{4,2}_x}+\Big(1+\|\mathcal T_\tau(\bfg)\|^{2}_{L^{\infty}_x}\Big)\|\bfh\|_{W^{2,2}_x}\Big)\\
&\quad+c\tau\Big(1+\|\bfg\|_{W^{2,2}_x}^{8}\Big)\Big(1+\|\mathcal T_\tau(\bfg)\|_{L^{6}_x}\|\nabla\mathcal T_\tau(\bfg)\|_{L^6_x}\|\nabla\bfh\|_{L^{6}_x}\Big)\\
&\quad+c\tau\Big(1+\|\bfg\|_{W^{2,2}_x}^{8}\Big)\Big(1+\|\mathcal T_\tau(\bfg))\|_{L^\infty_x}\Big)\|\nabla^2\mathcal T_\tau(\bfg)\|_{L^2_x}\|\bfh\|_{L^{\infty}_x}\\
&\quad+c\tau\Big(1+\|\bfg\|_{W^{2,2}_x}^{8}\Big)\|\nabla\mathcal T_\tau(\bfg)\|_{L^6_x}^2\|\bfh\|_{L^{6}_x}\\
&\leq\,c\tau\Big(1+\|\bfg\|_{W^{2,2}_x}^{8}\Big)\Big(\|\bfh\|_{W^{4,2}_x}+\Big(1+\|\mathcal T_\tau(\bfg)\|^{2}_{W^{2,2}_x}\Big)\|\bfh\|_{W^{2,2}_x}\Big).
\end{align*}
Using \eqref{est:TW22} we can finish the proof.

Ad~{\bf (c)}. Now we turn to the second derivative which can be written as
\begin{align*}
D^2\mathcal T_\tau(\bfg)&=-\textcolor{blue}{\tau}\big(D\mathcal T_\tau^{-1}(\mathcal T_\tau(\bfg))\big)^{-2}D^2\mathcal T_\tau^{-1}(\mathcal T_\tau(\bfg)) \big(D\mathcal T_\tau^{-1}(\mathcal T_\tau(\bfg))\big)^{-1}\\
&=-\textcolor{blue}{\tau}(D\mathcal T_\tau(\bfg))^2f''(\mathcal T_\tau(\bfg))D\mathcal T_\tau(\bfg).
\end{align*}
Using \eqref{eq:2704} and \eqref{est:TL2} with $p=6$ yields
\begin{align*}
\|D^2\mathcal T_\tau(\bfg)(\bfh_1,\bfh_2)\|_{L^2_x}&\leq \,c\tau\|f''(\mathcal T_\tau(\bfg))\|_{L^6_x}\|\bfh_1\|_{L^6_x}\|\bfh_2\|_{L^6_x}\\
&\leq \,c\tau\Big(1+\|\mathcal T_\tau(\bfg)\|_{L^{6}_x}\Big)\|\bfh_1\|_{L^6_x}\|\bfh_2\|_{L^6_x}\\
&\leq \,c\tau\Big(1+\|\bfg\|_{L^{6}_x}\Big)\|\bfh_1\|_{W^{1,2}_x}\|\bfh_2\|_{W^{1,2}_x}
\end{align*}
for all $\bfh_1,\bfh_2\in W^{1,2}_x$. Hence
the claim follows.
\end{proof}

\subsection{Semi-discretisation in time}
\label{sec:semi}
We consider an equidistant partition of $[0,T]$ with mesh size $\tau=T/M$ and set $t_m=m\Delta t$. 
Let $\bfu_0$ be an $\mathfrak F_0$-measurable random variable with values in $W^{1,2}(\mt)$. We aim at constructing iteratively a sequence of $\mathfrak F_{t_m}$-measurable random variables $\bfu_{m}$ with values in $W^{1,2}(\mt)$ such that
for every $\bfvarphi\in W^{1,2}(\mt)$ it holds true $\p$-a.s.
\begin{align}\label{tdiscrA}
\begin{aligned}
\int_{\mt}&\bfu_{m}\cdot\bfvarphi \dx +\tau\int_{\mt}f(\bfu_{m})\cdot\bfvarphi\dx+\tau\int_{\mt}\nabla\bfu_{m}:\nabla\bfvarphi\dx\\
&\qquad=\int_{\mt}\bfu_{m-1}\cdot\bfvarphi \dx+\int_{\mt}\Phi(\bfu_{m-1})\,\Delta_mW\cdot \bfvarphi\dx,
\end{aligned}
\end{align}
where $\Delta_m W=W(t_m)-W(t_{m-1})$. The existence of a unique $\bfu^m$ (given $\bfu_{m-1}$ and $\Delta_m W$) solving \eqref{tdiscrA} follows from its
re-interpretation as a convex minimisation problem. Moreover, the {\em discrete energy estimate} 
%
%
%
\begin{equation} \label{lem:3.1A}
\E\bigg[\max_{1\leq m\leq M} {\mathcal E}(u^m)+\tau\sum_{m=1}^M \|\mathcal E(u_m)\|_{L^2_x}^2 \bigg]\leq
cT
\end{equation}
holds under the assumptions made in Lemma \ref{lemma:3.1}. In fact, we will study the stability of $\bfu_m$ in detail in Section \ref{sec:esttau} and derive more general (and higher order) estimates in Lemma \ref{lemma:3.1}.

For the weak error analysis in Section \ref{sec:error} it will turn out to be useful to write \eqref{tdiscrA} as
\begin{align}\label{tdiscr'}
\bfu_m=\mathcal T_{\tau}\big(\bfu_{m-1}+\Phi(\bfu_{m-1})\Delta_m W\big),
\end{align}
where $\mathcal T_\tau$ is the discrete nonlinear semigroup corresponding to $D\mathcal E$, which we analysed in Section \ref{sec:Ttau} above. Note that different to previous works on
weak error analysis, \emph{cf.} \cite{BrDe,BrGo,De}, we treat the nonlinearity implicitly. Hence it is more complicated to define a time-continuous interpolant which coincides with $\bfu_m$ in $t_m$ and is still progressively measurable. Note, however, that $\mathcal T_\tau$ features nice properties similar to those that we have seen in the previous subsection.

 Setting
$$\quad \bfU_\tau(t)=\frac{1}{\tau}\int_{t_{m-1}}^t\bfu_{m-1}\ds+\int_{t_{m-1}}^t\Phi(\bfu_{m-1})\,\dd W$$
we introduce the $(\mathfrak F_t)$-adapted process
\begin{align}\label{eq:utau}
\begin{aligned}
\bfu_\tau(t)&=\frac{t_{m}-t}{\tau}\bfu_{m-1}+\mathcal T_\tau\bigg(\frac{1}{\tau}\int_{t_{m-1}}^t\bfu_{m-1}\ds+\int_{t_{m-1}}^t\Phi(\bfu_{m-1})\,\dd W\bigg)\\&=\frac{t_{m}-t}{\tau}\bfu_{m-1}+\mathcal T_\tau(\bfU_\tau(t))\end{aligned}
\end{align}
for $t\in[t_{m-1},t_m]$,
which coincides with $\bfu_{m-1}$ in $t_{m-1}$ and with $\bfu_m$ in $t_m$. In the following we linearise this formula around $U_\tau$, which gives a part which is (given $U_\tau$) linear in $u_{m-1}$ (similar to the method from \cite{BrDe,BrGo,De}) plus an error term. The latter will turn our to be globally of order $\tau$ as required.  
Applying It\^{o}'s formula to the second term in \eqref{eq:utau} yields
\begin{align}
\label{tdiscr'}
\begin{aligned}
\bfu_\tau(t)&=\frac{t_{m}-t}{\tau}\bfu_{m-1}+\frac{1}{\tau}\int_{t_{m-1}}^t D\mathcal T_\tau\big(\bfU_\tau(s)\big)\bfu_{m-1}\ds+\int_{t_{m-1}}^t D\mathcal T_\tau\big(\bfU_\tau\big)\Phi(\bfu_{m-1})\,\dd W\\
&\quad+\frac{1}{2}\sum_{k\geq 1}\int_{t_{m-1}}^t D^2\mathcal T_\tau\big(\bfU_\tau(s)\big)\big(\Phi(\bfu_{m-1})e_k,\Phi(\bfu_{m-1})e_k\big)\ds\\
&=\bfu_{m-1}+\frac{1}{\tau}\int_{t_{m-1}}^t \Big(D\mathcal T_\tau\big(\bfU_\tau(s)\big)-\id\Big)\bfu_{m-1}\ds\\&\quad+\int_{t_{m-1}}^t D\mathcal T_\tau\big(\bfU_\tau(s)\big)\Phi(\bfu_{m-1})\,\dd W\\
&\quad+\frac{1}{2}\sum_{k\geq 1}\int_{t_{m-1}}^t D^2\mathcal T_\tau\big(\bfU_\tau(s)\big)\big(\Phi(\bfu_{m-1})e_k,\Phi(\bfu_{m-1})e_k\big)\ds\quad\text{for}\quad t\in[t_{m-1},t_m].
\end{aligned}
\end{align}
%

Finally, we derive some uniform estimates for $U_\tau$ in terms of $u_m$. 
By the definition of $\bfU_\tau$, the Burkholder-Davis-Gundy inequality, and estimates \eqref{est:TL2} and \eqref{eq:phi}, we have for all $q\geq2$
\begin{align*}
\E\bigl[\|\bfU_\tau(t)\|^q_{L^2_x}\bigr]&\leq\,c\,\E[\|\bfu_{m-1}\|^q_{L^2_x}]+c\,\E\bigg[\sup_{s\in[t_{m-1},t]}\bigg\|\int_{t_{m-1}}^s\Phi(\bfu_{m-1})\,\dd W\bigg\|_{L^2_x}^q\bigg]\\&\leq\,c\,\E[\|\bfu_{m-1}\|^q_{L^2_x}]+c\,\E\biggl[\int_{t_{m-1}}^{t_m}\|\Phi(\bfu_{m-1})\|^q_{L_2(\mathfrak U;L^2_x)}\,\dif s\biggr]\\
&\leq\,c\,\E[1+\|\bfu_{m-1}\|^q_{L^2_x}]
\end{align*}
for $t\in[t_{m-1},t_m)$.
A similar argument applies when we replace $L^2(\mt)$ by $W^{1,2}(\mt)$
or $W^{2,2}(\mt)$ using this time \eqref{eq:S1} combined with \eqref{eq:phi} and \eqref{est:TW12} or \eqref{eq:S1}, \eqref{eq:phi2} and \eqref{est:TW22} respectively. We conclude
\begin{align}\label{eq:1102b}
\E \bigl[\|\bfU_\tau(t)\|^q_{W^{k,2}_x}\bigr]\leq\,c(q)\,\E\Bigl[1+\|\bfu_{m-1}\|^q_{W^{k,2}_x}\Bigr]
\end{align} 
uniformly in $\tau$ for $q\geq2$, $t\in[t_{m-1},t_m]$ and $k=0,1,2$. By the estimates \eqref{est:TL2} and \eqref{est:TW12} in Lemma \ref{lem:1909}, formula \eqref{eq:utau} thus yields
\begin{align}\label{eq:1102c}
\E \bigl[\|\bfu_\tau(t)\|^q_{W^{k,2}_x}\bigr]\leq\,c(q)\,\E\Bigl[1+\|\bfu_{m}\|^q_{W^{k,2}_x}\Bigr]
\end{align} 
uniformly in $\tau$ for $q\geq2$, $t\in[t_{m-1},t_m]$ and $k=0,1$. Similarly, \eqref{est:TW22} in Lemma \ref{lem:1909} implies
\begin{align}\label{eq:1102d}
\E \bigl[\|\bfu_\tau(t)\|^q_{W^{2,2}_x}\bigr]\leq\,c(q)\,\E\Bigl[1+\|\bfu_{m}\|^{3q}_{W^{2,2}_x}\Bigr].
\end{align} 
Note that controlling the right-hand sides of \eqref{eq:1102b}--\eqref{eq:1102d} is not straightforward and will be done in the next subsection.

\subsection{Estimates for the time-discrete solution}
\label{sec:esttau}
We now derive some uniform estimates for the solution of the time-discrete problem \eqref{tdiscrA}.
These estimates involve the energy ${\mathcal E}(\cdot)$ from (\ref{ac-1}) as relevant Liapunov functional, reflecting that the discrete problem \eqref{tdiscr'} inherits the relevant gradient flow property of the original problem (\ref{eq:SNS}).
\begin{lemma}\label{lemma:3.1} Let $(\Omega,\mf,(\mf_t)_{t\geq0},\prst)$ be a given stochastic basis with a complete right-con\-ti\-nuous filtration and an $(\mf_t)$-cylindrical Wiener process $W$. Let $T\equiv t_M >0$, assume  that $\mathcal E(\bfu_0)\in L^{2^q}(\Omega)$ for some  $q \in {\mathbb N}_0$. 
Choose $\tau \leq \frac{1}{4}$. The iterates $(\bfu_m)_{m=1}^M$ from \eqref{tdiscrA}
satisfy the following estimates.\begin{enumerate}
\item Suppose that either \ref{N1} {\bf (a)} holds and that $\Phi$ is bounded and summable in the sense that $\sup_ {x,\xi}|\Phi(x,\xi)e_k|\leq\,\mu_k$ with $\sum_{k\geq1}\mu_k<\infty$ or that assumption \ref{N2} is in place.
For all $q\in\N_0$ there exists $c=c(q,T,u_0)$, such that
\begin{align}
\label{lem:3.1a}\E\bigg[\max_{0\leq m\leq M} \bigl[{\mathcal E}(\bfu_m)\bigr]^{2^q} +
\tau\sum_{m=1}^M {\bigl[{\mathcal E}(\bfu_m)\bigr]^{2^q-1}}\|D{\mathcal E}(\bfu_m)\Vert^2_{L^2_x} \bigg]
&\leq\,c \, .
\end{align}
\item Assume that $\bfu_0\in L^{2^q}(\Omega,W^{2,2}(\mt))$,  $\mathcal E(u_0)\in L^{2^{q+2}}(\Omega)$ for some $q\in\N_0$ and that $\Phi$ satisfies \ref{N1}  if $d=1,2$, and \ref{N2} if $d=3$. Then we have
\begin{align}
\label{lem:3.1c}\E\bigg[\max_{1\leq m\leq M}\|\bfu_m\|^{2^q}_{W^{2,2}_x}+\tau\sum_{m=1}^M {\|\bfu_m\|_{W^{2,2}_x}^{2^q-2}}\|\nabla^3\bfu_m\|^2_{L^2_x}\bigg]&\leq\,c\, .
\end{align}
\item Assume that $\bfu_0\in L^{2^q}(\Omega,W^{3,2}(\mt))\cap L^{2^{q+2}}(\Omega,W^{2,2}(\mt))$, $\mathcal E(u_0)\in L^{2^{q+4}}(\Omega)$ for some $q\in\N$ and that $\Phi$ satisfies \ref{N1}   if $d=1,2$, and \ref{N2} if $d=3$. Then we have
\begin{align}
\label{lem:3.1d}\E\bigg[\max_{1\leq m\leq M}\|\bfu_m\|^{2^q}_{W^{3,2}_x}+\tau\sum_{m=1}^M {\|\bfu_m\|_{W^{3,2}_x}^{2^q-2}}\|\nabla^4\bfu_m\|^2_{L^2_x}\bigg]&\leq\,c\, .
\end{align}
\end{enumerate}
Here $c=c(q,T,\bfu_0)>0$ is independent of $\tau$.
\end{lemma}
\begin{remark}
Part {\bf (a)} is an {\em energy estimate} for the solution $(\bfu_m)_{m=1}^M$ from \eqref{tdiscrA}; in the context of phase-field models where the mesh width $\varepsilon >0$ enters, it avoids exponential growth with respect to $\varepsilon^{-1}$ in time -- as its continuous counterpart.  
Its derivation below is close to \cite[Lemma 3.2]{AnBaNuPr}, but here generalizes  to
 noise of type \ref{N2}. The remaining parts {\bf (b)} and {\bf (c)}
give {\em higher norm estimates}.
\end{remark}
\begin{proof}[Proof of Lemma \ref{lemma:3.1}.]
Ad~{\bf (a)}. Choose $q=0$ first. We proceed in several steps:  {\bf I.} derives a first energy estimate, which creates a new term, which is then bounded independently in {\bf II.}. To extend the estimate {\bf (a)} to $q \geq 1$ then follows from a general
inductive argument as in \cite[Lemma 3.2]{AnBaNuPr}.

{\bf I.} (Formally) choose $\varphi = D {\mathcal E}(u_m)$ in (\ref{tdiscrA}) and integrate to get
\begin{equation}\label{eins-1}
\bigl( D {\mathcal E}(u_m), u_m - u_{m-1}\bigr)_{L^2_x} + \tau \Vert D{\mathcal E}(u_m)\Vert^2_{L^2_x} = \bigl(\Phi(u_{m-1})\Delta_mW, D{\mathcal E}(u_m) \bigr)_{L^2_x} ,
\end{equation}
which we write as 
\begin{align*}
{\tt I}^m+\tau\Vert D{\mathcal E}(u_m)\Vert^2_{L^2_x} ={\tt II}^m,
\end{align*}
where ${\tt  I}^m={\tt I}^m_{\tt A}+{\tt I}^m_{\tt B}$ with 
\begin{align*}
{\tt I}^m_{\tt A}& := \bigl( \nabla u_m, \nabla(u_m - u_{m-1})\bigr)_{L^2_x}\\& =
\frac{1}{2} \Bigl( \Vert \nabla u_m\Vert^2_{L^2_x} - \Vert \nabla u_{m-1}\Vert^2_{L^2_x}+ \Vert \nabla (u_m-u_{m-1})\Vert^2_{L^2_x}\Bigr)_{L^2_x}\,, \\
{\tt I}^m_{\tt B} &:=   \bigl( f(u_m), u_m - u_{m-1}\bigr)_{L^2_x} \, .
\end{align*}
The following estimate holds (see {\em e.g.} \cite[estimate (35)]{FePr})
\begin{equation}\label{tfe1} {\tt I}^m_{\tt B}  \geq    \Bigl(F(u_m) - F(u_{m-1}) \Bigr) +
\frac{1}{4} \Bigl\Vert   \vert u_m\vert^2 - \vert u_{m-1}\vert^2 \Bigr\Vert^2_{L^2_x} - 
\frac{1}{2} \bigl\Vert u_m - u_{m-1}\bigr\Vert^2_{L^2_x}\, ,
\end{equation}
whose short proof is recalled here for the reader's convenience, and which is based on repeated use of binomial formulae: since $$f(u_m) =\bigl( \vert u_m\vert^2-1\bigr) u_m = \frac{1}{2} \bigl( \vert u_m\vert^2-1\bigr) \bigl( [u_m+u_{m-1}] + [u_m - u_{m-1}]\bigr)\,, $$
we have that
\begin{align*}
\bigl( f(u_m), u_m-u_{m-1}\bigr)_{L^2_x} &= \frac{1}{2} \bigl( \vert u_m\vert^2-1, \bigl[\vert u_m\vert^2-1\bigr] - \bigl[\vert u_{m-1}\vert^2-1\bigr]\bigr)_{L^2_x} \\
&+ \frac{1}{2} \bigl( \vert u_m\vert^2-1, \vert u_m\  -   u_{m-1}\vert^2\bigr)_{L^2_x} \\
&= \frac{1}{4} \Bigl(\bigl\Vert \vert u_m\vert^2-1\bigr\Vert_{L^2_x}^2 - \bigl\Vert \vert u_{m-1}\vert^2-1\bigr\Vert_{L^2_x}^2 + \bigl\Vert \vert u_m\vert^2 - \vert u_{m-1}\vert^2\bigr\Vert_{L^2_x}^2\Bigr) \\
&+ \frac{1}{2}\Vert u_m (u_m-u_{m-1})\Vert^2_{L^2_x} - \frac{1}{2}\Vert u_m - u_{m-1}\Vert^2_{L^2_x}\\
&\geq F(u_m) - F(u_{m-1}) + \frac{1}{4} \bigl\Vert \vert u_m\vert^2 - \vert u_{m-1}\vert^2\bigr\Vert^2_{L^2_x} - \frac{1}{2}\Vert u_m - u_{m-1}\Vert^2_{L^2_x}\,. 
\end{align*}
By coming back to (\ref{eins-1}), we then obtain after summation
\begin{align}\nonumber
\mathcal E ( u^m) &+ \frac{1}{} \sum_{n=1}^m\Big(\Vert \nabla (u_n-u_{n-1})\Vert^2_{L^2_x} +
 \frac{1}{2} \Bigl\Vert   \vert u_n\vert^2 - \vert u_{n-1}\vert^2 \Bigr\Vert^2_{L^2_x}\Big)+\sum_{n=1}^m\tau\Vert D{\mathcal E}(u_n)\Vert^2_{L^2_x} \\ \label{bas1}
&\leq \mathcal E(u_0)+ \frac{1}{2}\sum_{n=1}^m\bigl\Vert u_n - u_{n-1}\bigr\Vert^2_{L^2_x}+\sum_{n=1}^m {\tt II}^n.
\end{align}
To proceed we must control the stochastic term $\mathcal N_m:=\sum_{n=1}^m {\tt II}^n$, where
\begin{align*} 
 {\tt II}^n   &:=  \bigl(\Phi(u_{n-1})\Delta_nW, D{\mathcal E}(u_{n})\bigr)_{L^2_x}.
\end{align*}
It follows from the definition of $ {\mathcal E}(u^n)$ that
\begin{align*} 
 {\tt II}^n   
&=   -\bigl( D\Phi(u_{n-1})   \nabla u_{n-1}\Delta_nW,  \nabla (u_n- u_{n-1})  \bigr)_{L^2_x} \\
&\qquad   + \bigl( \Phi(u_{n-1})\Delta_n W , f(u_n) - f(u_{n-1})\bigr)_{L^2_x}   +\bigl(\Phi(u_{n-1})\Delta_nW, D{\mathcal E}(u_{n-1})\bigr)_{L^2_x}\\
&=: {\tt II}^n_{\tt A}+{\tt II}^n_{\tt B} +{\tt II}^n_{\tt C}.
\end{align*}
Corresponding terms are $\mathcal N_m^{\tt A}$, $\mathcal N_m^{\tt B}$,  and  $\mathcal N_m^{\tt C}$, which we now bound independently. 

{\bf ${\bf I}_{\bf 1}$)} We have by Young's inequality and It\^{o}-isometry for $\kappa>0$ arbitrary
\begin{align*}
\E\bigg[\max_{1\leq m\leq M}\mathcal N_m^{\tt A}\bigg]&\leq\,\kappa\, \E\bigg[\sum_{n=1}^M\|\nabla (u_n- u_{n-1})  \|_{L^2_x}^2\bigg]\\
&\quad+c(\kappa)\, \E\bigg[\sum_{n=1}^M\bigg\|\int_{t_{n-1}}^{t_n}D\Phi(u_{n-1})   \nabla u_{n-1}\,\dd W\bigg\|_{L^2_x}^2\bigg]\\
&=\,\kappa\, \E\bigg[\sum_{n=1}^M\|\nabla (u_n- u_{n-1})  \|_{L^2_x}^2\bigg]\\
&\quad+c(\kappa)\, \E\bigg[\sum_{n=1}^M\tau \big\|D\Phi(u_{n-1})   \nabla u_{n-1}\big\|_{L_2(\mathfrak U; L^2_x)}^2\bigg]\\
&\leq\,\kappa\, \E\bigg[\sum_{n=1}^M\|\nabla (u_n- u_{n-1})  \|_{L^2_x}^2\bigg]+c(\kappa)\, \E\bigg[\sum_{n=1}^M\tau \big(1+\big\|\nabla u_{n-1}\big\|_{ L^2_x}^2\big)\bigg]
\end{align*}
using also \eqref{eq:phi} in the last step. The first term can be absorbed for $\kappa\ll1$ on the left-hand side of (\ref{bas1}) and,
since $\big\|\nabla u_{n-1}\big\|_{ L^2_x}^2\leq\,2\mathcal E(u_{n-1})$, the final term can be handled by (discrete) Gronwall's lemma there. 
\\

{\bf ${\bf I}_{\bf 2}$)} We split the next term as $\mathcal N_m^{\tt B} := \mathcal N_m^{{\tt B},1} + \mathcal N_m^{{\tt B},2}$, which is motivated by the algebraic identity
$$f(b) - f(a) \equiv (b^3 - b) - (a^3 - a) = (b^2-a^2)b + (a^2-1)(b-a) \quad \forall\, a,b \in {\mathbb R} \, .$$
By generalized Young's inequality we estimate ($\kappa > 0$)
\begin{align}\nonumber
\E\bigg[\max_{1\leq m\leq M}\mathcal N_m^{{\tt B},1}\bigg]
&\leq\,\kappa\,\E\bigg[\sum_{n=1}^M \| \vert u_n\vert^2- \vert u_{n-1}\vert^2  \|_{L^2_x}^2\bigg]+
c(\kappa)\, \E\bigg[\sum_{n=1}^M\tau \big\| u_{n}\big\|^4_{L^4_x}  \bigg]\\ \label{dwd1}
&+c(\kappa)\, \E\bigg[\sum_{n=1}^M\tau^{-1}\bigg\|\int_{t_{n-1}}^{t_n}\Phi(u_{n-1})\,\dd W \bigg\|_{L_x^4}^4 \bigg]
\end{align}
For $\kappa \ll 1$ sufficiently small,  the first term may be absorbed on the left-hand side of (\ref{bas1}); for the second term, we note that
$$ \Vert u_n\Vert^4_{L^4_x}  \leq c \bigl(1+ \bigl\Vert\vert u_n\vert^2 - 1\bigr\Vert^2_{L^2_x}\bigr)\,.
$$
The latter part in the last inequality is part of ${\mathcal E}(u_n)$, which is why this term may now be bounded via discrete Gronwall inequality. For the final term we use estimates for the stochastic integral in UMD-Banach spaces, see \cite{VVW}. For an UMD Banach space $(X;\|\cdot\|)$ and a separable Hilbert space $\mathbb H$ with orthonormal basis $(h_k)_{k\geq1}$ we denote by $\gamma(\mathbb H,X)$
the space of $\gamma$-radonifying operators
 from $\mathbb H\rightarrow X$ with norm
\begin{align*}
\|\Phi\|_{\gamma(\mathbb H,X)}^2:=\mathbb E' \biggl[\bigg\|\sum_{k\geq 1}\gamma_k\Phi h_k\bigg\|_{X}^2\biggr] \quad \forall\, \Phi\in \mathcal L (\mathbb H,X)\, ,
\end{align*}
where $(\gamma_k)_{k\geq 1}$ is a sequence of standard Gaussian random variables given on probability space $(\Omega',\mathfrak F',\p')$ with expectation $\E'$. Note that this differs from the original probability space $(\Omega,\mathfrak F,\p)$. However, in the case of Hilbert spaces the norm above coincides with the Hilbert-Schmidt operator norm. In particular, the additional randomness disappears. In our case, it will be removed by the following estimate and is only introduced to quote the estimate from \cite{VVW}.
  We obtain using the assumption made on $\Phi$
\begin{align*}
\|\Phi(u)\|_{\gamma(\mathfrak U;L^4_x)}^2\leq \,c(1+\|u\|_{L^4_x}),\quad u\in L^4_x.
\end{align*}
Using now the estimate from \cite{VVW} for $X=L^4_x$
 as well as \eqref{eq:phi} we obtain
\begin{align*}
&\E\bigg[\sum_{n=1}^M\tau^{-1}\bigg\|\int_{t_{n-1}}^{t_n}\Phi(u_{n-1})\,\dd W \bigg\|_{L_x^4}^4 \bigg]\\
&\qquad\qquad\leq\,c\,\E\bigg[\sum_{n=1}^M\tau^{-1}\bigg\|\int_{t_{n-1}}^{\cdot}\Phi(u_{n-1})\bigg\|_{\gamma(L^2(0,T;\mathfrak U),L^4(\mt))}^4\bigg]\\
&\qquad\qquad=\,c\,\E\bigg[\sum_{n=1}^M\tau\big\|\Phi(u_{n-1})\big\|_{\gamma(\mathfrak U;L_x^4)}^4\bigg]\\
&\qquad\qquad\leq\,c\,\E\bigg[\sum_{n=1}^M\tau\big(1+\big\|u_{n-1}\big\|_{L_x^4}^4\big) \bigg]\\
&\qquad\qquad\leq\,c\,\E\bigg[\sum_{n=1}^M\tau\big(1+\big\||u_{n-1}|^2-1\big\|_{L_x^2}^2\big) \bigg].
\end{align*}

To bound the term $\mathcal N_m^{{\tt B},2}$, we distinguish the type of admissible noise:

{\bf ${\bf I}_{\bf 2_1}$)}  Let $\Phi$ satisfy \ref{N1} (a) and be bounded. Then there exists $c\geq0$ such that ($\kappa>0$ arbitrary):
\begin{align}\label{abb1}
\E\bigg[\max_{1\leq m\leq M}\mathcal N_m^{{\tt B},2}\bigg]
&\leq \E\bigg[\bigg(\sum_{n=1}^M\tau\sum_{k\geq 1}\bigg(\int_{\mt}\Phi(u_{n-1})e_k\,(|u_{n-1}|^2-1)(u_n-u_{n-1})\dx\bigg)^2\bigg)^{1/2}\bigg]\nonumber\\
&\leq \E\bigg[\bigg(\sum_{n=1}^M\|u_{n-1}|^2-1\|_{L^2_x}^2\|u_n-u_{n-1}\|_{L^2_x}^2\bigg)^{1/2}\bigg]\nonumber\\
& \leq \kappa\, {\mathbb E}\biggl[ \max_{1\leq n\leq M}   \bigl\Vert \vert u_{n-1}\vert^2 -1 \bigr\Vert^2_{L^2_x}\biggr] + c(\kappa){\mathbb E}\biggl[\sum_{n=1}^M\Vert u_n - u_{n-1}\Vert^2_{L^2_x}\biggr]
\end{align}
The first term will again be be bounded via discrete Gronwall inequality, while the second term will be bounded below.

{\bf ${\bf I}_{\bf 2_2}$)} Let $\Phi$ satisfy \ref{N2}. We use a compatibility property of data $\Phi$  and $f$ from $D{\mathcal E}$. For the derivation of the bound, it suffices to consider the case that all $\beta_k \equiv 0$ in (N2) only. Now fix one $k \in {\mathbb N}$. Then
\begin{align*}
{\tt II}_{\tt B}^n=&\sum_{k\geq1}\bigl( \alpha_k u_{n-1}\Delta_n W_k, \bigl[ \vert u_{n-1}\vert^2-1\bigr] \cdot[ u_n - u_{n-1}]\bigr)_{L^2_x} \\
&\qquad = \sum_{k\geq1}\bigl( \alpha_k \bigl[f(u_{n-1}) - \Delta u_{n-1}\bigr] \Delta_n W_k, u_n - u_{n-1}\bigr)_{L^2_x}\\&\qquad -\sum_{k\geq1} \bigl( \alpha_k \nabla u_{n-1} \Delta_n W_k, \nabla[u_n - u_{n-1}] \bigr)_{L^2_x} \\
& \qquad =: {\tt B}_n^1 + {\tt B}_{n}^2\, .
\end{align*}
After summation, and taking expectations, we first use Young's inequality, and then use It\^{o}-isometry as well as \eqref{eq:phismooth} ($\kappa > 0$)
%
%
\begin{align}\nonumber
{\mathbb E}\biggl[ \max_{1\leq m\leq M}\sum_{n=1}^m {\tt B}_n^1\biggr]
&\leq {\mathbb E}\biggl[ \max_{1 \leq m \leq M} \sum_{n=1}^m\bigg\Vert \sum_{k\geq1}\alpha_kD{\mathcal E}(u_{n-1})\Delta_n W_k\bigg\Vert_{L^2_x} \Vert u_n - u_{n-1}\Vert_{L^2_x}\biggr] \\
\nonumber&= {\mathbb E}\biggl[ \sum_{n=1}^M\bigg\Vert \sum_{k\geq1}\alpha_kD{\mathcal E}(u_{n-1})\Delta_n W_k\bigg\Vert_{L^2_x} \Vert u_n - u_{n-1}\Vert_{L^2_x}\biggr] \\
&\leq {\mathbb E}\biggl[ \sum_{n=1}^M \kappa\, \biggl\Vert \sum_{k\geq1}\alpha_kD{\mathcal E}(u_{n-1})\Delta_n W_k\biggr\Vert^2_{L^2_x} + c(\kappa)\, \Vert u_n - u_{n-1}\Vert_{L^2_x}^2\biggr]
\nonumber \\
&=  \kappa \, {\mathbb E}\biggl[ \sum_{m=1}^M  \tau\, \bigl\Vert D {\mathcal E}(u_{m-1})\bigr\Vert^2_{L^2_x}\biggr] +
c(\kappa)\, {\mathbb E}\biggl[ \sum_{n=1}^M \Vert u_n - u_{n-1}\Vert^2_{L^2_x}\biggr]\, . \label{abb2}
\end{align}
The leading term may now be absorbed on the left-hand side of (\ref{bas1}), while the last term is the same as in (\ref{abb1}),  and will be bounded below. The
estimation of ${\mathbb E}\bigl[ \max_{1\leq m\leq M}\sum_{n=1}^m {\tt B}_n^2\bigr]$ is immediate.\\

{\bf ${\bf I}_{\bf 3}$)} Since $\mathcal N_m^{\tt C}$ is a martingale we finally obtain by Burkholder-Davis-Gundy inequality, the expression $D\mathcal E(u_{n-1})=-\Delta u_{n-1}+f(u_{n-1})$ and
 \eqref{eq:phi},
 \begin{align*}
\E\bigg[\max_{1\leq m\leq M}\mathcal N_m^{\tt C}\bigg] &\leq c\, 
\E\bigg[\bigg(\sum_{n=1}^M\sum_{i\geq1}\tau\bigg(\int_{\mt}\Phi(u_{n-1})e_i\,(-\Delta u_{n-1}+u_{n-1}^3-u_{n-1})\dx\bigg)^2\bigg)^{\frac{1}{2}}\bigg]\\
&\leq\,c\,\E\bigg[\bigg(\sum_{n=1}^M\tau\big(1+\mathcal E(u_{n-1})^2\big)\bigg)^{\frac{1}{2}}\bigg]\\
&\leq\,\kappa\,\E\bigg[\max_{1\leq n\leq M}\mathcal E(u_{n-1})\bigg]+c(\kappa)\sum_{n=1}^M\tau\E\big[1+\mathcal E(u_{n-1})\big],
\end{align*}
which can be handled by absorption (for $\kappa$ small enough) and Gronwall's lemma.\\

By inserting the estimates {\bf ${\bf I}_{\bf 1}$)}--{\bf ${\bf I}_{\bf 3}$)} into (\ref{bas1}) and choosing $\kappa \ll 1$ small enough to allow absorption of terms,  there exists $c \equiv c(t_M)> 0$ such
\begin{align}\nonumber
&\E\bigg[\max_{1\leq n\leq M}\mathcal E ( u_n)\bigg]  + \frac{1}{2}{\mathbb E}\Bigl[\sum_{n=1}^m\Big(\Vert \nabla (u_n-u_{n-1})\Vert^2_{L^2_x} +\frac{1}{2}\Bigl\Vert   \vert u_n\vert^2 - \vert u_{n-1}\vert^2 \Bigr\Vert^2_{L^2_x}\Big)\Bigr] \\ \label{eq:0807}
&\qquad + \E\bigg[\sum_{n=1}^M\tau\Vert D{\mathcal E}(u_n)\Vert^2_{L^2_x}\bigg] \leq c\, \Bigl(\E\big[\mathcal E(u_0)\big]+\E\bigg[\sum_{n=1}^m\bigl\Vert u_n - u_{n-1}\bigr\Vert^2_{L^2_x}\bigg]\Bigr)\, .
\end{align}
{\bf II.} To bound the  last term in (\ref{eq:0807}) we
choose $\varphi = \bfu_{m}- \bfu_{m-1}$ in (\ref{tdiscrA}) to get
 \begin{align}\nonumber
\bigl\Vert \bfu_{m}-\bfu_{m-1}\Vert^2_{L^2_x} +\tau \bigl( {\tt I}_a^m + {\tt I}^m_b\bigr) &= \bigl( \Phi(\bfu_{m-1}) \,\Delta_{m}W, \bfu_{m}-\bfu_{m-1} \big)_{L^2_x} \\ \label{eq:0807b}
&\leq\Vert \Phi(\bfu_{m-1}) \,\Delta_{m}W\Vert_{L^2_x}^2 + \frac{1}{4} \Vert \bfu_{m}-\bfu_{m-1}\Vert^2_{L^2_x}\, ,
\end{align}
where the terms ${\tt I}^m_{\tt A}$ and ${\tt I}^m_{\tt B}$ are from (\ref{tfe1}).
Hence, we resume that
\begin{eqnarray*}
\bigl(1-\frac{1}{4} - \frac{\tau}{2}\bigr) \Vert \bfu_{m}-\bfu_{m-1}\Vert^2_{L^2_x}
+ \tau \bigl( {\mathcal E}(u_m) - {\mathcal E}(u_{m-1}) \bigr)
\leq \Vert \Phi(\bfu_{m-1}) \,\Delta_{m}W\Vert_{L^2_x}^2
\end{eqnarray*}
For $\tau \leq \frac{1}{4}$, and after summation over $1 \leq m \leq M$,
 \begin{equation*}
\frac{1}{2} \E\bigg[\sum_{m=1}^M\|\bfu_{m}-\bfu_{m-1}\|_{L^2_x}^2\Bigr] + 
\tau \, \E\bigl[{\mathcal E}(u_M)\bigr]  \leq \tau\, \E\big[ {\mathcal E}(\bfu_{0})\big] +
{\mathbb E}\bigg[ \sum_{m=1}^M\bigl\Vert \Phi(u_{m-1}) \Delta_mW \bigr\Vert^2_{L^2_x}\bigg]\, .
\end{equation*}
By It\^{o}-isometry, the last term is bounded by $${\mathbb E}\bigl[ \sum_{m=1}^M
\tau \Vert \Phi(u_{m-1})\Vert^2_{L_2(\mathfrak U;L^2_x)}\bigr]\leq\,c\,{\mathbb E}\Bigl[ \sum_{m=1}^M
\tau \big(1+\Vert u_{m-1}\Vert^2_{L^2_x}\big)\Bigr];$$ upon inserting this estimate into (\ref{eq:0807}) settles the estimate {\bf (a)} for $q = 0$.\\

{\bf III.} We may now proceed inductively  to settle {\bf (a)} for $q \geq 1$, by multiplying (\ref{eins-1}) with $\bigl[{\mathcal E}(\bfu_{m})\bigr]^{2^q-1}$ before summation; it is the implicit numerical treatment of drift terms in scheme (\ref{tdiscrA}) that generates new numerical diffusion terms which then control newly arising terms; see \cite[Lemma 3.2]{AnBaNuPr} for details.%
%

\medskip

Ad~{\bf (b)}. Before we come to the proof of 
\eqref{lem:3.1c} we need some preliminary estimates for lower order deriavtives. We
choose $\bfu_{m}$ as a test function to get
 \begin{align}\label{eq:0807b}
 \begin{aligned}
\big(&\bfu_{m}-\bfu_{m-1},\bfu_m\big)_{L^2_x} +\tau\|\nabla\bfu_{m}\|^2_{L^2_x}\\&=
 -\tau(f(u_m),u_m)_{L^2_x} 
+ \big(\Phi(\bfu_{m-1})\,\Delta_{m-1}W,\bfu_{m}\big)_{L^2_x},
\end{aligned}
\end{align}
where, clearly,
\begin{align*}
 -\tau(f(u_m),u_m)_{L^2_x} \leq\,\tau \bigl(\|u_m\|_{L^2_x}^2 - \Vert u_m\Vert_{L^4_x}^4\bigr)\, .
\end{align*}
By binomial formula, and iterating this estimate and applying Gronwall's lemma proves
 \begin{align*}
\E\bigg[\max_{1\leq m\leq M}\tfrac{1}{2}\|\bfu_{m}\|_{L^2_x}^2\bigg]&+\E\bigg[\tfrac{1}{2}\sum_{m=1}^M\|\bfu_{m}-\bfu_{m-1}\|_{L^2_x}^2 + \tau\sum_{m=1}^M \bigl(\|\nabla\bfu_{m}\|^2_{L^2_x}+ \Vert u_m\Vert_{L^4_x}^4\bigr)\bigg]\\&\leq \,c\,\E\big[\tfrac{1}{2} \|\bfu_{0}\|^2_{L^2_x}\big]+c\,\E\bigg[\max_{1\leq m\leq M}\mathscr M_m^1\bigg]+c\,\E\bigg[\max_{1\leq m\leq M}\mathscr M_m^2\bigg],
\end{align*}
where 
\begin{align*}
\mathscr M_m^1&=\sum_{n=1}^m\big(\Phi(\bfu_{n-1})\,\Delta_nW,\bfu_{n-1}\big)_{L^2_x},\\
\mathscr M_m^2&=\sum_{n=1}^m\big(\Phi(\bfu_{n-1})\,\Delta_nW, \bfu_n-\bfu_{n-1}\big)_{L^2_x}.
\end{align*}
Since $\bfu_{m-1}$ is $\mathfrak F_{t_{m-1}}$-measurable $\mathscr M_m^1$ is a (discrete) martingale
such that, by Burkholder-Davis-Gundy inequality, \eqref{eq:phi} and Young's inequality,
\begin{align*}
\E\bigg[\max_{1\leq m\leq M}\big|\mathscr M_{m}^1\big|\bigg]&\leq\,c\,\E\bigg[\bigg(\sum_{n=1}^M\int_{t_{n-1}}^{t_n}\|\Phi(\bfu_{n-1})\|^2_{L_2(\mathfrak U,L^2_x)}\|\bfu_{n-1}\|^2_{L^2_x}\ds\bigg)^{\frac{1}{2}}\bigg]\\
&\leq\,c\,\E\bigg[\max_{0\leq n\leq M}\|\bfu_{n}\|_{L^2_x}\bigg(\tau\sum_{n=1}^M\|\Phi(\bfu_{n-1})\|^2_{L_2(\mathfrak U,L^2_x)}\ds\bigg)^{\frac{1}{2}}\bigg]\\
&\leq\,\kappa\,\E\bigg[\max_{0\leq n\leq M}\|\bfu_{n}\|^2_{L^2_x}\bigg]+\,c(\kappa)  \Bigl(\E\Big[\tau\sum_{n=1}^M\|\bfu_{n-1}\|_{L^2_x}^2\Big]+ {\color{blue} t_M}
\Bigr)
\end{align*} 
for any $\kappa>0$.
Furthermore, we have for arbitrary $\kappa>0$
\begin{align*}
\E\bigg[\max_{1\leq m\leq M}|\mathscr M^2_{m}|\bigg]&\leq \,\kappa\,\E\bigg[  \sum_{n=1}^M\big\|\bfu_{n}-\bfu_{n-1}\big\|_{L^2_x}^2\bigg]+c(\kappa)\,\E\bigg[\sum_{n=1}^M\bigg\| \int_{t_{n-1}}^{t_n}\Phi(\bfu_{n-1})\,\dd W  \bigg\|_{L^2_x}^2\bigg]\\
&\leq \,\kappa\,\E\bigg[ \sum_{n=1}^M\|\bfu_{n}-\bfu_{n-1}\|_{L^2_x}^2 \bigg]+ c(\kappa)\,\E\bigg[\tau\sum_{n=1}^M\|\Phi(\bfu_{n-1})\|_{L_2(\mathfrak U;L^2_x)}^2\bigg]\\
&\leq \,\kappa\,\E\bigg[\sum_{n=1}^M \|\bfu_{n}-\bfu_{n-1}\|_{L^2_x}^2 \bigg]+ c(\kappa)\,\E\bigg[\tau\sum_{n=1}^M\|\bfu_{n-1}\|^2_{L^2_x}  + t_M\bigg]
\end{align*}
due to Young's inequality, It\^{o}-isometry and \eqref{eq:phi}. Absorbing the $\kappa$-terms and applying Gronwall's lemma we conclude
 \begin{align*}
\E\bigg[\max_{1\leq n\leq M}&\|\bfu_{n}\|^2_{L^2_x}
+\tau\sum_{n=1}^M \bigl( \|\nabla\bfu_n\|^2_{L^2_x}+ \Vert \bfu_n\Vert^4_{L^4_x}\bigr)\bigg]\leq
\,c\,\E\Bigl[\|\bfu_{0}\|^2_{L^2_x}+1\Bigr]\, .
\end{align*}
Let now $q\in\N$: we argue by induction as in {\bf (a)} (first multiply (\ref{eq:0807b}) by $\|u_n\|^{2^{q}-2}_{L^2_x}$, then take expections to get
bounds for involved; afterwards, use these bounds when $\max$ is applied first before taking expectations; see \cite[Lemma 3.2]{AnBaNuPr} for details) to get
  \begin{align}\label{eq:0807d}
 \begin{aligned}
\E\bigg[\max_{1\leq n\leq M}&\|\bfu_{n}\|^{2^q}_{L^2_x}+\tau\sum_{n=1}^M\|u_n\|^{2^{q}-2}_{L^2_x} \bigl(\|\nabla\bfu_n\|^2_{L^2_x}+ \Vert \bfu_n\Vert^4_{L^4_x}\bigr)\bigg]\leq
\,c\,\E\Bigl[\|\bfu_{0}\|^{2^q}_{L^2_x}+1\Bigr]\, .
\end{aligned}
\end{align}

By formally testing with $-\Delta u_m$, using that
\begin{align*}
 -\tau(f(\bfu_m),\Delta\bfu_m)_{L^2_x}&\geq \tau \bigl(
3 \Vert u_m \nabla u_m\Vert^2_{L^2_x} - \|\nabla u_m\|_{L^2_x}^2\bigr)
\end{align*}
and controlling the stochastic integral with the help of \eqref{eq:phi1b} we obtain similarly
 \begin{align*}
\E\bigg[\max_{1\leq n\leq M}&\|\bfu_{n}\|^2_{W^{1,2}_x}
+\tau\sum_{n=1}^M \|\Delta\bfu_n\|^2_{L^2_x}\bigg]\leq
\,c\,\E\Bigl[\|\bfu_{0}\|^2_{W^{1,2}_x}+1\Bigr].
\end{align*}
and again for $q\in \N$, by multiplication with $\|u_m\|_{W^{1,2}_x}^{2^{q}-2}$ before taking expectations,
 \begin{align}\label{eq:0807e}
 \begin{aligned}
\E\bigg[\max_{1\leq n\leq M}&\|\bfu_{n}\|^{2^q}_{W^{1,2}_x}+\tau\sum_{n=1}^M\|\bfu_{n}\|^{2^{q}-2}_{W^{1,2}_x} \|\Delta\bfu_n\|^2_{L^2_x}\bigg]\leq
\,c\,\E\Bigl[\|\bfu_{0}\|^{2^q}_{W^{1,2}_x}+1\Bigr],
\end{aligned}
\end{align}

Now we come to the proof of \eqref{lem:3.1c}. We formally test 
the equation with $\Delta^2 u_m$. For the nonlinear term we obtain ($\kappa >0$)
\begin{align*}
-\tau(f(\bfu_m),\Delta^2\bfu_m)_{L^2_x}&=\tau(f'(\bfu_m)\nabla\bfu_m,\Delta\nabla\bfu_m)_{L^2_x}\\
&\leq\tau\|f'(u_m)\|_{L^6_x}\|\nabla u_m\|_{L^3_x}\|\nabla\Delta u_m\|_{L^2_x}\\
&\leq\,c\tau(1+\|u_m\|_{L^{6}_x})\|u_m\|_{L^\infty_x}\|\nabla u_m\|_{L^2_x}^{\frac{1}{2}}\|\nabla^2 u_m\|_{L^2_x}^{\frac{1}{2}}\|\nabla\Delta u_m\|_{L^2_x}\\
&\leq\,c\tau(1+\|\nabla u_m\|^2_{L^{2}_x})\|\nabla^2 u_m\|_{L^2_x}\|\nabla\Delta u_m\|_{L^2_x}\\
&\leq\tau\,c(\kappa)(1+\|\nabla u_m\|^4_{L^{2}_x})\|\nabla^2 u_m\|_{L^2_x}^2+\kappa\|\nabla\Delta u_m\|^2_{L^2_x}.
\end{align*}
Here the second term can be absorbed for $\kappa\ll1$, whereas the first one (in expectation and summed form) can be controlled by \eqref{eq:0807e} (with $q=3$).
 If $d=3$ we use that $\Phi$ is assumed to be affine linear in $u$, \emph{cf.}~assumption \ref{N2}. In this situation the stochastic terms can be estimated exactly as in the proof of \eqref{eq:0807d}.
 If $d=1,2$ this problem can be overcome by Ladyshenskaya's inequality and \ref{N1}. In particular, we have by \eqref{eq:phi2b}
 \begin{align*}
 \E\bigg[\tau\sum_{n=1}^M\|\Phi(\bfu_{n-1})\|_{L_2(\mathfrak U;W^{2,2}_x)}^2\bigg]&\leq\,c\, \E\bigg[\tau\sum_{n=1}^M\big(1+\|\bfu_{n-1}\|_{W^{1,4}_x}^4+\|\bfu_{n-1}\|^2_{W^{2,2}_x}\big)\bigg]\\
 &\leq\,c\, \E\bigg[\tau\sum_{n=1}^M\big(1+\|\bfu_{n-1}\|_{W^{1,2}_x}^2\|\bfu_{n-1}\|_{W^{2,2}_x}^2+\|\bfu_{n-1}\|^2_{W^{2,2}_x}\big)\bigg],
 \end{align*}
 which is uniformly bounded by \eqref{lem:3.1c} with $q=2$.
This settles the proof of \eqref{lem:3.1c} for $q=1$.
An inductive argument may now be employed to complete the proof for (\ref{lem:3.1c}) for $q\in\N$.

Ad~{\bf (c)}. The proof works along the same lines testing with $\Delta^3\bfu_m$ and estimating the stochastic terms by means of \eqref{eq:phismooth}.
For the nonlinear term we obtain
\begin{align*}
-\tau( &f(\bfu_m),\Delta^3\bfu_m)_{L^2_x}=-\tau(\mathrm{div}(f'(\bfu_m)\nabla\bfu_m),\Delta^2\bfu_m)_{L^2_x}\\
&=-\tau(f''(\bfu_m)|\nabla\bfu_m|^2,\Delta^2\bfu_m)_{L^2_x}-\tau(f'(\bfu_m)\Delta\bfu_m,\Delta^2\bfu_m)_{L^2_x}\\
&\leq\tau\int_{\mt}|\bfu_m| |\nabla\bfu_m|^2|\Delta^2\bfu_m|\dx+\tau\int_{\mt}(|\bfu_m|^2+1) |\Delta\bfu_m||\Delta^2\bfu_m|\dx
\\&\leq\,c(\kappa)\tau\|\bfu_m\|_{W^{2,2}_x}^6+\kappa\tau\|\Delta^2\bfu_m\|^2_{L^2_x}.
\end{align*}
The second term on the right-hand side can be absorbed provided we choose $\kappa>0$ small enough. The first one (summed with respect to $m$) ins bounded by \eqref{lem:3.1c}.
\end{proof}

%

\subsection{The Kolmogorov equation}\label{sec:kol}
We set
\begin{align*}
\UU(t,\bfh):=\E \big[\varphi(\bfu(t,\bfh))\big],
\end{align*}
where $\bfu(t,\bfh)$ is the solution to \eqref{eq:SNS} with $\bfu_0=\bfh$ and $\varphi\in C_b^1(L^2_x)$. It is well-known that $\UU$ solves
\begin{align}\label{eq:Kolm'}
\partial_t \UU(t,\bfh)=\frac{1}{2}\tr\big(\Phi(\bfh)\Phi^\ast(\bfh) D^2\UU(t)\big)+\big(\mathcal A\bfh+f(\bfh),D\UU(t,\bfh)\big),\quad \UU(0,\bfh)=\varphi(\bfh).
\end{align}
 In the following we derive some estimates for $\UU$. The proofs are only formal and can be made rigorous by considering a finite-dimensional Galerkin approximation\footnote{Though some care is required for Lemma \ref{lem:kol1}.} of \eqref{eq:SNS}
(leading to a finite-dimensional Kolmogorov equation approximating \eqref{eq:Kolm'}) and establishing estimates which are uniform with respect to the Galerkin parameter. Such a procedure is technical and tedious but standard in literature. We refer to \cite{Ce} and \cite[Chapter 9]{PrZa} for a detailed analysis of Kolmogorov equations.

A crucial ingredient in our proof is the monotonicity of the leading term in $f(u)=u^3-u$. It is currently an open problem to obtain similar results for semilinear equations without this property such as the 2D Navier--Stokes equations. 
Since we allow more regularity for the solution to \eqref{eq:SNS} through more regular data when compared to previous papers, we only require estimates in $L^2_x$ for the weak error analysis. For example, the estimates in \cite{BrDe,BrGo,De} are given in fractional Sobolev spaces (with differentiability strictly smaller than 1/2) for the derivatives of $\UU$ are proved instead; morover, the situation here is more complicated than that in \cite{BrDe,De} due to the non-Lipschitz nonlinearity in \eqref{eq:SNS}. In \cite{BrGo} the Kolmogorov equation for \eqref{eq:SNS} in 1D is considered (with additive space-time white noise), while we study \eqref{eq:SNS} with smooth multiplicative noise in cases $d=1,2,3$ in the following.
\begin{lemma}\label{lem:kol1}
Let $\varphi\in C^1_b(L^p(\mt))$ for some $p\in(1,\infty)$ and suppose that \ref{N1} {\bf (a)} holds. Then we have for $t\in(0,T)$ and $\bfh\in L^p(\mt)$
\begin{align*}
\|D\UU(t,\bfh)\|_{L^{p}_x}\leq\,c(p,\varphi,T).
\end{align*}
\end{lemma}
\begin{proof}
We proceed formally. A rigorous proof can be obtained as follows:
\begin{itemize}
\item We first proof the result for $p=2$ by means of a Galerkin approximation and pass to the limit obtaining well-defined infinite-dimensional objects.
\item In particular, we obtain a variational solution to \eqref{eq:3003} below to which we can apply It\^{o}'s formula. This can be justified by truncating the function $z\mapsto z^p$ and applying the It\^o-formula in Hilbert spaces from \cite[Theorem 4.17]{PrZa}.
\end{itemize}
Differentiating $\UU$ with respect to $\bfh$ in direction $\bfg\in L^{p'}(\mt)$ with $p':=p/(p-1)$ yields
\begin{align}\label{eq:DU}
D\UU(t,\bfh)\cdot\bfg=\E\big[D\varphi(\bfu(t))\, \bfeta^{\bfh,\bfg}(t)\big],
\end{align}
where $\bfeta^{\bfh,\bfg}$ solves
\begin{align}\label{eq:3003}
\dd \bfeta^{\bfh,\bfg}=\big(\Delta \bfeta^{\bfh,\bfg}-f'(\bfu)\,\bfeta^{\bfh,\bfg}\big)\dt+D\Phi(\bfu)\bfeta^{\bfh,\bfg}\,\dd W,\qquad \bfeta^{\bfh,\bfg}(0)=\bfg.
\end{align}
Since $D\varphi$ is bounded we have
\begin{align}\label{eq:1604}
|D\UU(t,\bfh)\cdot\bfg|\leq\Big(\E\big[\|\bfeta^{\bfh,\bfg}(t)\|_{L^{p'}_x}^{p'}\big]\Big)^{1/p'}.
\end{align}
Applying It\^{o}'s formula to \eqref{eq:3003} yields
\begin{align*}
\tfrac{1}{p'}\|\bfeta^{\bfh,\bfg}(t)\|_{L^{p'}_x}^{p'}&+\int_0^t\|\nabla|\bfeta^{\bfh,\bfg}(s)|^{p'/2}\|_{L^2_x}^2\ds\\&=\tfrac{1}{p'}\|\bfg\|_{L^{p'}_x}^{p'}-\int_0^t\int_{\mt}f'(u(s))|\bfeta^{\bfh,\bfg}(s)|^{p'}\dxs\\
&+\int_{\mt}\int_0^t|\bfeta^{\bfh,\bfg}|^{p'-2}\bfeta^{\bfh,\bfg}D\Phi(\bfu)\bfeta^{\bfh,\bfg}\,\dd W\dx\\&+\tfrac{p'-1}{2}\sum_{i\geq1}\int_0^t\int_{\mt}|\bfeta^{\bfh,\bfg}(s)|^{p'-2}|D\Phi(\bfu(s))\bfeta^{\bfh,\bfg}(s)e_i|^2\dxs.
\end{align*}
Since $f'\geq -1$, the second term on the right-hand side is clearly bounded by $\int_0^t\|\bfeta^{\bfh,\bfg}(t)\|_{L^{p'}_x}^{p'}\ds$, while the third one vanishes under the expectation. Using \ref{N1}, the same bound follows for the last term. We conclude
\begin{align*}
\E\Big[\|\bfeta^{\bfh,\bfg}(t)\|_{L^{p'}_x}^{p'}\Big]+\E\bigg[\int_0^t\|\nabla|\bfeta^{\bfh,\bfg}(s)|^{p'/2}\|_{L^2_x}^2\ds\bigg]\leq \|\bfg\|_{L^{p'}_x}^{p'}+c\int_0^t\E\Big[\|\bfeta^{\bfh,\bfg}(s)\|_{L^{p'}_x}^{p'}\Big]\ds
\end{align*}
such that, by Gronwall's lemma,
\begin{align}\label{eq:1102a}
\E\Big[\|\bfeta^{\bfh,\bfg}(t)\|_{L^{p'}_x}^{p'}\Big]+\E\bigg[\int_0^t\|\nabla|\bfeta^{\bfh,\bfg}(s)|^{p'/2}\|_{L^2_x}^2\ds\bigg]\leq c\|\bfg\|_{L^{p'}_x}^{p'}.
\end{align}
Plugging this into \eqref{eq:1604} and taking the supremum with respect to $\bfg$ proves the claim.
\end{proof}
Note that it is easy to generalise the argument above to obtain
 \begin{align}\label{eq:1102a}
\E\Big[\|\bfeta^{\bfh,\bfg}(t)\|_{L^{p'}_x}^q\Big]+\E\bigg[\bigg(\int_0^t\|\nabla|\bfeta^{\bfh,\bfg}(s)|^{p'/2}\|_{L^2_x}^2\ds\bigg)^{\frac{q}{p'}}\bigg]\leq\,c_q\|\bfg\|^q_{L^{p'}_x}
\end{align}
for all $q\geq p'$. For that purpose it is sufficient to argue as before
and estimate the stochastic integral by means of the Burkholder-Davis-Gundy inequality and \eqref{eq:phi}. By a parabolic interpolation \eqref{eq:1102a} (with $p'=2$) yields
 \begin{align}\label{eq:1102}
\E\Big[\|\bfeta^{\bfh,\bfg}(t)\|_{L^{10/3}(0,T;L^{10/3}_{x})}^q\Big]\leq\,c_q\|\bfg\|^q_{L^{p'}_x}
\end{align}
for all $q\geq 2$.

\begin{lemma}\label{lem:kol2}
Let $\varphi\in C^2_b(L^2(\mt))$ and suppose that \eqref{N1} holds. Then we have for $t\in(0,T)$, $\bfh,\bfg\in L^2(\mt)$ and $\bfk\in L^6(\mt)$
\begin{align*}
|D^2\UU(t,\bfh)(\bfg,\bfk)|\leq\,c(\varphi,T)\|\bfg\|_{L^2_x}\|\bfk\|_{L^6_x}.
\end{align*}
\end{lemma}
\begin{proof}
Differentiating \eqref{eq:DU} again with respect to $\bfh$ (in direction $\bfk\in L^2(\mt)$) gives
\begin{align}\label{eq:D2U}
D^2\UU(t,\bfh)(\bfg,\bfk)=\E[D\varphi(\bfu(t,\bfh))\, \bfzeta^{\bfh,\bfg,\bfk}(t)]+\E[D^2\varphi(\bfu(t,\bfh))(\bfeta^{\bfh,\bfg}(t),\bfeta^{\bfh,\bfk}(t))],
\end{align}
where $\bfzeta^{\bfh,\bfg,\bfk}$ solves
\begin{align}\label{eq:3003'}
\begin{aligned}
\dd \bfzeta^{\bfh,\bfg,\bfk}&=\big(\Delta \bfzeta^{\bfh,\bfg,\bfk}-6u\bfeta^{\bfh,\bfg}\bfeta^{\bfh,\bfk}-f'(u)\,\bfzeta^{\bfh,\bfg,\bfk}\big)\dt\\
&\quad+D^2\Phi(\bfu)(\bfeta^{\bfh,\bfg},\bfeta^{\bfh,\bfk})\,\dd W
+D\Phi(\bfu)\bfzeta^{\bfh,\bfg,\bfk}\,\dd W,
\end{aligned}
\end{align}
with $\bfzeta^{\bfh,\bfg,\bfk}(0)=0$, which can be seen from  differentiating \eqref{eq:3003}. Note that the second term in \eqref{eq:D2U} can be estimated by means of \eqref{eq:1102}. In order to estimate the first one,
we apply It\^{o}'s formula to \eqref{eq:3003'} yielding
\begin{align}\label{eq:3003''}
\begin{aligned}
\frac{1}{2}&\|\bfzeta^{\bfh,\bfg,\bfk}(t)\|_{L^2_x}^2+\int_0^t\|\nabla\bfzeta^{\bfh,\bfg,\bfk}(s)\|_{L^2_x}^2\ds\\&=-\int_0^t\int_{\mt}6u\,\bfeta^{\bfh,\bfg}\,\bfeta^{\bfh,\bfk}(s)\,\bfzeta^{\bfh,\bfg,\bfk}(s)\dxs-\int_0^t\int_{\mt}f'(\bfu(s))\,|\bfzeta^{\bfh,\bfg,\bfk}(s)|^2\dxs\\
&\quad+\int_{\mt}\int_0^t\bfzeta^{\bfh,\bfg,\bfk}D^2\Phi(\bfu)(\bfeta^{\bfg,\bfk},\bfeta^{\bfh,\bfk})\,\dd W\dx+\int_{\mt}\int_0^t\bfzeta^{\bfh,\bfg,\bfk}D\Phi(\bfu)\bfzeta^{\bfh,\bfg,\bfk}\,\dd W\dx\\&\quad+\frac{1}{2}\int_0^t\|D\Phi(\bfu)\bfzeta^{\bfh,\bfg,\bfk}(s)\|_{L_2(\mathfrak U;L^2_x)}^2\ds+\frac{1}{2}\int_0^t\|D^2\Phi(\bfu(s))(\bfeta^{\bfh,\bfg}(s),\bfeta^{\bfh,\bfk}(s))\|_{L_2(\mathfrak U;L^2_x)}^2\ds\\
&\quad+\sum_{i\geq1}\int_0^t\int_{\mt}D^2\Phi(\bfu)(\bfeta^{\bfh,\bfg}(s),\bfeta^{\bfh,\bfk})(s)e_i\,D\Phi(\bfu(s))\bfzeta^{\bfh,\bfg,\bfk}(s)e_i\dx\ds.
\end{aligned}
\end{align}
By H\"older's inequality we can bound the last line by the second last one, whereas the stochastic integrals have zero expectation and thus can be ignored. By \eqref{eq:phi} we have
\begin{align*}
\frac{1}{2}\int_0^t\|D\Phi(\bfu)\bfzeta^{\bfh,\bfg,\bfk}(s)\|_{L_2(\mathfrak U;L^2_x)}^2\ds\leq\,c\int_0^t\|\bfzeta^{\bfh,\bfg,\bfk}(s)\|_{L^2_x}^2\ds.
\end{align*}
Clearly, it also holds that
\begin{align*}
-\int_0^t\int_{\mt}f'(\bfu(s))\,|\bfzeta^{\bfh,\bfg,\bfk}(s)|^2\dxs\leq\,\int_0^t\|\bfzeta^{\bfh,\bfg,\bfk}(s)\|_{L^2_x}^2\ds.
\end{align*}
The remaining two terms are more complicated. First of all,
\eqref{eq:phi2} yields
\begin{align*}
\frac{1}{2}\int_0^t&\|D^2\Phi(\bfu(s))(\bfeta^{\bfh,\bfg}(s),\bfeta^{\bfh,\bfk}(s))\|_{L_2(\mathfrak U;L^2_x)}^2\ds\\&\leq\,c\int_0^t\|\bfeta^{\bfh,\bfg}(s)\bfeta^{\bfh,\bfk}(s)\|^2_{L^2_x}\ds\\
&\leq\,c\int_0^t\Big(\|\bfeta^{\bfh,\bfg}(s)\|_{L^{10/3}_x}^{10/3}+\|\bfeta^{\bfh,\bfk}(s)\|_{L^5_x}^5\Big)\ds\\
&\leq\,c\int_0^t\Big(\|\bfeta^{\bfh,\bfg}(s)\|_{L^{10/3}_x}^{10/3}+\|\bfeta^{\bfh,\bfk}(s)\|_{L^6_x}^6+1\Big)\ds,
\end{align*}
the expectation of which can be controlled by $\|\bfg\|_{L^2_x}$ and $\|\bfk\|_{L^6_x}$ using \eqref{eq:1102a}.
The most critical term is
\begin{align*}
\int_0^t\int_{\mt}&u(s)\,\bfeta^{\bfh,\bfg}(s)\,\bfeta^{\bfh,\bfk}\,\bfzeta^{\bfh,\bfg,\bfk}(s)\dxs\\&=\int_0^t\int_{\mt}\A^{-1/2}\big(\bfu(s)\,\bfeta^{\bfh,\bfg}(s)\,\bfeta^{\bfh,\bfk}(s)\big)\,\A^{1/2}\bfzeta^{\bfh,\bfg,\bfk}(s)\dxs\\
&\leq\,\kappa\int_0^t\|\nabla\bfzeta^{\bfh,\bfg,\bfk}(s)\|^2_{L^2_x}\ds+c(\kappa)\int_0^t\|\bfu(s)\,\bfeta^{\bfh,\bfg}(s)\,\bfeta^{\bfh,\bfk}(s)\|_{W^{-1,2}}^2\ds.
\end{align*}
We can absorb the first term, whereas using the embedding $L^{6/5}_x\hookrightarrow W^{-1,2}_x$ the second one is bounded by
\begin{align*}
\int_0^t\|\bfu(s)&\,\bfeta^{\bfh,\bfg}(s)\,\bfeta^{\bfh,\bfk}(s)\|_{L^{6/5}_x}^2\ds\\&\leq \int_0^t\|\bfeta^{\bfh,\bfk}\|_{L^{2}_x}^2\|\bfeta^{\bfh,\bfg}(s)\|_{L^{6}_x}^2\|\bfu(s)\|_{L^{6}_x}^2\ds\\
&\leq \int_0^t\|\bfeta^{\bfh,\bfk}(s)\|_{L^{2}_x}^6\ds+\int_0^t\|\bfeta^{\bfh,\bfg}(s)\|_{L^{6}_x}^6\ds+\int_0^t\Big(1+\|\bfu(s)\|_{L^{6}_x}^{6}\Big).\ds\end{align*}
Hence we obtain from \eqref{lem:} (with $p=6$) and \eqref{eq:1102}
\begin{align*}
\E\bigg[\int_0^t\|\bfu(s)\,\bfeta^{\bfh,\bfg}(s)\,\bfeta^{\bfh,\bfk}(s)\|_{L^{6/5}_x}^2\ds\bigg]
&\leq \,c\Big(\|\bfk\|_{L^{2}_x}^6+\|\bfg\|_{L^{6}_x}^6+1 \Big).
\end{align*}
We conclude for \eqref{eq:3003''} that
 \begin{align}\label{eq:1505}
\E\big[\|\bfzeta^{\bfh,\bfg,\bfk}(t)\|_{L^2_x}^2\big]+\E\bigg[\int_0^t\|\nabla\bfzeta^{\bfh,\bfg,\bfk}(s)\|_{L^2_x}^2\ds\bigg]\leq \,c\Big(\|\bfg\|_{L^{2}_x},\|\bfk\|_{L^{6}_x}\Big),
\end{align}
which finishes the proof by bilinearity of $D^2\UU(t,\bfh)$.
\end{proof}
Estimate \eqref{eq:1505} above can be strengthened to 
 \begin{align}\label{eq:1505b}
\E\big[\|\bfzeta^{\bfh,\bfg,\bfk}(t)\|_{L^2_x}^q\big]+\E\bigg[\int_0^t\|\nabla\bfzeta^{\bfh,\bfg,\bfk}(s)\|_{L^2_x}^2\ds\bigg]^{\frac{q}{2}}\leq \,c\Big(q,\|\bfg\|_{L^{2}_x},\|\bfk\|_{L^{6}_x}\Big)
\end{align}
for any $q\geq 2$: By
taking the power $q/2$ in \eqref{eq:3003''}, the only difference is that the stochastic integrals do not vanish and must be estimated. Using Burkholder-Davis-Gundy inequality and \eqref{eq:1102} we obtain the upper bounds (assuming that $q\geq4$)
\begin{align*}
\E\bigg[&\bigg(\int_0^T\sum_{i\geq 1}\bigg(\int_{\mt}\zeta^{\bfh,\bfg,\bfk}(s)D^2\Phi(u(s))e_i(\eta^{\bfg,\bfk}(s),\eta^{\bfg,\bfk}(s))\dx\bigg)^2\ds\bigg)^{\frac{q}{4}}\bigg]\\
&\leq \,c\,\E\bigg[\bigg(\int_0^T\|\zeta^{\bfh,\bfg,\bfk}(s)\|_{L^{10/3}_x}^2\|\bfeta^{\bfh,\bfg}(s)\|_{L^2_x}^2\|\bfeta^{\bfh,\bfk}(s)\|_{L^5_x}^2\ds\bigg)^{\frac{q}{4}}\bigg]\\
&\leq \,\kappa\,\E\bigg[\bigg(\int_0^T\|\zeta^{\bfh,\bfg,\bfk}(s)\|_{L^{10/3}_x}^{10/3}\ds\bigg)^{\frac{q}{4}}\bigg]+c(\kappa)\,\E\bigg[\bigg(\int_0^T\big(\|\bfeta^{\bfh,\bfg}(s)\|_{L^2_x}^{10}+\|\bfeta^{\bfh,\bfk}(s)\|_{L^6_x}^{10}\big)\ds\bigg)^{\frac{q}{4}}\bigg]\\
&\leq\,c\kappa\,\E\big[\|\bfzeta^{\bfh,\bfg,\bfk}(s)\|_{L^{10/3}(0,T;L^{10/3}_{x})}^q\big]+\,c\big(\kappa,\|\bfg\|_{L^2_x},\|\bfk\|_{L^6_x}\big),
\end{align*}
as well as
\begin{align*}
\E\bigg[\bigg(\int_0^T&\sum_{i\geq 1}\bigg(\int_{\mt}\zeta^{\bfh,\bfg,\bfk}(s)D\Phi(u(s))e_i\zeta^{\bfh,\bfg,\bfk}(s)\dx\bigg)^2\ds\bigg)^{\frac{q}{4}}\bigg]\\&\leq \,c\,\E\bigg[\bigg(\int_0^T\|\bfzeta^{\bfh,\bfg,\bfk}(s)\|_{L^2_x}^4\ds\bigg)^{\frac{q}{4}}\bigg]\leq \,c\,\E\bigg[\int_0^T\|\bfzeta^{\bfh,\bfg,\bfk}(s)\|_{L^2_x}^q\ds\bigg],
\end{align*}
which can be handled by Gronwall's lemma. By a parabolic interpolation we obtain \eqref{eq:1505b} as well as
 \begin{align}\label{eq:1505c}
\E\Big[\|\bfzeta^{\bfh,\bfg,\bfk}\|_{L^{10/3}(0,T;L^{10/3}_{x})}^{q}\Big]\leq \,c\big(q,\|\bfg\|_{L^{2}_x},\|\bfk\|_{L^{6}_x}\big)
\end{align}
for any $q\geq 2$.

\begin{lemma}\label{lem:kol3}
Let $\varphi\in C^3_b(L^2(\mt))$ and suppose that \ref{N1} holds. Then we have for $t\in(0,T)$ and $\bfh\in L^2(\mt)$ and $\bfk,\bfg,\bfl\in L^6(\mt)$
\begin{align*}
|D^3\UU(t,\bfh)(\bfg,\bfk,\bfl)|\leq\,c(\varphi,T)\|\bfk\|_{L^6_x}\|\bfg\|_{L^6_x}\|\bfl\|_{L^6_x}.
\end{align*}
\end{lemma}
\begin{proof}
Differentiating \eqref{eq:D2U} again with respect to $\bfh$ (in direction $\bfl\in L^2(\mt)$) gives
\begin{align}\label{eq:D3U}
\begin{aligned}
D^3\UU(t,\bfh)(\bfg,\bfk,\bfl)&=\E[D\varphi(\bfu(t,\bfh))\, \bfxi^{\bfh,\bfg,\bfk,\bfl}(t)]+\E[D^2\varphi(\bfu(t,\bfh))(\bfzeta^{\bfh,\bfg,\bfl}(t),\bfeta^{\bfh,\bfk}(t))]\\&+\E[D^2\varphi(\bfu(t,\bfh))(\bfeta^{\bfh,\bfg}(t),\bfzeta^{\bfh,\bfk,\bfl}(t))]\\&+\E[D^3\varphi(\bfu(t,\bfh))(\bfeta^{\bfh,\bfg}(t),\bfeta^{\bfh,\bfk}(t),\bfeta^{\bfh,\bfl}(t))],
\end{aligned}
\end{align}
where (by differentiating \eqref{eq:3003'})
\begin{align}\label{eq:3003'b}
\begin{aligned}
\dd \bfxi^{\bfh,\bfg,\bfk,\bfl}&=\big(\Delta \bfxi^{\bfh,\bfg,\bfk,\bfl}-6\bfeta^{\bfh,\bfl}\bfeta^{\bfh,\bfg}\bfeta^{\bfh,\bfk}-6\bfu\bfzeta^{\bfh,\bfg,\bfl}\bfeta^{\bfh,\bfk}-6\bfu\bfeta^{\bfh,\bfg}\bfzeta^{\bfh,\bfk,\bfl}\big)\dt\\
&-\big(f'(\bfu)\,\bfxi^{\bfh,\bfg,\bfk,\bfl}+6u\bfeta^{\bfh,\bfl}\bfzeta^{\bfh,\bfg,\bfk}\big)\dt\\
&+D^3\Phi(\bfu)(\bfeta^{\bfh,\bfl},\bfeta^{\bfh,\bfg},\bfeta^{\bfh,\bfk})\,\dd W+D^2\Phi(\bfu)(\bfeta^{\bfh,\bfl},\bfzeta^{\bfh,\bfg,\bfk}\,)\dd W
+D\Phi(\bfu)\bfxi^{\bfh,\bfg,\bfk,\bfl}\,\dd W\\
&+D^2\Phi(\bfu)(\bfzeta^{\bfh,\bfg,\bfl},\bfeta^{\bfh,\bfk})\,\dd W+D^2\Phi(\bfu)(\bfeta^{\bfh,\bfg},\bfzeta^{\bfh,\bfk,\bfl})\,\dd W,
\end{aligned}
\end{align}
with $\bfxi^{\bfh,\bfg,\bfk,\bfl}(0)=0$. The last three terms in \eqref{eq:D3U} can be estimated by means of \eqref{eq:1102} and \eqref{eq:1505b}. So, our focus is on the first one.
We apply now It\^{o}'s formula and argue similarly to the proof of Lemmas \ref{lem:kol1} and \ref{lem:kol2}. The correction terms can be handled as there using \eqref{eq:phi}--\eqref{eq:phi3}. 
For instance, we have
\begin{align*}
\E\bigg[\int_0^t&\|D^2\Phi(\bfu(s))(\bfzeta^{\bfh,\bfg,\bfl}(s),\bfeta^{\bfh,\bfk}(s))\|_{L_2(\mathfrak U;L^2_x)}^2\ds\bigg]\\&\leq\E\bigg[\int_0^t\|\bfzeta^{\bfh,\bfg,\bfl}(s)\bfeta^{\bfh,\bfk}(s)\|^2_{L^2_x}\ds\bigg]\\
&\leq\E\bigg[\int_0^t\Big(\|\bfzeta^{\bfh,\bfg,\bfl}(s)\|_{L^{10/3}_x}^{10/3}+\|\bfeta^{\bfh,\bfk}(s)\|_{L^5_x}^5\Big)\ds\bigg]\\
&\leq\E\bigg[\int_0^t\Big(\|\bfzeta^{\bfh,\bfg,\bfl}(s)\|_{L^{10/3}_x}^{10/3}+\|\bfeta^{\bfh,\bfk}(s)\|^6_{L^6_x}+1\Big)\ds\bigg]\\
&\leq\,c\big(q,\|\bfg\|_{L^{2}_x},\|\bfk\|_{L^{6}_x},\|\bfl\|_{L^{6}_x}\big)
\end{align*}
using \eqref{eq:1102} and \eqref{eq:1505c}.
We clearly have
\begin{align*}
-\int_0^t\int_{\mt}f'(u)|\bfxi^{\bfh,\bfg,\bfk,\bfl}(s)|^2\dxs\leq\int_0^t\|\bfxi^{\bfh,\bfg,\bfk,\bfl}(s)\|^2_{L^2_x}\ds.
\end{align*}
Let us now focus on the remaining terms arising from the nonlinearity being more critical. It holds
\begin{align*}
-\int_0^t\int_{\mt}&6\bfeta^{\bfh,\bfl}(s)\bfeta^{\bfh,\bfg}(s)\bfeta^{\bfh,\bfk}(s)\bfxi^{\bfh,\bfg,\bfk,\bfl}(s)\dxs\\&\leq\,c(\kappa)\int_0^t\|\bfeta^{\bfh,\bfl}(s)\bfeta^{\bfh,\bfg}(s)\bfeta^{\bfh,\bfk}(s)\|_{W^{-1,2}_x}^2\ds+\kappa\int_0^t\|\bfxi^{\bfh,\bfg,\bfk,\bfl}(s)\|^2_{W^{1,2}_x}\ds,\\
-\int_0^t\int_{\mt}&6\bfu(s)\bfzeta^{\bfh,\bfg,\bfl}(s)\bfeta^{\bfh,\bfk}(s)\bfxi^{\bfh,\bfg,\bfk,\bfl}(s)\dxs\\&\leq\,c(\kappa)\int_0^t\|6u(s)\bfzeta^{\bfh,\bfg,\bfl}(s)\bfeta^{\bfh,\bfk}(s)\|_{W^{-1,2}_x}^2\ds+\kappa\int_0^t\|\bfxi^{\bfh,\bfg,\bfk,\bfl}(s)\|^2_{W^{1,2}_x}\ds,\\
-\int_0^t\int_{\mt}&6\bfu(s)\bfeta^{\bfh,\bfg}\bfzeta^{\bfh,\bfk,\bfl}(s)\bfxi^{\bfh,\bfg,\bfk,\bfl}(s)\dxs\\&\leq\,c(\kappa)\int_0^t\|6u(s)\bfeta^{\bfh,\bfg}(s)\bfzeta^{\bfh,\bfk,\bfl}(s)\|_{W^{-1,2}_x}^2\ds+\kappa\int_0^t\|\bfxi^{\bfh,\bfg,\bfk,\bfl}(s)\|^2_{W^{1,2}_x}\ds,\\
\int_0^t\int_{\mt}&6u(s)\bfeta^{\bfh,\bfl}(s)\bfzeta^{\bfh,\bfg,\bfk}(s)\bfxi^{\bfh,\bfg,\bfk,\bfl}(s)\dxs\\&\leq\,c(\kappa)\int_0^t\|u(s)\bfeta^{\bfh,\bfl}(s)\bfzeta^{\bfh,\bfg,\bfl}(s)\|_{W^{-1,2}_x}^2\ds+\kappa\int_0^t\|\bfxi^{\bfh,\bfg,\bfk,\bfl}(s)\|^2_{W^{1,2}_x}\ds.
\end{align*}
Estimating the $W^{-1,2}_x$-norm by the $L^2_x$-norm and applying H\"older's inequality, the $c(\kappa)$ terms are controlled (line by line) by the terms
\begin{align*}
&\int_0^t\|\bfeta^{\bfh,\bfl}(s)\|_{L^6_x}^6\ds+\int_{0}^t\|\eta^{\bfh,\bfg}(s)\|_{L^6_x}^6\ds+\int_{0}^t\|\eta^{\bfh,\bfk}(s)\|_{L^6_x}^6\ds,\\
&\int_0^t\|u(s)\|_{L^{6}_x}^{6}\ds+\int_{0}^t\|\eta^{\bfh,\bfl}(s)\|_{L^6_x}^6\ds+\int_{0}^t\|\zeta^{\bfh,\bfg,\bfl}(s)\|_{L^2_x}^6\ds,\\
&\int_0^t\|u(s)\|_{L^{6}_x}^{6}\ds+\int_{0}^t\|\eta^{\bfh,\bfg}(s)\|_{L^6_x}^6\ds+\int_{0}^t\|\zeta^{\bfh,\bfk,\bfl}(s)\|_{L^2_x}^6\ds,\\
&\int_0^t\|u(s)\|_{L^{6}_x}^{6}\ds+\int_{0}^t\|\eta^{\bfh,\bfl}(s)\|_{L^6_x}^6\ds+\int_{0}^t\|\zeta^{\bfh,\bfg,\bfl}(s)\|_{L^2_x}^6\ds.
\end{align*}
Using \eqref{eq:1102} and \eqref{eq:1505b} their expectations are bounded by the $L^6_x$-norms of $\bfg,\bfk$ and $\bfl$. 
\end{proof}


\section{Weak first Order Convergence Rate}\label{sec:error}
This section is the heart of the paper and is dedicated to the proof of following theorem, which establishes an optimal weak error rate for the time discretisation \eqref{tdiscrA}.
\begin{theorem}\label{thm:main}
 Let $(\Omega,\mf,(\mf_t)_{t\geq0},\prst)$ be a given stochastic basis with a complete right-con\-ti\-nuous filtration and an $(\mf_t)$-cylindrical Wiener process $W$.
Suppose that $\bfu_0\in W^{3,2}(\mt)$ and that $\Phi$ satisfies \ref{N1} if $d=1,2$, and \ref{N2} if $d=3$. Let $\bfu$ be the unique pathwise solution to \eqref{eq:SNS} in the sense of Definition \eqref{def:inc2d}
and let $(\bfu_m)_{m=1}^M$ be the solution
to \eqref{tdiscrA}. For any $\varphi\in C^2_b(L^2(\mt))$ we have
\begin{align}\label{eq:main}
\big|\E\big[\varphi(\bfu(T))-\varphi({\bfu}_M)\big]\big|\leq\,c\tau,
\end{align}
where $c>0$ depends on $\varphi,\bfu_0,T$ and $\Phi$.
\end{theorem}
\begin{remark}\label{remark:noise'}
{\bf 1.}~For $d=1$ and additive {\em white} noise in (\ref{eq:SNS}), a related result has been obtained in \cite[Theorem 3.3]{BrGo} for the implementable splitting scheme 
\cite[(1.1)$_2$]{BrGo}, using bounds for related iterates from \cite[Proposition 3]{BG1} in {\em supremum norm},  whose derivation used a Gronwall argument. The result is obtained under a compatibility condition for white noise and drift operator \cite[Section 2.1.2]{BrGo} whose physical interpretation is not immediate, and crucially exploits the {\em additive} character of noise to accomplish the estimate right before
\cite[Section 2.1.2]{BrGo}. The key part of the weak error analysis then is based on
the Kolmogorov equation \cite[(2.11)]{BrGo}, where the appearing ${\mathscr L}^{(\Delta t)}$ is generator of the semigroup generated by the regularized problem \cite[(2.7)]{BrGo} with solution $X^{(\Delta t)}$ --- rather than for (\ref{eq:SNS}) as we do here; note that this modification also gives rise to a modified energy ${\mathcal E}^{(\Delta t)}$ for the scalar order parameter $u$, if compared to the Helmholtz free energy functional ${\mathcal E}$ in (\ref{ac-1}).
---
In this work, we heavily profited from the {\em tools given in the error analysis} in \cite{BrGo}, as well as \cite{DePr,De,BrDe}.

{\bf 2.} The estimates for the solution of the Kolmogorov equation from Section \ref{sec:kol} are somewhat weaker than those used for the weak error analysis in \cite{BrDe,BrGo,DePr,De} (see also \cite{B,CH,CHS1} for related results). We compensate this by assuming more regularity for the data (initial datum and noise) which results in the higher order estimates for the time-discrete solution from  Section \ref{sec:esttau}. Our analysis in the additive case in Section \ref{sec:add} is also based on higher spatial regularity of the diffusion coefficient.

{\bf 3.}
The reason for the restriction to affine linear noise in Theorem \ref{thm:main} for $d=3$ is only due to
estimate \eqref{lem:3.1c} in Lemma \ref{lemma:3.1} which is heavily used. Apart from that the proof of Theorem \ref{thm:main} does not require this restriction. In particular,
in the 2D case the result also hold for nonlinear noise.

{\bf 4.}
A large part of the analysis in this paper directly extends to the Dirichlet case -- at least if the underlying domain is sufficiently smooth. Only the results from Section  \ref{sec:Ttau} require some additional technical effort: as we are working with nonlinear test functions it is not possible to justify the estimates from Lemma \ref{lem:T} by means of a Galerkin approximation (using eigenfunctions of the Laplace operator) anymore. Instead one has to prove localised estimates then by means of cut-off functions. In order to obtain global estimates one has to parametrise the boundary with local charts, change the coordinates and reflect the transformed solution at the hence obtained flat boundary. The local estimates can then be applied to the reflected transformed solution.
Such a procedure is tedious and technical but standard in literature. A detailed presentation can be found in detail, \emph{e.g.}, in \cite[Section 4]{BCDS}.

{\bf 5.} Large parts of the analysis of this section extend to the case of more general functions $f$ in equation \eqref{eq:SNS} with $q$-growth for some $q\geq2$. However, controlling the terms in \eqref{eq:qinN} requires that the leading part of $f$ is exactly of the form $f(z)=az^q$ with $a>0$ and $q\in\N$.
\end{remark}
The remainder of this section is dedicated to the proof of Theorem \ref{thm:main}, which is split into several subsections corresponding to the estimates of individual error terms.
 
 We start by decomposing the error in several parts which will be analysed in the subsequent subsections. Let $\bfu$ be the solution of \eqref{eq:SNS} and $(\bfu_m)_{m=1}^M$ be the solution to its time-discretisation
\eqref{tdiscrA}, both with respect to the initial datum $\bfu_0=\bfh\in L^2(\mt)$.
The error is
\begin{align*}
\E[\varphi(\bfu(T))]-\E[\varphi(\bfu_M)]=\UU(T,\bfh)-\E[\varphi(\bfu_M)],
\end{align*}
where $\UU$ is the solution to the Kolmogorov equation \eqref{eq:Kolm'} with initial condition $\varphi\in C^2_b(L^2(\mt))$.
 We decompose
\begin{align*}
\UU(T,\bfh)-\E[\varphi(\bfu_M)]&= \UU(T,\bfh)-\E[ \UU(0,\bfu_M)]\\&=\sum_{m=1}^{M}\Big[\UU(T-t_{m-1},\bfu_{m-1})-\UU(T-t_{m},\bfu_{m})\Big].
\end{align*}
Recalling the definition 
\begin{align}\label{eq:Utau}\bfU_\tau(t)=\frac{1}{\tau}\int_{t_{m-1}}^t\bfu_{m-1}\ds+\int_{t_{m-1}}^t\Phi(\bfu_{m-1})\,\dd W,\quad t\in[t_{m-1},t_m],\end{align}
we rewrite equation \eqref{tdiscr'} as \begin{align*}
\bfu_\tau(t)
&=\bfu_{m-1}+\int_{t_{m-1}}^t\mathcal S_{\tau}\A\bfu_{m-1}\ds-\int_{t_{m-1}}^t\mathcal S_{\tau}f(\bfu_{m-1})\ds\\
&\quad+\int_{t_{m-1}}^t\mathcal S_{\tau}f(\bfu_{m-1})\ds+\frac{1}{\tau}\int_{t_{m-1}}^t\big(D\mathcal T_{\tau}(\bfU_\tau)-\mathcal S_\tau\big)\bfu_{m-1}\ds\\
&\quad+\int_{t_{m-1}}^t D\mathcal T_\tau\big(\bfU_\tau\big)\Phi(\bfu_{m-1})\,\dd W\\
&\quad+\frac{1}{2}\sum_{i\geq 1}\int_{t_{m-1}}^t D^2\mathcal T_\tau(\bfU_\tau)\big(\Phi(\bfu_{m-1})e_i,\Phi(\bfu_{m-1})e_i\big)\ds
\end{align*}
We apply It\^o's formula (see \cite[Thm. 4.17]{PrZa}) to $\Psi(t,\bfu_{\tau})$
to get for $t\in[t_{m-1},t_m)$
\begin{align*}
\Psi&(t,\bfu_{\tau}(t))
=\Psi(t_{m-1},\bfu_{\tau}(t_{m-1}))+\int_{t_{m-1}}^{t}\partial_t \Psi(s,\bfu_{\tau})\ds\\&+\int_{t_{m-1}}^t\big(\mathcal S_{\tau}\A\bfu_{m-1},D\Psi(s,\bfu_{\tau})\big)_{L^2_x}\ds-\int_{t_{m-1}}^{t}\big(\mathcal S_{\tau}f(\bfu_{m-1}),D\Psi(s,\bfu_{\tau})\big)_{L^2}\ds\\
&+\int_{t_{m-1}}^t\Big( D\Psi(s,\bfu_{\tau}),D\mathcal T_\tau\big(\bfU_\tau\big)\Phi(\bfu_{m-1})\,\dd W\Big)_{L^2_x}\\
&+\frac{1}{2}\sum_{i\geq 1}\int_{t_{m-1}}^t\Big( D\Psi(s,\bfu_{\tau}),D^2\mathcal T_\tau\big(\bfU_\tau\big)\big(\Phi(\bfu_{m-1})e_i,\Phi(\bfu_{m-1})e_i\big)\Big)_{L^2_x}\ds\\
&+\int_{t_{m-1}}^{t_m}\Big(\mathcal S_{\tau}f(\bfu_{m-1})+\frac{1}{\tau}\big(D\mathcal T_{\tau}(\bfU_\tau)-\mathcal S_\tau\big)\bfu_{m-1},D\Psi(s,\bfu_{\tau})\Big)_{L^2_x}\ds\\
&+\tfrac{1}{2}\sum_{i\geq1}\int_{t_{m-1}}^tD^2\Psi(s,\bfu_{\tau})\Big(D\mathcal T_\tau\big(\bfU_\tau\big)\Phi(\bfu_{m-1})e_i,D\mathcal T_\tau\big(\bfU_\tau\big)\Phi(\bfu_{m-1})e_i\Big)\ds
\end{align*}
such that
\begin{align*}
\E[\Psi&(t_m,\bfu_{m})]
=\E[\Psi(t_{m-1},\bfu_{m-1})]+\E\bigg[\int_{t_{m-1}}^{t}\partial_t \Psi(s,\bfu_{\tau})\ds\bigg]\\&+\E\bigg[\int_{t_{m-1}}^{t_m}\big(\mathcal S_{\tau}\A\bfu_{m-1},D\Psi(s,\bfu_{\tau})\big)_{L^2_x}\ds-\E\int_{t_{m-1}}^{t_m}\big(\mathcal S_{\tau}f(\bfu_{m-1}),D\Psi(s,\bfu_{\tau})\big)_{L^2_x}\ds\bigg]\\\
&+\frac{1}{2}\sum_{i\geq 1}\E\bigg[\int_{t_{m-1}}^{t_m}\bigg( D\Psi(s,\bfu_{\tau}),D^2\mathcal T_\tau\big(\bfU_\tau\big)\big(\Phi(\bfu_{m-1})e_i,\Phi(\bfu_{m-1})e_i\big)\bigg)_{L^2_x}\ds\bigg]\\
&+\E\bigg[\int_{t_{m-1}}^{t_m}\Big(\mathcal S_{\tau}f(\bfu_{m-1})+\frac{1}{\tau}\big(D\mathcal T_{\tau}(\bfU_\tau)-\mathcal S_\tau\big)\bfu_{m-1},D\Psi(\sigma,\bfu_{\tau})\Big)_{L^2_x}\ds\bigg]\\
&+\tfrac{1}{2}\sum_{i\geq1}\E\bigg[\int_{t_{m-1}}^{t_m}D^2\Psi(s,\bfu_{\tau})\Big(D\mathcal T_\tau\big(\bfU_\tau\big)\Phi(\bfu_{m-1})e_i,D\mathcal T_\tau\big(\bfU_\tau\big)\Phi(\bfu_{m-1})e_i\Big)\ds\bigg].
\end{align*}
Setting $\Psi(t,\bfh)=\UU(T-t,\bfh)$ we obtain by \eqref{eq:Kolm'}
\begin{align*}
&\sum_{m=1}^{M}\E\Big[ \UU(T-t_{m-1},\bfu_{m-1})-\UU(T-t_{m},\bfu_{m})\Big]\\&=\UU(T,\bfh)-\E [\UU(T-\tau,\bfu_1)]+\sum_{m=2}^{M}\E\bigg[\int_{t_{m-1}}^{t_m}\big(\mathcal A\bfu_{\tau}-\mathcal S_{\tau}\A \bfu_{m-1},D\UU(T-s,\bfu_{\tau})\big)_{L^2_x}\ds\bigg]\\
&\quad+\sum_{m=2}^{M}\E\bigg[\int_{t_{m-1}}^{t_m}\big(\mathcal S_{\tau}f(\bfu_{m-1})-f(\bfu_{\tau}),D\UU(T-s,\bfu_{\tau})\big)_{L^2}\ds\bigg]\\
&\quad+\frac{1}{2}\sum_{m=2}^M\sum_{i\geq 1}\E\bigg[\int_{t_{m-1}}^{t_m}\bigg( D \UU(T-s,\bfu_{\tau}),D^2\mathcal T_\tau\big(\bfU_\tau\big)\big(\Phi(\bfu_{m-1})e_i,\Phi(\bfu_{m-1})e_i\big)\bigg)_{L^2_x}\ds\bigg]\\
&\quad+\sum_{m=2}^M\E\bigg[\int_{t_{m-1}}^{t_m}\Big(\mathcal S_{\tau}f(\bfu_{m-1})+\frac{1}{\tau}\big(D\mathcal T_{\tau}(\bfU_\tau)-\mathcal S_\tau\big)\bfu_{m-1},D\UU(T-s,\bfu_{\tau})\Big)_{L^2_x}\ds\bigg]\\
&\quad+\tfrac{1}{2}\sum_{m=2}^{M}\sum_{i\geq1}\E\bigg[\int_{t_{m-1}}^{t_m}\Big[D^2\UU(T-s,\bfu_{\tau})\Big(\Phi(\bfu_{\tau})e_i,\Phi(\bfu_{\tau})e_i\Big)\\&\qquad\qquad\qquad\qquad -D^2\UU(T-s,\bfu_{\tau})\Big(D\mathcal T_\tau\big(\bfU_\tau\big)\Phi(\bfu_{m-1})e_i,D\mathcal T_\tau\big(\bfU_\tau\big)\Phi(\bfu_{m-1})e_i\Big)\Big]\ds\bigg]\\
&=:(\mathrm{I})+(\mathrm{II})+(\mathrm{III})+(\mathrm{IV})+(\mathrm{V})+(\mathrm{VI}).
\end{align*}
We will estimate the terms $(\mathrm{I})$--$(\mathrm{V})$ in the following five subsections.

\subsection{The initial error $(\mathrm{I})$}\label{subsec:1}
Using $\UU(T,\bfh)=\E[\varphi(\bfu(T))]=\E [\UU(T-\tau,\bfu(\tau))]$ and Lemma \ref{lem:kol1} we have
\begin{align*}
|\UU(T,\bfh)-\E [\UU(T-\tau,\bfu_1)]|\leq c(\varphi)\E\Big[\|\bfu(\tau)-\bfu_1\|_{L^{2}_x}\Big],
\end{align*}
where, by \eqref{tdiscr'},
\begin{align}\label{eq:2009}
\begin{aligned}
\bfu(\tau)-\bfu_1&=(\mathcal S(\tau)-\mathcal S_{\tau})\bfh+\int_0^{\tau}\mathcal S(\tau-\sigma)f(\bfu)\ds-\frac{1}{\tau}\int_0^\tau(D\mathcal T_\tau(\bfU_\tau)-\mathcal S_{\tau})\bfh\ds\\
&\quad+\int_0^{\tau}\big(\mathcal S(\tau -\sigma)\Phi(\bfu)-D\mathcal T_{\tau}(\bfU_\tau)\Phi(\bfh)\big)\,\dd W\\
&\quad-\frac{1}{2}\sum_{i\geq 1}\int_{0}^{\tau} D^2\mathcal T_\tau\big(\bfU_\tau\big)\big(\Phi(\bfh)e_i,\Phi(\bfh)e_i\big)\ds
\end{aligned}
\end{align}
We estimate now the right-hand side term by term.
Using \eqref{eq:S5} (with $\beta=0$ and $r=1$) we obtain
\begin{align*}
\|(\mathcal S(\tau)-\mathcal S_{\tau})\bfh\|_{L^2_x}
\leq\,c\tau\|\bfh\|_{W^{1,2}_x}\leq\,c\tau\end{align*}
using the regularity of $\bfh$. 
We further obtain
\begin{align*}
\E\bigg[\bigg\|\int_0^{\tau}\mathcal S(\tau-\sigma)f(\bfu)\ds\bigg\|_{L^{2}_x}\bigg]
&\leq\,c\,\E\bigg[\int_0^\tau\|\mathcal S(\tau-\sigma)\|_{\mathcal L(L^2_x)} \big(1+\|u\|^3_{L^{6}_x}\big)\ds\bigg]\\
&\leq\,c\E\bigg[\int_0^\tau\big(1+\| \bfu\|_{L^{6}_x}^3\big)\ds\bigg]\leq\,c\tau
\end{align*}
by \eqref{eq:S2} and \eqref{lem:B1}. Furthermore, it holds by \eqref{eq:T-id} and \eqref{eq:S3}
\begin{align*}
\|(D\mathcal T_\tau(\bfU_\tau)-\mathcal S_{\tau})\bfh\|_{L^{2}_x}\leq\,c\tau(1+\|\bfh\|_{W^{2,2}_x})(1+\|\bfU_\tau\|_{W^{1,2}_x}).
\end{align*}
Hence we obtain by \eqref{eq:1102b}
\begin{align*}
\E\bigg[\bigg\|\frac{1}{\tau}\int_0^\tau(D\mathcal T_\tau(\bfU_\tau)-\mathcal S_{\tau})\bfh\ds\bigg\|_{L^{2}_x}\bigg]\leq\,c\tau.
\end{align*}
We write the stochastic term from \eqref{eq:2009} as
\begin{align*}
\int_0^\tau \big(\SSS(\tau-\sigma)\Phi(\bfu)&-D\mathcal T_{\tau}(\bfU_\tau)\Phi(\bfh)\big)\,\dd W\\&=\int_0^\tau \big(\SSS(\tau-\sigma)-\SSS_{\tau-\sigma}\big)\Phi(\bfu)\,\dd W
+\int_0^\tau \big(\SSS_{\tau-\sigma}-\SSS_\tau\big)\Phi(\bfu)\,\dd W\\
&\quad+\int_0^\tau \SSS_\tau\big(\Phi(\bfu)-\Phi(\bfh)\big)\,\dd W+\int_0^\tau \big(\SSS_\tau-D\mathcal T_{\tau}(\bfU_\tau)\big)\Phi(\bfh)\,\dd W\\
&=:\mathfrak M_1(\tau)+\mathfrak M_2(\tau)+\mathfrak M_3(\tau)+\mathfrak M_4(\tau).
\end{align*}
It\^o-isometry yields
\begin{align*}
\E\big[\big\|\mathfrak M_1(\tau)\big\|_{L^{2}_x}\big]
&\leq\,c\,\bigg(\E\bigg[\int_0^{\tau}\|\big(\mathcal S(\tau -\sigma)-\SSS_{\tau-\sigma}\big)\Phi(\bfu)\|_{L_2(\mathfrak U;L^{2}_x)}^2\ds\bigg]\bigg)^{\frac{1}{2}}\\
&\leq\,c\tau\,\bigg(\E\bigg[\int_0^{\tau}\|\Phi(\bfu)\|_{L_2(\mathfrak U;W^{1,2}_x)}^2\ds\bigg]\bigg)^{\frac{1}{2}}\\
&\leq\,c\tau\,\bigg(\E\bigg[\int_0^{\tau}\big(\|\bfu\|_{W^{1,2}_x}^2+1\big)\ds\bigg]\bigg)^{\frac{1}{2}}
\leq\,c\tau^{3/2}
\end{align*}
as a consequence of \eqref{eq:phi}, \eqref{eq:S5} and \eqref{eq:inerror}.
Similarly, we have
\begin{align*}
\E\Big[\big\|\mathfrak M_2(\tau)\big\|_{L^{2}_x}\Big]
&\leq\,c\,\bigg(\E\bigg[\int_0^{\tau}\|\big(\SSS_{\tau-\sigma}-\SSS_\tau\big)\Phi(\bfu)\|_{L_2(\mathfrak U;L^{2}_x)}^2\ds\bigg]\bigg)^{\frac{1}{2}}\\
&\leq\,c\tau\,\bigg(\E\bigg[\int_0^{\tau}\|\Phi(\bfu)\|_{L_2(\mathfrak U;W^{1,2}_x)}^2\ds\bigg]\bigg)^{\frac{1}{2}}\leq\,c\,\tau
\end{align*}
by \eqref{eq:S3}. Furthermore,
\begin{align*}
\E\Big[\big\|\mathfrak M_3(\tau)&\big\|_{L^{2}_x}
\Big]\leq\,c\,\bigg(\E\bigg[\int_0^{\tau}\|\SSS_\tau\big(\Phi(\bfu)-\Phi(\bfh)\big)\|_{L_2(\mathfrak U;L^{2}_x)}^2\ds\bigg]\bigg)^{\frac{1}{2}}\\
&\leq\,c\bigg(\E\bigg[\int_0^{\tau}\|\Phi(\bfu)-\Phi(\bfh)\|_{L_2(\mathfrak U;L^{2}_x)}^2\ds\bigg]\bigg)^{\frac{1}{2}}\leq\,c\bigg(\E\bigg[\int_0^{\tau}\|\bfu-\bfh\|_{L^2_x}^2\ds\bigg]\bigg)^{\frac{1}{2}}\\
&\leq\,c\,\sqrt{\tau}\bigg(\E\bigg[\sup_{0\leq s\leq T}\|\bfu-\bfh\|_{L^2_x}^2\bigg]\bigg)^{\frac{1}{2}}\leq\,c\tau
\end{align*}
due to \eqref{eq:phi}, \eqref{eq:S1} and \eqref{eq:inerror}.
Finally, by \eqref{eq:T-id}, \eqref{eq:S3}, It\^{o}-isometry, \eqref{eq:1102b} and \eqref{eq:phi}
\begin{align*}
\E\Big[\big\|\mathfrak M_4(\tau)\big\|_{L^{2}_x}\Big]
&\leq\,c\,\bigg(\E\bigg[\int_0^{\tau}\|\big(\SSS_\tau-D\mathcal T_\tau(\bfU_\tau)\big)\Phi(\bfh)\|_{L_2(\mathfrak U;L^{2}_x)}^2\ds\bigg]\bigg)^{\frac{1}{2}}\\
&\leq\,c\tau\bigg(\E\bigg[\int_0^{\tau}(1+\|\bfh\|^2_{W^{2,2}})\big(1+\|\bfU_\tau\|^2_{W^{1,2}_x}\big)\ds\bigg]\bigg)^{\frac{1}{2}}\\
&\leq\,c\tau\bigg(\E\bigg[\int_0^{\tau}(1+\|\bfh\|^4_{W^{2,2}}+\|\bfU_\tau\|_{W^{1,2}_x}^4\big)\ds\bigg]\bigg)^{\frac{1}{2}}\\
&\leq\,c\tau\bigg(\E\bigg[\int_0^{\tau}\big(1+\|\bfh\|^4_{W^{1,2}_x}\big)\ds\bigg]\bigg)^{\frac{1}{2}}\leq\,c\tau^{3/2}.
\end{align*}
Finally, we have due to \eqref{eq:D2T} and \eqref{eq:1102b}
\begin{align*}
\E\bigg[\bigg\|&\sum_{i\geq 1}\int_{0}^{\tau} D^2\mathcal T_\tau\big(\bfU_\tau\big)\big(\Phi(\bfh)e_i,\Phi(\bfh)e_i\big)\ds\bigg\|_{L^2_x}\bigg]\\
&\leq\sum_{i\geq1}\E\bigg[\int_{0}^{\tau}\|\Phi(\bfh)e_i\|_{W^{1,2}_x}^2\|\bfU_\tau\|_{W^{1,2}_x}\ds\bigg]\\
&\leq\E\bigg[\int_{0}^{\tau}\big(\|\Phi(\bfh)\|_{L_2(\mathfrak U;W^{1,2}_x)}^4+\|\bfU_\tau\|_{W^{1,2}_x}^2\big)\ds\bigg]\\
&\leq\,c\,\E\bigg[\int_{0}^{\tau}\big(1+\|\bfh\|_{W^{1,2}_x}^4\big)\ds\bigg]
\leq\,c\tau.
\end{align*}
Hence we conclude
\begin{align}\label{estsubsec:1}
|(\mathrm{I})|\leq\,c\tau.
\end{align}

\subsection{The error in the linear part $(\mathrm{II})$}\label{subsec:2}
We use $\A\SSS_\tau=\SSS_\tau\A$ and $\mathcal A-\mathcal S_\tau\mathcal A=-\tau\mathcal S_\tau\mathcal A^2$ to get
\begin{align*}
(\mathrm{II})&=\sum_{m=2}^{M}\E\bigg[\int_{t_{m-1}}^{t_m}\big((\mathcal A-\mathcal S_{\tau}\A)\bfu_{m-1},D\UU(T-s,\bfu_{\tau})\big)_{L^2_x}\ds\bigg]\\
&\quad+\sum_{m=2}^{M}\E\bigg[\int_{t_{m-1}}^{t_m}\big(\mathcal A(\bfu_{\tau}-\bfu_{m-1}),D\UU(T-s,\bfu_{\tau})\big)_{L^2_x}\ds\bigg]\\
&=-\tau\sum_{m=2}^{M}\E\bigg[\int_{t_{m-1}}^{t_m}\big(\mathcal S_{\tau}\A^2\bfu_{m-1},D\UU(T-s,\bfu_{\tau})\big)_{L^2_x}\ds\bigg]\\
&\quad+\sum_{m=2}^{M}\E\bigg[\int_{t_{m-1}}^{t_m}\big(\mathcal A(\bfu_{\tau}-\bfu_{m-1}),D\UU(T-s,\bfu_{\tau})\big)_{L^2_x}\ds\bigg]\\
&=:(\mathrm{II})_1+(\mathrm{II})_2.
\end{align*}
We clearly have
\begin{align*}
(\mathrm{II})_1&\leq\,c\tau\sum_{m=2}^{M}\E\bigg[\int_{t_{m-1}}^{t_m}\|\mathcal S_\tau\|_{\mathcal L(L^2_x)}\|\A^2\bfu_{m-1}\|_{L^2_x}\| D\UU(T-s,\bfu_{\tau})\|_{L^2_x}\ds\bigg]\\
&\leq\,c\tau\sum_{m=2}^{M}\E\bigg[\int_{t_{m-1}}^{t_m}\,\|\A^2\bfu_{m-1}\|_{L^2_x}\ds\bigg]\leq\,c\tau
\end{align*}
using \eqref{eq:S2}, Lemma \ref{lem:kol1} and Lemma \ref{lemma:3.1}. Using Malliavin claculus, cf. Section \ref{sec:mal},
the term $(\mathrm{II})_2$ can be decomposed into the sum of
\begin{align*}
(\mathrm{II})_2^1&=\sum_{m=2}^{M}\E\bigg[\int_{t_{m-1}}^{t_m}\frac{(t-t_{m-1})}{\tau}\big(\A(D\mathcal T_\tau(\bfU_\tau)-\mathrm{id})\bfu_{m-1},D\UU(T-s,\bfu_{\tau})\big)_{L^2_x}\ds\bigg],\\
(\mathrm{II})_2^2
&=\frac{1}{2}\sum_{m=2}^M\sum_{i\geq 1}\E\bigg[\int_{t_{m-1}}^{t_m}\int_{t_{m-1}}^{s}\bigg( D \UU(T-s,\bfu_\tau),D^2\mathcal T_\tau\big(\bfU_\tau\big)\big(\Phi(\bfu_{m-1})e_i,\Phi(\bfu_{m-1})e_i\big)\bigg)_{L^2_x}\,\dd\sigma\ds\bigg],\\
(\mathrm{II})_2^3
&=\tau\sum_{m=2}^M\sum_{i\geq1}\E\bigg[\int_{t_{m-1}}^{t_m}\int_0^{t_{m-1}} D^2\UU(T-s,\bfu_{\tau})\\&\qquad\qquad\qquad\qquad\qquad\qquad\Big(D\mathcal T_\tau(\bfU_\tau)\Phi(\bfu_{m-1})e_i,\A D\mathcal T_{\tau}(\bfU_\tau)\Phi(\bfu_{m-1})e_i\Big)\,\dd\sigma\ds\bigg],
\end{align*}
taking into account \eqref{tdiscr'}, \eqref{eq:malpart} and $\mathcal D_s\bfu_\tau^\lambda=D\mathcal T_{\tau}(\bfU_\tau)\Phi(\bfu_{m-1})$ for $t\in (t_{m-1},t_m)$ and $s\in[t_{m-1},t]$.
We obtain from \eqref{eq:T-idW22}, Lemma \ref{lem:kol1}, \eqref{eq:1102b} and  Lemma \ref{lemma:3.1}
\begin{align*}
|(\mathrm{II})_2^1|
&\leq\,c \sum_{m=2}^M\E\bigg[\int_{t_{m-1}}^{t_m}\|(D\mathcal T_\tau(\bfU_\tau)-\mathrm{id})\bfu_{m-1}\|_{W^{2,2}_x}\ds\bigg]\\
&\leq\,c\tau \sum_{m=2}^M\E\bigg[\int_{t_{m-1}}^{t_m}\|\bfu_{m-1}\|_{W^{4,2}_x}(1+\|\bfU_\tau\|_{W^{2,2}_x}^{8})\ds\bigg]\\
&\leq\,c\tau \sum_{m=2}^M\E\bigg[\int_{t_{m-1}}^{t_m}\Big(1+\|\bfu_{m-1}\|_{W^{4,2}_x}^2+\|\bfU_\tau\|_{W^{2,2}_x}^{16}\Big)\ds\bigg]\leq\,c\tau .
\end{align*}

Moreover, it holds by \eqref{eq:D2T} and Lemma \ref{lemma:3.1}
\begin{align*}
|(\mathrm{II})_2^2|
&\leq\frac{1}{2}\sum_{m=2}^M\sum_{i\geq 1}\E\bigg[\int_{t_{m-1}}^{t_m}\int_{t_{m-1}}^{s}\big\|D^2\mathcal T_\tau\big(\bfU_\tau\big)\big(\Phi(\bfu_{m-1})e_i,\Phi(\bfu_{m-1})e_i\big)\big\|_{L^2_x}\,\dd\sigma\ds\bigg]\\
&\leq\sum_{m=2}^M\E\bigg[\int_{t_{m-1}}^{t_m}\int_{t_{m-1}}^s\|\Phi(\bfu_{m-1})\|_{L_2(\mathfrak U;W^{1,2}_x)}^2\big(1+\big\|\bfU_\tau\big\|_{L^{6}_x}\big)\,\dd\sigma\ds\bigg]\\
&\leq\sum_{m=2}^M\E\bigg[\int_{t_{m-1}}^{t_m}\int_{t_{m-1}}^s\bigg(1+\|\bfu_{m-1}\|_{W^{1,2}_x}^3+\big\|\bfU_\tau\big\|_{W^{1,2}_x}^{3}\bigg)\dd\sigma\ds\bigg]\\
&\leq\,c\tau\sum_{m=2}^M\E\bigg[\int_{t_{m-1}}^{t_m}\Big(1+\|\bfu_{m-1}\|_{W^{1,2}_x}^3\Big)\dd\sigma\ds\bigg]
\leq\,c\tau
\end{align*}
with the help of \eqref{eq:phi} and \eqref{eq:1102d}.
Finally,
we obtain by \eqref{eq:2704}, \eqref{eq:2704B}, \eqref{eq:2704C} and Lemma \ref{lem:kol2}
\begin{align*}
|(\mathrm{II})_2^3|
&\leq\tau\sum_{m=2}^M\sum_{i\geq1}\E\bigg[\int_{t_{m-1}}^{t_m}\int_0^{t_{m-1}}\|D\mathcal T_{\tau}(\bfU_\tau)\Phi(\bfu_{m-1})e_i\|_{W^{2,2}_x}\\
&\qquad\qquad\qquad\qquad\times\|D\mathcal T_{\tau}(\bfU_\tau)\Phi(\bfu_{m-1})e_i\|_{W^{1,2}_x}\dd\sigma\ds\bigg]\\
&\leq\,c\tau\sum_{m=2}^M\E\bigg[\int_{t_{m-1}}^{t_m}\int_0^{t_{m-1}}\|\Phi(\bfu_{m-1})\|^2_{L_2(\mathfrak U;W^{2,2}_x)}\big(1+\big\|\bfU_\tau\big\|_{W^{2,2}_x}^{8}\big)\,\dd\sigma\ds\bigg]
\end{align*}
such that, by \eqref{eq:phi2} and \eqref{eq:1102b},
\begin{align*}
&\leq\tau\sum_{m=2}^M\E\bigg[\int_{t_{m-1}}^{t_m}\int_0^{t_{m-1}}\big(1+\|\bfu_{m-1}\|^{28}_{W^{2,2}_x}\big)\,\dd\sigma\ds\bigg]\leq\,c\tau
\end{align*}
using Lemma \ref{lemma:3.1} in the last step. Plugging all together we have shown
\begin{align}\label{estsubsec:2}
|(\mathrm{II})|\leq\,c\tau.
\end{align}

\subsection{The error in the non-linear part $(\mathrm{III})$}\label{subsec:3}
The non-linear part can be written as
\begin{align*}
(\mathrm{III})&=-\sum_{m=2}^M\E\bigg[\int_{t_{m-1}}^{t_m}\big(\big(\mathrm{id}-\mathcal S_{\tau}\big)f(\bfu_{m-1}),D\UU(T-s,\bfu_{\tau}))_{L^2_x}\ds\bigg]\\
&-\sum_{m=2}^M\E\bigg[\int_{t_{m-1}}^{t_m}\big( f(\bfu_{\tau})-f(\bfu_{m-1}),D\UU(T-s,\bfu_{\tau})\big)_{L^2_x}\ds\bigg]\\
&=:(\mathrm{III})_1+(\mathrm{III})_2,
\end{align*}
where
\begin{align*}
|(\mathrm{III})_1|&\leq \sum_{m=2}^M\E\bigg[\int_{t_{m-1}}^{t_m}\|\A^{-1}\big(\mathrm{id}-\mathcal S_{\tau}\big)\|_{\mathcal L(L^2_x)}\|f(\bfu_{m-1})\|_{W^{2,2}_x}\|D\UU(T-s,\bfu_{\tau})\|_{L^2_x}\ds\bigg]\\
&\leq \,c\tau\sum_{m=2}^M\E\bigg[\int_{t_{m-1}}^{t_m}(\|\bfu_{m-1}\|^q_{W^{2,2}_x}+1)\ds\bigg]
\end{align*}
due to \eqref{eq:S3} and Lemma \ref{lem:kol1}.
The expectation is bounded by $c\tau$ on account of Lemma \ref{lemma:3.1}. Estimating $(\mathrm{III})_2$ requires more effort.
Introducing the notation $f^j(v)=( f(v),v_j)_{L^2_x}$ 
 with the orthonormal basis $(v_j)\subset L^2(\mt)$ allows us to write
\begin{align*}
(\mathrm{III})_2&=-\sum_{m=2}^M\sum_{j\geq1}\E\bigg[\int_{t_{m-1}}^{t_m}\Big( f^j(\bfu_{\tau}(s))-f^j(\bfu_{m-1}),\partial_j\UU(T-s,\bfu_{\tau})\Big)_{L^2_x}\ds\bigg].
\end{align*}
In order to proceed we apply It\^{o}'s formula to $f^j(\bfu_{\tau})$ and recall the definition of $\bfu_\tau$ in \eqref{tdiscr'}. Using the orthonormal basis $(e_i)_{i\geq1}$ of the Hilbert space $\mathfrak{U}$ introduced in Section \ref{sec:prob} we obtain
\begin{align*}
f^j(\bfu_{\tau}(t))&-f^j(\bfu_{m-1})\\&=\frac{1}{2}\sum_{i\geq 1}\int_{t_{m-1}}^tD^2f^j\big(\bfu_{\tau}\big)\Big(D\mathcal T_{\tau}(\bfU_\tau)\Phi(\bfu_{m-1}) e_i,D\mathcal T_{\tau}(\bfU_\tau)\Phi(\bfu_{m-1})e_i\Big)\ds\\
&\quad+\int_{t_{m-1}}^t\Big(\frac{D\mathcal T_{\tau}(\bfU_\tau)-\mathrm{Id}}{\tau}\bfu_{m-1},Df^j(\bfu_{\tau})\Big)_{L^2_x}\ds\\
&\quad+\frac{1}{2}\sum_{i\geq 1}\int_{t_{m-1}}^{t}\bigg( D f^j(\bfu_\tau),D^2\mathcal T_\tau\big(\bfU_\tau\big)\big(\Phi(\bfu_{m-1})e_i,\Phi(\bfu_{m-1})e_i\big)\bigg)_{L^2_x}\ds\\
&\quad+\int_{t_{m-1}}^t\big( Df^j(\bfu_{\tau}),D\mathcal T_{\tau}(\bfU_\tau)\Phi(\bfu_{m-1})\big)_{L^2_x}\,\dd W
\end{align*}
such that
\begin{align*}
&(\mathrm{III})_2\\&=\sum_{m=2}^M\sum_{i\geq1}\E\bigg[\int_{t_{m-1}}^{t_m}\int_{t_{m-1}}^{s}\Big( f''(\bfu_{\tau})|D\mathcal T_{\tau}(\bfU_\tau)\Phi(\bfu_{m-1}) e_i|^2,D \UU\Big)_{L^2_x}\,\dd\sigma\ds\bigg]\\
&+\sum_{m=2}^M\sum_{j\geq1}\E\bigg[\int_{t_{m-1}}^{t_m}\int_{t_{m-1}}^s\Big(\frac{D\mathcal T_{\tau}(\bfU_\tau)-\mathrm{Id}}{\tau}\bfu_{m-1},Df^j(\bfu_{\tau})\Big)_{L^2_x}\partial_j\UU\dd\sigma\ds\bigg]\\
&+\frac{1}{2}\sum_{m=2}^M\sum_{i,j\geq 1}\E\bigg[\int_{t_{m-1}}^{t_m}\int_{t_{m-1}}^{s}\bigg( D f^j(\bfu_\tau),D^2\mathcal T_\tau\big(\bfU_\tau\big)\big(\Phi(\bfu_{m-1})e_i,\Phi(\bfu_{m-1})e_i\big)\bigg)_{L^2_x}\partial_j\UU\,\dd\sigma\ds\bigg]\\
&+\sum_{m=2}^M\sum_{j\geq1}\E\bigg[\int_{t_{m-1}}^{t_m}\int_{t_{m-1}}^s\big( Df^j(\bfu_{\tau}),D\mathcal T_{\tau}(\bfU_\tau)\Phi(\bfu_{m-1})\big)_{L^2_x}\partial_j\UU\,\dd W\ds\bigg]\\
&=:(\mathrm{III})_2^1+(\mathrm{III})_2^2+(\mathrm{III})_2^3+(\mathrm{III})_2^4,
\end{align*}
where $\partial_ju$ and $Du$ are evaluated at $(T-s,\bfu_{\tau})$.
We have by Lemma \ref{lem:kol1}, \eqref{eq:D2T} and \eqref{eq:2704}
\begin{align*}
|(\mathrm{III})^1_2|&\leq \sum_{m=2}^M\sum_{i\geq1}\E\bigg[\int_{t_{m-1}}^{t_m}\int_{t_{m-1}}^{s}\Big\|f''(\bfu_{\tau})|D\mathcal T_{\tau}(\bfU_\tau)\Phi(\bfu_{m-1}) e_i|^2\Big\|_{L^2_x}\| D\UU(T-s,\bfu_{\tau})\|_{L^2_x}\,\dd\sigma\ds\bigg]\\
&\leq\,c\sum_{m=2}^M\sum_{i\geq1}\E\bigg[\int_{t_{m-1}}^{t_m}\int_{t_{m-1}}^{s}\|\bfu_\tau\|_{L^{6}_x}\|\mathcal D\mathcal T_{\tau}(\bfU_\tau)\Phi(\bfu_{m-1})e_i\|_{L^{6}_x}^2\dd\sigma\ds\bigg]\\
&\leq\,c\sum_{m=2}^M\sum_{i\geq1}\E\bigg[\int_{t_{m-1}}^{t_m}\int_{t_{m-1}}^{s}(1+\|\bfu_\tau\|_{W^{1,2}_x})\|\Phi(\bfu_{m-1})e_i\|_{L^{6}_x}^2\,\dd\sigma\ds\bigg]\\
&\leq\,c\sum_{m=2}^M\E\bigg[\int_{t_{m-1}}^{t_m}\int_{t_{m-1}}^{s}(1+\|\bfu_\tau\|_{W^{1,2}_x})\|\Phi(\bfu_{m-1})\|^2_{L_2(\mathfrak U;W^{1,2}_x)}\,\dd\sigma\ds\bigg]
\end{align*}
using also \eqref{est:TL2}--\eqref{est:TW22} in the last step. Finally, by \eqref{eq:phi},
\begin{align*}
(\mathrm{III})^1_2
&\leq \,c\sum_{m=2}^M\E\bigg[\int_{t_{m-1}}^{t_m}\int_{t_{m-1}}^{s}(1+\|\bfu_\tau\|_{W^{1,2}_x}^{3}+\|\bfu_{m-1}\|^3_{W^{1,2}_x})\,\dd\sigma\ds\bigg]\leq\,c\tau
\end{align*}
on account of Lemma \ref{lemma:3.1} and \eqref{eq:1102c}.
In order to estimate $(\mathrm{III})_2^2$ we rewrite
\begin{align*}
(\mathrm{III})_2^2&=\sum_{m=2}^M\sum_{j\geq1}\E\bigg[\int_{t_{m-1}}^{t_m}\int_{t_{m-1}}^s\Big( Df(\bfu_{\tau})\Big(\frac{D\mathcal T_{\tau}(\bfU_\tau)-\mathrm{Id}}{\tau}\bfu_{m-1}\Big),\bfu_j\Big)_{L^2_x}\partial_j\UU\dd\sigma\ds\bigg]
\end{align*}
We obtain further with the help \eqref{eq:T-id}
\begin{align*}
|(\mathrm{III})_2^2|&\leq \,\sum_{m=2}^M\E\bigg[\int_{t_{m-1}}^{t_m}\int_{t_{m-1}}^s\Big\|Df(\bfu_{\tau})\Big(\frac{D\mathcal T_{\tau}(\bfU_\tau)-\mathrm{Id}}{\tau}\bfu_{m-1}\Big)\Big\|_{L^2_x}\|D\UU\|_{L^2_x}\,\dd\sigma\ds\bigg]\\
&\leq\,c\sum_{m=2}^M\E\bigg[\int_{t_{m-1}}^{t_m}\int_{t_{m-1}}^{s}\Big(1+\|\bfu_{\tau}\|_{W^{2,2}_x}^2\Big)\Big\|\frac{D\mathcal T_{\tau}(\bfU_\tau)-\mathrm{Id}}{\tau}\bfu_{m-1}\Big\|_{L^{2}_x}\,\dd\sigma\ds\bigg]\\
&\leq\,c\sum_{m=2}^M\E\bigg[\int_{t_{m-1}}^{t_m}\int_{t_{m-1}}^{s}\Big(1+\|\bfu_{\tau}\|^{2}_{W^{2,2}_x}\Big)\big(1+\|\bfU_\tau\|_{W^{1,2}_x}^2\big)\|\bfu_{m-1}\|_{W^{2,2}_x}\,\dd\sigma\ds\bigg]\\
&\leq\,c\sum_{m=2}^M\E\bigg[\int_{t_{m-1}}^{t_m}\int_{t_{m-1}}^{s}\big(1+\|\bfu_{m-1}\|^9_{W^{2,2}_x}+\|\bfU_\tau\|^{9}_{W^{1,2}_x}+\|\bfu_{\tau}\|_{W^{2,2}_x}^{3}\big)\,\dd\sigma\ds\bigg]\\
&\leq\,c\sum_{m=2}^M\E\bigg[\int_{t_{m-1}}^{t_m}\int_{t_{m-1}}^{s}\big(1+\|\bfu_{m-1}\|^{9}_{W^{2,2}_x}\big)\,\dd\sigma\ds\bigg]\leq\,c\tau
\end{align*}
using also Lemma \ref{lemma:3.1}, as well as \eqref{eq:1102b} and \eqref{eq:1102d}.
With the aid of the integration by parts rule for Malliavin derivatives, \emph{cf.} equation \eqref{eq:malpart}, and using again
$\mathcal D_\sigma\bfu_\tau=D\mathcal T_{\tau}(\bfU_\tau)\Phi(\bfu_{m-1})$ for $s\in (t_{m-1},t_m)$ and $\sigma\in[t_{m-1},t]$ we rewrite
\begin{align*}
&(\mathrm{III})_2^4=\sum_{m=2}^M\E\bigg[\int_{t_{m-1}}^{t_m}\Big( D\UU,\int_{t_{m-1}}^sDf(\bfu_{\tau})D\mathcal T(\bfU_\tau)\Phi(\bfu_{m-1})\,\dd W\Big)_{L^2_x}\ds\bigg]\\
&=\sum_{m=2}^M\sum_{i\geq1}\E\bigg[\int_{t_{m-1}}^{t_m}\int_{t_{m-1}}^sD^2\UU\Big(D\mathcal T_{\tau}(\bfU_\tau)\Phi(\bfu_{m-1})e_i,Df(\bfu_{\tau})D\mathcal T_{\tau}(\bfU_\tau)\Phi(\bfu_{m-1})e_i\Big)\,\dd\sigma\ds\bigg]\\
&\leq\sum_{m=2}^M\sum_{i\geq1}\E\bigg[\int_{t_{m-1}}^{t_m}\int_{t_{m-1}}^s\|D\mathcal T_{\tau}(\bfU_\tau)\Phi(\bfu_{m-1})e_i\|_{L^6_x}^2(\|\bfu_{\tau}\|^2_{L^\infty_x+1)}\,\dd\sigma\ds\bigg]\\
&\leq\sum_{m=2}^M\sum_{i\geq1}\E\bigg[\int_{t_{m-1}}^{t_m}\int_{t_{m-1}}^s\|\Phi(\bfu_{m-1})e_i\|_{W^{1,2}_x}^2(\|\bfu_{\tau}\|_{L^\infty_x}^2+1)\,\dd\sigma\ds\bigg]
\end{align*}
using also \eqref{eq:2704}.
We estimate further with the help of \eqref{eq:phi}, Lemma \ref{lem:kol2} and \eqref{eq:1102d}
\begin{align*}
|(\mathrm{III})_2^3|+|(\mathrm{III})_2^4|&\leq\,c\sum_{m=2}^M\E\bigg[\int_{t_{m-1}}^{t_m}\int_{t_{m-1}}^s\|\Phi(\bfu_{m-1})\|_{L_2(\mathfrak U;W^{1,2}_x)}^2(\|\bfu_{\tau}\|^2_{W^{2,2}_x}+1)\,\dd\sigma\ds\bigg]\\
&\leq\,c\sum_{m=2}^M\E\bigg[\int_{t_{m-1}}^{t_m}\int_{t_{m-1}}^s(1+\|\bfu_{m-1}\|^2_{W^{1,2}_x})(\|\bfu_{\tau}\|_{W^{2,2}_x}^2+1)\,\dd\sigma\ds\bigg]\\
&\leq\,c\sum_{m=2}^M\E\bigg[\int_{t_{m-1}}^{t_m}\int_{t_{m-1}}^t\big(1+\|\bfu_{m-1}\|_{W^{1,2}_x}^4+\|\bfu_{\tau}\|_{W^{2,2}_x}^{4}\big)\ds\dt\bigg]\\
&\leq\,c\sum_{m=2}^M\E\bigg[\int_{t_{m-1}}^{t_m}\int_{t_{m-1}}^s\big(1+\|\bfu_{m-1}\|_{W^{2,2}_x}^4\big)\,\dd\sigma\ds\bigg]\leq\,c\tau
\end{align*}
 using Lemma \ref{lemma:3.1}. By collecting the previous estimates
 we conclude that
 \begin{align}\label{estsubsec:3}|(\mathrm{III})|\leq\,c\tau.
\end{align}

\subsection{The corrector error $(\mathrm{IV})$}
Now wer are concerned with $(\mathrm{IV})$ which is the error which arises from the linearisation and thus does not occur in related paper dealoing with linear semigrouops. We call it the corrector error.
We have by \eqref{eq:D2T} and \eqref{eq:1102b}
\begin{align*}
&\frac{1}{2}\sum_{m=2}^M\sum_{i\geq 1}\E\bigg[\int_{t_{m-1}}^{t_m}\bigg( D\UU(t,\bfu_{\tau}),D^2\mathcal T_\tau\big(\bfU_\tau\big)\big(\Phi(\bfu_{m-1})e_i,\Phi(\bfu_{m-1})e_i\big)\bigg)_{L^2_x}\ds\bigg]\\
&\leq\,c\tau\sum_{m=2}^M\sum_{i\geq 1}\E\bigg[\int_{t_{m-1}}^{t_m}\big\| \bfU_\tau\big\|_{W^{1,2}_x}\big\|\Phi(\bfu_{m-1})e_i\big\|^2_{W^{1,2}_x}\ds\bigg]\\
&=\,c\tau\sum_{m=2}^M\E\bigg[\int_{t_{m-1}}^{t_m}\big\| \bfU_\tau\big\|_{W^{1,2}_x}\big\|\Phi(\bfu_{m-1})\big\|^2_{L_2(\mathfrak U;W^{1,2}_x)}\ds\bigg]\\
&\leq\,c\tau\sum_{m=2}^M\E\bigg[\int_{t_{m-1}}^{t_m}\big(1+\big\| \bfU_\tau\big\|^2_{W^{1,2}_x}+\big\|\bfu_{m-1}\big\|^4_{W^{1,2}_x}\big)\ds\bigg]\\
&\leq\,c\tau\sum_{m=2}^M\E\bigg[\int_{t_{m-1}}^{t_m}\big(1+\big\|\bfu_{m-1}\big\|^4_{W^{1,2}_x}\big)\ds\bigg]\leq \,c\tau
\end{align*}
using Lemma \ref{lemma:3.1} in the last step. This proves
 \begin{align}\label{estsubsec:4a}|(\mathrm{IV})|\leq\,c\tau.
\end{align}

\subsection{The interpolation error $(\mathrm{V})$}\label{subsecint}
In order to estimate $(\mathrm{V})$ we differentiate \eqref{eq:0107} and obtain for $f(z)=z^3-z$
\begin{align*}
\mathcal S_{\tau}&f(\bfu_{m-1})+\frac{1}{\tau}\big(D\mathcal T_{\tau}(\bfU_\tau)-\mathcal S_\tau\big)\bfu_{m-1}\\&=\mathcal S_{\tau}\big(f(\bfu_{m-1})- f'(\mathcal T_\tau(\bfU_\tau))D\mathcal T_\tau(\bfU_\tau)\bfu_{m-1}\big)
\end{align*}
where
\begin{align*}
f(\bfu_{m-1})- f'(\mathcal T_\tau(\bfU_\tau))D\mathcal T_\tau(\bfU_\tau)\bfu_{m-1}&=f(\bfu_{m-1})- f'\Big(\frac{t-t_{m-1}}{\tau}\bfu_{m-1}\Big)\bfu_{m-1}\\
&\quad+\Big(f'\Big(\frac{t-t_{m-1}}{\tau}\bfu_{m-1}\Big)- f'\big(\bfU_\tau)\Big)\bfu_{m-1}\\
&\quad+\big(f'(\bfU_\tau)- f'(\mathcal T_\tau(\bfU_\tau))\big)\bfu_{m-1}\\
&\quad+f'(\mathcal T_\tau(\bfU_\tau))\big(\id-D\mathcal T_\tau(\bfU_\tau)\big)\bfu_{m-1}.
\end{align*}
One easily checks that the first term disappears when integrating over $[t_{m-1},t_m]$ using that $f(0)=0$. As far as the second term is concerned we write
\begin{align}\label{eq:qinN}
\begin{aligned}
\Big(f'\Big(\frac{t-t_{m-1}}{\tau}\bfu_{m-1}\Big)- f'\big(\bfU_\tau)\Big)\bfu_{m-1}&=6\frac{t-t_{m}}{\tau}|\bfu_{m-1}|^2\Phi(\bfu_{m-1})(W_t-W_{t_{m-1}})\\
&\quad+\bfu_{m-1}|\Phi(\bfu_{m-1})(W_t-W_{t_{m-1}})|^2.
\end{aligned}
\end{align}
The first term vanishes under the expectation and hence can be ignored, while the expectation of the ($L^2_x$-norm of the) second one can be controlled by $\tau\E\big[\|\bfu_{m-1}\|_{L^{2}_x}\|\Phi(\bfu_{m-1})\|^2_{L_2(\mathfrak U;W^{2,2}_x)}\big]$.
as a consequence of It\^{o}-isometry. By \eqref{est:TL2} and \eqref{est:Tidhigher}
\begin{align*}
\big\|\big(f'(\bfU_\tau)- f'(\mathcal T_\tau(\bfU_\tau))\big)\bfu_{m-1}\|_{L^2_x}&\leq\,c\big(\|\bfU_\tau\|_{L^6_x}+\|\mathcal T_\tau(\bfU_\tau)\|_{L^6_x}\big)\|\bfU_\tau-\mathcal T_\tau(\bfU_\tau)\|_{L^6_x}\|\bfu_{m-1}\|_{L^6_x}\\
&\leq\,c\tau\|\bfU_\tau\|_{W^{1,2}_x}\big(1+\|\bfU_\tau\|_{W^{3,2}_x}+\|\bfU_\tau\|_{W^{2,2}_x}^3\big)\|\bfu_{m-1}\|_{W^{1,2}_x}\\
&\leq\,c\tau\big(1+\|\bfU_\tau\|_{W^{3,2}_x}^5+\|\bfu_{m-1}\|_{W^{1,2}_x}^5\big).
\end{align*}
Finally,
\begin{align*}
\|f'(\mathcal T_\tau(\bfU_\tau))\big(\id-D\mathcal T_\tau(\bfU_\tau)\big)\bfu_{m-1}\big\|_{L^2}&\leq\,c\|f'(\mathcal T_\tau(\bfU_\tau))\|_{L^{\infty}_x}\|\id-D\mathcal T_\tau(\bfU_\tau)\|_{\mathcal L(L^{2}_x)}\|\bfu_{m-1}\big\|_{L^{2}_x}\\
&\leq\,c\tau\|\mathcal T_\tau(\bfU_\tau)\|_{W^{2,2}_x}^2\big(1+\|\bfU_\tau\|^2_{W^{1,2}_x}\big)\|\bfu_{m-1}\big\|_{L^{2}_x}\\
&\leq\,c\tau\big(1+\|\bfU_\tau\|_{W^{2,2}_x}^6\big)\big(1+\|\bfU_\tau\|^2_{W^{1,2}_x}\big)\|\bfu_{m-1}\big\|_{L^{2}_x}\\
&\leq\,c\tau\big(1+\|\bfU_\tau\|_{W^{2,2}_x}^{10}+\|\bfu_{m-1}\big\|^{10}_{L^{2}_x}\big)
\end{align*}
using \eqref{eq:T-id} and \eqref{est:TW22}. On account of
Lemma \ref{lem:kol1}, \eqref{eq:1102b} and Lemma \ref{lemma:3.1} we conclude
\begin{align}\label{error:v}
|(\mathrm V)|
&\leq\,c\tau^2\sum_{m=2}^{M}\E\Big(1+\|\bfu_{m-1}\|_{W^{3,2}_x}^5+\|\bfu_{m-1}\|_{W^{2,2}_x}^{10}\Big)\leq\,c\tau. \end{align}

\subsection{The error in the correction term $(\mathrm{VI})$}\label{subsec:4}
First of all we decompose
\begin{align*}
(\mathrm{VI})
&=\frac{1}{2}\sum_{m=2}^{M}\sum_{i\geq1}\E\bigg[\int_{t_{m-1}}^{t_m}D^2\UU(T-s,\bfu_{\tau})\Big(\Phi(\bfu_{\tau})e_i,\big(\Phi(\bfu_{\tau})-\Phi(\bfu_{m-1})\big)e_i\Big)\ds\bigg]\\
&+\frac{1}{2}\sum_{m=2}^{M}\sum_{i\geq1}\E\bigg[\int_{t_{m-1}}^{t_m}D^2\UU(T-s,\bfu_{\tau})\Big(\big(\Phi(\bfu_{\tau})-\Phi(\bfu_{m-1})\big)e_i,\Phi(\bfu_{m-1})e_i\Big)\ds\bigg]\\
&+\frac{1}{2}\sum_{m=2}^{M}\sum_{i\geq1}\E\bigg[\int_{t_{m-1}}^{t_m}D^2\UU(T-s,\bfu_{\tau})\Big((\id-D\mathcal T_\tau\big(\bfU_\tau\big))\Phi(\bfu_{m-1})e_i,D\mathcal T_\tau\big(\bfU_\tau\big)\Phi(\bfu_{m-1})e_i\Big)\ds\bigg]\\
&+\frac{1}{2}\sum_{m=2}^{M}\sum_{i\geq1}\E\bigg[\int_{t_{m-1}}^{t_m}D^2\UU(T-s,\bfu_{\tau})\Big(D\mathcal T_\tau\big(\bfU_\tau\big)\Phi(\bfu_{m-1})e_i,(\id-D\mathcal T_\tau\big(\bfU_\tau\big))\Phi(\bfu_{m-1})e_i\Big)\ds\bigg]\\
&=:(\mathrm{VI})_1+\dots+(\mathrm{VI})_4.
\end{align*}
By Lemma \ref{lem:kol2}, \eqref{eq:phi}, \eqref{eq:2704}
and \eqref{eq:T-id} we obtain
\begin{align*}
|(\mathrm{VI})_4|&\leq \,c\sum_{m=2}^{M}\sum_{i\geq1}\E\bigg[\int_{t_{m-1}}^{t_m}\|D\mathcal T_\tau\big(\bfU_\tau\big)\Phi(\bfu_{m-1})e_i\|_{L^6_x}\|(\id-D\mathcal T_\tau\big(\bfU_\tau\big))\Phi(\bfu_{m-1})e_i\|_{L^2_x}\ds\bigg]\\
&\leq \,c\tau \sum_{m=2}^{M}\sum_{i\geq1}\E\bigg[\int_{t_{m-1}}^{t_m}\|\Phi(\bfu_{m-1})e_i\|_{L^6_x}\|\Phi(\bfu_{m-1})e_i\|_{W^{2,2}_x}\big(1+\|\bfU_\tau\|_{W^{1,2}_x}^2\big)\ds\bigg]\\
&\leq \,c\tau \sum_{m=2}^{M}\E\bigg[\int_{t_{m-1}}^{t_m}\|\Phi(\bfu_{m-1})\|_{L_2(\mathfrak U;W^{2,2}_x)}^2\big(1+\|\bfU_\tau\|_{W^{1,2}_x}\big)\ds\bigg]\\
&\leq \,c\sum_{m=2}^{M}\E\bigg[\int_{t_{m-1}}^{t_m}(1+\|\bfu_{m-1}\|_{W^{2,2}_x}^3+\|\bfU_\tau\|^3_{W^{1,2}_x}\big)\ds\bigg].
\end{align*}
 Hence we obtain
\begin{align*}
|(\mathrm{VI})_4|&\leq \,c\sum_{m=2}^{M}\E\bigg[\int_{t_{m-1}}^{t_m}(1+\|\bfu_{m-1}\|^3_{W^{2,2}_x})\ds\bigg]\leq\,c\tau
\end{align*}
using \eqref{eq:1102b} and Lemma \ref{lemma:3.1}.
The same idea can be used to estimate $(\mathrm{VI})_3$.
In order to estimate $(\mathrm{VI})_1$ and $(\mathrm{VI})_2$ we apply It\^{o}'s formula to the functional
\begin{align*}
t\mapsto \big(\Phi(\bfu_{\tau})-\Phi(\bfu_{m-1})\big)e_i
\end{align*}
taking $\mathbb P$-a.s. values in $L^2(\mt)$.
We obtain from \eqref{tdiscr'} for
$t\in[t_{m-1},t_m]$
\begin{align}\label{eq:1605}
\begin{aligned}
\big(\Phi(\bfu_{\tau}(t))&-\Phi(\bfu_{m-1})\big)e_i\\&=\frac{1}{2}\int_{t_{m-1}}^tD^2\Phi(\bfu_\tau)\big(D\mathcal T_{\tau}(\bfU_\tau)\Phi(\bfu_{m-1})\big)^\ast D\mathcal T_{\tau}(\bfU_\tau)\Phi(\bfu_{m-1})e_i\ds\\
&\quad+\int_{t_{m-1}}^t\frac{D\mathcal T_{\tau}(\bfU_\tau)-\mathrm{id}}{\tau}\bfu_{m-1}D\Phi(\bfu_{\tau})e_i\ds\\
&\quad+\frac{1}{2}\sum_{j\geq 1}\int_{t_{m-1}}^t D\Phi(\bfu_{\tau})e_iD^2\mathcal T_\tau\big(\bfU_\tau\big)\big(\Phi(\bfu_{m-1})e_j,\Phi(\bfu_{m-1})e_j\big)\ds\\
&\quad+\int_{t_{m-1}}^tD\mathcal T_{\tau}(\bfU_\tau)\Phi(\bfu_{m-1})D\Phi(\bfu_{\tau})e_i\,\dd \beta^i.
\end{aligned}
\end{align}
The last term leads us to 
\begin{align*}
&\frac{1}{2}\sum_{m=2}^{M}\sum_{i\geq1}\E\bigg[\int_{t_{m-1}}^{t_m}D^2\UU(T-s,\bfu_{\tau})\bigg(\Phi(\bfu_{\tau})e_i,\int_{t_{m-1}}^sD\mathcal T_{\tau}(\bfU_\tau)\Phi(\bfu_{m-1})D\Phi(\bfu_{\tau})e_i\,\dd \beta^i\bigg)\ds\bigg]\\
&=\frac{1}{2}\sum_{m=2}^{M}\sum_{i\geq1}\E\bigg[\int_{t_{m-1}}^{t_m}\int_{t_{m-1}}^sD^3\UU(T-s,\bfu_{\tau})\\&\qquad\qquad\qquad\qquad\qquad\bigg(\mathcal D^i_\sigma\bfu_\tau,\Phi(\bfu_{\tau})e_i,D\mathcal T_{\tau}(\bfU_\tau)\Phi(\bfu_{m-1})D\Phi(\bfu_{\tau})e_i\bigg)\,\dd\sigma\ds\bigg]\\
&\quad+\frac{1}{2}\sum_{m=2}^{M}\sum_{i\geq1}\E\bigg[\int_{t_{m-1}}^{t_m}\int_{t_{m-1}}^sD^2\UU(T-t,\bfu_{\tau})\\&\qquad\qquad\qquad\qquad\qquad\bigg(D\Phi(\bfu_{\tau})e_i\mathcal D^i_\sigma\bfu_\tau,D\mathcal T_{\tau}(\bfU_\tau)\Phi(\bfu_{m-1})D\Phi(\bfu_{\tau})e_i\bigg)\,\dd\sigma\ds\bigg]\\
&=\frac{1}{2}\sum_{m=2}^{M}\sum_{i\geq1}\E\bigg[\int_{t_{m-1}}^{t_m}\int_{t_{m-1}}^tD^3\UU(T-s,\bfu_{\tau})\\&\qquad\qquad\qquad\qquad\qquad\bigg(D\mathcal T_{\tau}(\bfU_\tau)\Phi(\bfu_{m-1})e_i,\Phi(\bfu_{\tau})e_i,D\mathcal T_{\tau}(\bfU_\tau)\Phi(\bfu_{m-1})D\Phi(\bfu_{\tau})e_i\bigg)\,\dd\sigma\ds\bigg]\\
&\quad+\frac{1}{2}\sum_{m=2}^{M}\sum_{i\geq1}\E\bigg[\int_{t_{m-1}}^{t_m}\int_{t_{m-1}}^sD^2\UU(T-s,\bfu_{\tau})\\&\qquad\qquad\qquad\qquad\qquad\bigg(D\Phi(\bfu_{\tau})e_iD\mathcal T_{\tau}(\bfU_\tau)\Phi(\bfu_{m-1})e_i,D\mathcal T_{\tau}(\bfU_\tau)\Phi(\bfu_{m-1})D\Phi(\bfu_{\tau})e_i\bigg)\,\dd\sigma\ds\bigg]
\end{align*}
where we used \eqref{eq:malpart} and $\mathcal D_s\bfu_\tau=D\mathcal T_{\tau}(\bfU_\tau)\Phi(\bfu_{m-1})$ for $s\in[t_{m-1},t_m]$. Using Lemma \ref{lem:kol3} the term involving $D^3\UU$ can be estimated by
\begin{align*}
\sum_{m=2}^{M}&\sum_{i\geq1}\E\bigg[\int_{t_{m-1}}^{t_m}\int_{t_{m-1}}^s\|D\mathcal T_{\tau}(\bfU_\tau)\Phi(\bfu_{m-1})e_i\|_{L^6_x}\|\Phi(\bfu_{\tau})e_i\|_{L^6_x}\\&\qquad\qquad\qquad\qquad\qquad\times\|D\mathcal T_{\tau}(\bfU_\tau)\Phi(\bfu_{m-1})D\Phi(\bfu_{\tau})e_i\|_{L^6_x}\,\dd\sigma\ds\bigg]\\
&\leq\,c\sum_{m=2}^{M}\sum_{i\geq1}\E\bigg[\int_{t_{m-1}}^{t_m}\int_{t_{m-1}}^s\|\Phi(\bfu_{m-1})e_i\|_{L^6_x}\|\Phi(\bfu_{\tau})e_i\|_{L^6_x}\|\Phi(\bfu_{m-1})D\Phi(\bfu_{\tau})e_i\|_{L^6_x}\,\dd\sigma\ds\bigg]\\
&\leq\,c\sum_{m=2}^{M}\sum_{i\geq1}\E\bigg[\int_{t_{m-1}}^{t_m}\int_{t_{m-1}}^s\|\Phi(\bfu_{m-1})e_i\|_{W^{1,2}_x}\|\Phi(\bfu_{\tau})e_i\|_{W^{1,2}_x}\\&\qquad\qquad\qquad\qquad\qquad\qquad\times\|\Phi(\bfu_{m-1})\|_{\mathcal L(W^{1,2}_x)}\|D\Phi(\bfu_{\tau})e_i\|_{W^{1,2}_x}\,\dd\sigma\ds\bigg]\\
&\leq\,c\sum_{m=2}^{M}\E\int_{t_{m-1}}^{t_m}\int_{t_{m-1}}^s\Big(\|\Phi(\bfu_{m-1})\|_{L_2(\mathfrak U;W^{1,2}_x)}^4+\|\Phi(\bfu_{\tau})\|_{L_2(\mathfrak U;W^{1,2}_x)}^4\Big)\,\dd\sigma\ds\\
&\quad+\,c\sum_{m=2}^{M}\E\bigg[\int_{t_{m-1}}^{t_m}\int_{t_{m-1}}^s\|D\Phi(\bfu_{\tau})\|_{L_2(\mathfrak U;W^{1,2}_x)}^4\,\dd\sigma\ds\bigg]\\
&\leq\,c\sum_{m=2}^{M}\E\bigg[\int_{t_{m-1}}^{t_m}\int_{t_{m-1}}^s\Big(\|\bfu_{m-1}\|_{W^{1,2}_x}^4+\|\bfu_{\tau}\|_{W^{1,2}_x}^4+1\Big)\,\dd\sigma\ds\bigg]
\end{align*}
taking into account \eqref{eq:2704}, \eqref{eq:phi} and Lemma \ref{lemma:3.1}.
For the other term we obtain similarly by Lemma \ref{lem:kol2} the bound
\begin{align*}
\sum_{m=2}^{M}\sum_{i\geq1}&\E\bigg[\int_{t_{m-1}}^{t_m}\int_{t_{m-1}}^s\|D\Phi(\bfu_{\tau})e_iD\mathcal T_{\tau}(\bfU_\tau)\Phi(\bfu_{m-1})e_i\|_{L^2_x}\|D\mathcal T_{\tau}(\bfU_\tau)\Phi(\bfu_{m-1})D\Phi(\bfu_{\tau})e_i\|_{L^6_x}\,\dd\sigma\ds\bigg]\\
&\leq\,c\sum_{m=2}^{M}\sum_{i\geq1}\E\bigg[\int_{t_{m-1}}^{t_m}\int_{t_{m-1}}^s\|\Phi(\bfu_{m-1})e_i\|_{L^2_x}\|\Phi(\bfu_{m-1})D\Phi(\bfu_{\tau})e_i\|_{L^6_x}\,\dd\sigma\ds\bigg]\\
&\leq\,c\sum_{m=2}^{M}\sum_{i\geq1}\E\bigg[\int_{t_{m-1}}^{t_m}\int_{t_{m-1}}^s\|\Phi(\bfu_{m-1})e_i\|_{L^2_x}\|\Phi(\bfu_{m-1})\|_{\mathcal L(W^{1,2}_x)}\|D\Phi(\bfu_{\tau})e_i\|_{W^{1,2}_x}\,\dd\sigma\ds\bigg]\\
&\leq\,c\sum_{m=2}^{M}\E\bigg[\int_{t_{m-1}}^{t_m}\int_{t_{m-1}}^s\Big(\|\Phi(\bfu_{m-1})\|_{L_2(\mathfrak U;W^{1,2}_x)}^3+\|D\Phi(\bfu_{\tau})\|_{L_2(\mathfrak U;W^{1,2}_x)}^3\Big)\,\dd\sigma\ds\bigg]\\
&\leq\,c\sum_{m=2}^{M}\E\bigg[\int_{t_{m-1}}^{t_m}\int_{t_{m-1}}^s\Big(\|\bfu_{m-1}\|_{W^{1,2}_x}^3+1\Big)\,\dd\sigma\ds\bigg]\leq\,c\tau.
\end{align*}
 The $L^2_x$-norm of the remaining terms in \eqref{eq:1605} can be estimated
by the sum of
\begin{align*}
&\tau\|D\mathcal T_{\tau}(\bfU_\tau)\Phi(\bfu_{m-1})e_i\|_{L^2_x}\|D\mathcal T_{\tau}(\bfU_\tau)\Phi(\bfu_{m-1})\|_{\mathcal L(L^2_x)},\\
&\|(D\mathcal T_{\tau}(\bfU_\tau)-\mathrm{id})\bfu_{m-1}\|_{L^2_x}\|D\Phi(u_\tau)e_i\|_{\mathcal L(L^2_x)},\\& \tau\sum_{i\geq1}\|D^2\mathcal T_\tau\big(\bfU_\tau\big)\big(\Phi(\bfu_{m-1})e_j,\Phi(\bfu_{m-1})e_j\big)\|_{L^2_x}\|D\Phi(u_\tau)e_i\|_{\mathcal L(L^2_x)},
\end{align*}
using boundedness of $D\Phi$ and $D^2\Phi$, \emph{cf.} \eqref{eq:phi}. By the properties of $\mathcal T_\tau$ from \eqref{eq:2704}, \eqref{eq:T-id} and \eqref{eq:D2T} we bound these terms by
\begin{align*}
&\tau(1+\|\bfu_{m-1}\|_{L^2_x})\|\Phi(\bfu_{m-1})e_i\|_{L^2_x}
,\\
&\tau\|\bfu_{m-1}\|_{W^{2,2}_x}(1+\|\bfU_\tau\|_{W^{1,2}_x})\|D\Phi(u_\tau)e_i\|_{\mathcal L(L^2_x)},\\
&\|\bfU_\tau\|_{W^{1,2}_x}\|\Phi(\bfu_{m-1})\|^2_{L_2(\mathfrak U;W^{1,2}_x)}\|D\Phi(u_\tau)e_i\|_{\mathcal L(L^2_x)}.
\end{align*}
 We conclude
\begin{align*}
|(\mathrm V)_1|
&\leq\,c\tau^2\sum_{m=2}^{M}\E\Big(1+\|\bfu_{m-1}\|_{W^{2,2}_x}^2+\|\bfu_{m-1}\|_{W^{1,2}_x}^3\Big)\leq\,c\tau \end{align*}
using \eqref{eq:1102b} and Lemma \ref{lemma:3.1}.
The estimate for $(\mathrm{IV})_2$ is analogous and we conclude
\begin{align}\label{estsubsec:4}
|(\mathrm{IV})|\leq\,c\tau.
\end{align}

\textbf{Conclusion.} Combining the estimates \eqref{estsubsec:1}, \eqref{estsubsec:2}, \eqref{estsubsec:3}, \eqref{estsubsec:4a}, \eqref{error:v} and \eqref{estsubsec:4} we have shown
\begin{align*}
\E[\varphi(\bfu(T,\bfh))]-\E[\varphi(\bfu_M)]\leq\,c\tau.
\end{align*}
The proof of Theorem \ref{thm:main} is hereby complete.

\section*{Compliance with Ethical Standards}\label{conflicts}
\smallskip
\par\noindent
{\bf Conflict of Interest}. The author declares that he has no conflict of interest.

\smallskip
\par\noindent
{\bf Data Availability}. Data sharing is not applicable to this article as no datasets were generated or analysed during the current study.

\end{document}